\documentclass{article}

\usepackage[english]{babel}
\usepackage[latin1]{inputenc}
\usepackage{mathdots,graphicx,amsfonts,amssymb,amsmath,amsthm,mathrsfs,epstopdf,enumitem,mathtools,color,accents,abraces,caption,booktabs,url} 
\usepackage[top=2cm,bottom=3cm,left=2cm,right=2cm]{geometry}
\usepackage{subcaption}
\usepackage{sectsty}\sectionfont{\normalsize}\subsectionfont{\normalsize}

\newtheorem{theorem}{Theorem}[section]
\newtheorem{lemma}{Lemma}[section]
\newtheorem{corollary}{Corollary}[section]
\theoremstyle{definition}
\newtheorem{definition}{Definition}[section]
\newtheorem{remark}{Remark}[section]
\newtheorem{example}{Example}[section]
\newtheorem{conjecture}{Conjecture}[section]

\numberwithin{equation}{section}
\numberwithin{figure}{section}
\numberwithin{table}{section}

\renewcommand{\epsilon}{\varepsilon}
\renewcommand{\i}{{\rm i}}
\newcommand{\e}{{\rm e}}
\renewcommand{\aa}{{\boldsymbol a}}
\newcommand{\bb}{{\boldsymbol b}}
\newcommand{\ii}{{\boldsymbol i}}
\newcommand{\jj}{{\boldsymbol j}}
\newcommand{\kk}{{\boldsymbol k}}
\newcommand{\nn}{{\boldsymbol n}}
\newcommand{\hh}{{\boldsymbol h}}
\newcommand{\xx}{{\boldsymbol x}}
\newcommand{\yy}{{\boldsymbol y}}
\newcommand{\bu}{{\mathbf 1}}
\newcommand{\II}{{\mathcal I}}
\newcommand{\rd}{\color{red}}\newcommand{\bl}{\color{blue}}\newcommand{\bk}{\color{black}}

\begin{document}

\title{\Large From asymptotic distribution and vague convergence to uniform convergence, with numerical applications}

\author{Giovanni Barbarino\\
\footnotesize Department of Mathematics and Operations Research, University of Mons, Belgium (giovanni.barbarino@umons.ac.be)\\[10pt]
Sven-Erik Ekstr\"om\\
\footnotesize Division of Scientific Computing, Department of Information Technology, Uppsala University, Sweden (sven-erik.ekstrom@it.uu.se)\\[10pt]
Carlo Garoni\\
\footnotesize Department of Mathematics, University of Rome Tor Vergata, Italy (garoni@mat.uniroma2.it)\\[10pt]
David Meadon\\
\footnotesize Division of Scientific Computing, Department of Information Technology, Uppsala University, Sweden (david.meadon@it.uu.se)\\[10pt]
Stefano Serra-Capizzano\thanks{Corresponding author (correspondence to: s.serracapizzano@uninsubria.it)}\\
\footnotesize Department of Science and High Technology, University of Insubria, Italy (s.serracapizzano@uninsubria.it)\\[-2pt]
\footnotesize Division of Scientific Computing, Department of Information Technology, Uppsala University, Sweden (stefano.serra@it.uu.se)\\[10pt]
Paris Vassalos\\
\footnotesize Department of Informatics, Athens University of Economics and Business, Greece (pvassal@aueb.gr)}
\date{}

\maketitle

\begin{abstract}
Let $\{\Lambda_n=\{\lambda_{1,n},\ldots,\lambda_{d_n,n}\}\}_n$ be a sequence of finite multisets of real numbers such that $d_n\to\infty$ as $n\to\infty$, and let $f:\Omega\subset\mathbb R^d\to\mathbb R$ be a Lebesgue measurable function defined on a domain $\Omega$ with $0<\mu_d(\Omega)<\infty$, where $\mu_d$ is the Lebesgue measure in $\mathbb R^d$.
We say that $\{\Lambda_n\}_n$ has an asymptotic distribution described by $f$, and we write $\{\Lambda_n\}_n\sim f$, if
\begin{equation}\label{lr}
\lim_{n\to\infty}\frac1{d_n}\sum_{i=1}^{d_n}F(\lambda_{i,n})=\frac1{\mu_d(\Omega)}\int_\Omega F(f(\xx)){\rm d}\xx\tag{$*$}
\end{equation}
for every continuous function $F$ with bounded support.
If $\Lambda_n$ is the spectrum of a matrix $A_n$, we say that $\{A_n\}_n$ has an asymptotic spectral distribution described by $f$ and we write $\{A_n\}_n\sim_\lambda f$.
In the case where $d=1$, $\Omega$~is a bounded interval, $\Lambda_n\subseteq f(\Omega)$ for all $n$, and $f$ satisfies suitable conditions, Bogoya, B\"ottcher, Grudsky, and Maximenko proved that the asymptotic distribution \eqref{lr} implies the uniform convergence to $0$ of the difference between the properly sorted vector $[\lambda_{1,n},\ldots,\lambda_{d_n,n}]$ and the vector of samples $[f(x_{1,n}),\ldots,f(x_{d_n,n})]$, i.e.,
\begin{equation}\label{uc}
\lim_{n\to\infty}\,\max_{i=1,\ldots,d_n}|f(x_{i,n})-\lambda_{\tau_n(i),n}|=0, \tag{$**$}
\end{equation}
where $x_{1,n},\ldots,x_{d_n,n}$ is a uniform grid in $\Omega$ and $\tau_n$ is the sorting permutation.
We extend this result to the case where $d\ge1$ and $\Omega$ is a Peano--Jordan measurable set (i.e., a bounded set with $\mu_d(\partial\Omega)=0$).
We also formulate and prove a uniform convergence result analogous to \eqref{uc} in the more general case where the function $f$ takes values in the space of $k\times k$ matrices. Our derivations are based on the concept of monotone rearrangement (quantile function) as well as on matrix analysis arguments stemming from the theory of generalized locally Toeplitz sequences and the observation that any finite multiset of numbers $\Lambda_n=\{\lambda_{1,n},\ldots,\lambda_{d_n,n}\}$ can always be interpreted as the spectrum of a matrix $A_n={\rm diag}(\lambda_{1,n},\ldots,\lambda_{d_n,n})$.
The theoretical results are illustrated through numerical experiments, and a reinterpretation of them in terms of vague convergence of probability measures is hinted.

\smallskip


\smallskip

\end{abstract}

\section{Introduction}

Throughout this paper, a matrix-sequence is a sequence of the form $\{A_n\}_n$, where $A_n$ is a square matrix such that ${\rm size}(A_n)=d_n\to\infty$ as $n\to\infty$.
Matrix-sequences arise in several contexts. 
For example, when a linear differential equation is discretized by a linear numerical method, such as the finite difference method, the finite element method, the isogeometric analysis, etc., the actual computation of the numerical solution reduces to solving a linear system $A_n\mathbf u_n=\mathbf f_n$. The size $d_n$ of this system diverges to $\infty$ as the mesh-fineness parameter $n$ tends to $\infty$, and we are therefore in the presence of a matrix-sequence $\{A_n\}_n$.
It is often observed in practice that $\{A_n\}_n$ belongs to the class of generalized locally Toeplitz (GLT) sequences and it therefore enjoys an asymptotic singular value and/or eigenvalue distribution. 
We refer the reader to the books \cite{GLTbookI,GLTbookII} and the papers \cite{barbarinoREDUCED,rGLT,GLTbookIII,GLTbookIV} for a comprehensive treatment of GLT sequences and to \cite{SbMath} for a more concise introduction to the subject.
Another noteworthy example of matrix-sequences concerns the finite sections of an infinite Toeplitz matrix. An infinite (block) Toeplitz matrix is a matrix of the form
\begin{equation}\label{iTm}
[f_{i-j}]_{i,j=1}^\infty = \begin{bmatrix} 
f_0 & f_{-1} & f_{-2} & \cdots & \cdots & \cdots \\[5pt]
f_1 & f_0 & f_{-1} & f_{-2} & \cdots & \cdots\\[5pt]
f_2 & f_1 & f_0 & f_{-1} & f_{-2} & \cdots \\
\vdots & f_2 & f_1 & \ddots & \ddots & \ddots \\
\vdots & \vdots & f_2 & \ddots & \ddots & \ddots \\
\vdots & \vdots & \vdots & \ddots & \ddots & \ddots 
\end{bmatrix},
\end{equation}
where the entries $\ldots,f_{-2},f_{-1},f_0,f_1,f_2,\ldots$ are $k\times k$ matrices (blocks) for some $k\ge1$. If $k=1$, then \eqref{iTm} is a classical (scalar) Toeplitz matrix. 
The $n$th section of \eqref{iTm} is the matrix defined by
\begin{equation}\label{A_nT}
A_n=[f_{i-j}]_{i,j=1}^n.
\end{equation}
A case of special interest arises when the entries $f_k$ are the Fourier coefficients of a function $f:[-\pi,\pi]\to\mathbb C^{k\times k}$ with components $f_{ij}\in L^1([-\pi,\pi])$, i.e., 
\[ f_k=\frac1{2\pi}\int_{-\pi}^\pi f(\theta)\e^{-\i k\theta}{\rm d}\theta,\qquad k\in\mathbb Z, \]
where the integrals are computed componentwise. In this case, the matrix $A_n$ in \eqref{A_nT} is denoted by $T_n(f)$ and is referred to as the $n$th (block) Toeplitz matrix generated by $f$.
The asymptotic singular value and eigenvalue distributions of the Toeplitz sequence $\{T_n(f)\}_n$ have been deeply investigated over time, starting from Szeg\H o's first limit theorem \cite{GS} and the Avram--Parter theorem \cite{Avram,Parter}, up to the works by Tyrtyshnikov--Zamarashkin \cite{Ty96,ZT,ZT'} and Tilli \cite{TilliL1,Tilli-complex}. For more on this subject, see \cite[Chapter~5]{BoSi} and \cite[Chapter~6]{GLTbookI}.

An important concept related to matrix-sequences is the notion of asymptotic spectral (or eigenvalue) distribution. 
After the publication of Tyrtyshnikov's paper in 1996 \cite{Ty96}, there has been an ever growing interest in this topic, which led, among others, to the birth of GLT sequences \cite{glt1,glt2,Tilli98}.
The reasons behind this widespread interest are not purely academic, because the asymptotic spectral distribution has significant practical implications.
For example, suppose that $\{A_n\}_n$ is a matrix-sequence resulting from the discretization of a differential equation $\mathscr Au=f$ through a given numerical method.
Then, the asymptotic spectral distribution of $\{A_n\}_n$ can be used to measure the accuracy of the method in approximating the spectrum of the differential operator $\mathscr A$ \cite{DavideCalcolo2021}, to establish whether the method preserves the so-called average spectral gap 
\cite{DavideCalcolo2018}, or to formulate analytical predictions for the eigenvalues of both $A_n$ and $\mathscr A$ \cite{Tom-paper}.
Moreover, the asymptotic spectral distribution of $\{A_n\}_n$ can be exploited to design efficient iterative solvers for linear systems with matrix $A_n$ and to analyze/predict their performance; see \cite{BeKu,Kuij-SIREV} for accurate convergence estimates of Krylov methods based on the asymptotic spectral distribution 
and \cite[p.~3]{GLTbookI} for more details on this subject.

Before proceeding further, let us introduce the formal definition of asymptotic singular value and eigenvalue distribution.
Let $C_c(\mathbb R)$ (resp., $C_c(\mathbb C)$) be the space of continuous complex-valued functions with bounded support defined on $\mathbb R$ (resp., $\mathbb C$).
If $A\in\mathbb C^{m\times m}$, the singular values and eigenvalues of $A$ are denoted by $\sigma_1(A),\ldots,\sigma_m(A)$ and $\lambda_1(A),\ldots,\lambda_m(A)$, respectively. If $\lambda_1(A),\ldots,\lambda_m(A)$ are real, their maximum and minimum are also denoted by $\lambda_{\max}(A)$ and $\lambda_{\min}(A)$.
We denote by $\mu_d$ the Lebesgue measure in $\mathbb R^d$. Throughout this paper, unless otherwise specified, ``measurable'' means ``Lebesgue measurable'' and ``a.e.''\ means ``almost everywhere (with respect to the Lebesgue measure)''. 
A matrix-valued function $f:\Omega\subseteq\mathbb R^d\to\mathbb C^{k\times k}$ is said to be measurable (resp., bounded, continuous, continuous a.e., in $L^p(\Omega)$, etc.)\ if its components $f_{ij}:\Omega\to\mathbb C$, $i,j=1,\ldots,k$, are measurable (resp., bounded, continuous, continuous a.e., in $L^p(\Omega)$, etc.).

\begin{definition}[\textbf{asymptotic singular value and eigenvalue distribution of a matrix-sequence}]\label{dd}
Let $\{A_n\}_n$ be a matrix-sequence with $A_n$ of size $d_n$, and let $f:\Omega\subset\mathbb R^d\to\mathbb C^{k\times k}$ be measurable with $0<\mu_d(\Omega)<\infty$.
\begin{itemize}[nolistsep,leftmargin=*]
	\item We say that $\{A_n\}_n$ has an asymptotic eigenvalue (or spectral) distribution described by $f$ if
	\begin{equation}\label{sd}
	\lim_{n\to\infty}\frac1{d_n}\sum_{i=1}^{d_n}F(\lambda_i(A_n))=\frac1{\mu_d(\Omega)}\int_\Omega\frac{\sum_{i=1}^kF(\lambda_i(f(\xx)))}{k}\mathrm{d}\xx,\qquad\forall\,F\in C_c(\mathbb C).
	\end{equation}
	In this case, $f$ is called the eigenvalue (or spectral) symbol of $\{A_n\}_n$ and we write $\{A_n\}_n\sim_\lambda f$.
	\item We say that $\{A_n\}_n$ has an asymptotic singular value distribution described by $f$ if
	\begin{equation}\label{svd}
	\lim_{n\to\infty}\frac1{d_n}\sum_{i=1}^{d_n}F(\sigma_i(A_n))=\frac1{\mu_d(\Omega)}\int_\Omega\frac{\sum_{i=1}^kF(\sigma_i(f(\xx)))}{k}\mathrm{d}\xx,\qquad\forall\,F\in C_c(\mathbb R).
	\end{equation}
	In this case, $f$ is called the singular value symbol of $\{A_n\}_n$ and we write $\{A_n\}_n\sim_\sigma f$.
\end{itemize}
\end{definition}

We remark that Definition~\ref{dd} is well-posed as the functions $\xx\mapsto\sum_{i=1}^kF(\lambda_i(f(\xx)))$ and $\xx\mapsto\sum_{i=1}^kF(\sigma_i(f(\xx)))$ appearing in \eqref{sd}--\eqref{svd} are measurable \cite[Lemma~2.1]{GLTbookIII}. Throughout this paper, whenever we write a relation such as $\{A_n\}_n\sim_\lambda f$ or $\{A_n\}_n\sim_\sigma f$, it is understood that $\{A_n\}_n$ and $f$ are as in Definition~\ref{dd}, i.e., $\{A_n\}_n$ is a matrix-sequence and $f$ is a measurable function taking values in $\mathbb C^{k\times k}$ for some $k$ and defined on a subset $\Omega$ of some $\mathbb R^d$ with $0<\mu_d(\Omega)<\infty$. 
Since any finite multiset of numbers can always be interpreted as the spectrum of a matrix, a byproduct of Definition~\ref{dd} is the following definition. 

\begin{definition}[\textbf{asymptotic distribution of a sequence of finite multisets of numbers}]\label{adn}
Let $\{\Lambda_n=\{\lambda_{1,n},\ldots,\lambda_{d_n,n}\}\}_n$ be a sequence of finite multisets of numbers such that $d_n\to\infty$ as $n\to\infty$, and let $f$ be as in Definition~\ref{dd}. 
We say that $\{\Lambda_n\}_n$ has an asymptotic distribution described by $f$, and we write $\{\Lambda_n\}_n\sim f$, if $\{A_n\}_n\sim_\lambda f$, where $A_n$ is any matrix whose spectrum equals $\Lambda_n$ (e.g., $A_n={\rm diag}(\lambda_{1,n},\ldots,\lambda_{d_n,n})$). 
\end{definition}

In the previous  
literature, it has often been claimed that the informal meaning behind the asymptotic spectral distribution \eqref{sd} is the following \cite[Remark~2.9]{GLTbookIII}: assuming that 
there exist $k$ a.e.\ continuous functions $\lambda_1(f),\ldots,\lambda_k(f):\Omega\to\mathbb C$ such that $\lambda_1(f(\xx)),\ldots,\lambda_k(f(\xx))$ are the eigenvalues of $f(\xx)$ for every $\xx\in\Omega$, the eigenvalues of $A_n$, except possibly for $o(d_n)$ outliers, can be subdivided into $k$ different subsets of approximately the same cardinality, and, for $n$ large enough, the eigenvalues belonging to the $i$th subset are approximately equal to the samples of $\lambda_i(f)$ over a uniform grid in the domain $\Omega$.
For instance, if $d=1$, $d_n=nk$, and $\Omega=[a,b]$, then, assuming we have no outliers, the eigenvalues of $A_n$ are approximately equal to
\[ \lambda_i\Bigl(f\Bigl(a+j\frac{b-a}n\Bigr)\Bigr),\qquad j=1,\ldots,n,\qquad i=1,\ldots,k, \]
for $n$ large enough; similarly, if $d=2$, $d_n=n^2k$, and $\Omega=[a_1,b_1]\times [a_2,b_2]$, then, assuming we have no outliers, the eigenvalues of $A_n$ are approximately equal to
\[ \lambda_i\Bigl(f\Bigl(a_1+j_1\frac{b_1-a_1}{n},a_2+j_2\frac{b_2-a_2}{n}\Bigr)\Bigr),\qquad j_1,j_2=1,\ldots,n,\qquad i=1,\ldots,k, \]
for $n$ large enough; and so on for $d\ge3$.
In the case where $d=k=1$, $\Omega=[a,b]$ is a bounded interval, $\{\lambda_1(A_n),\ldots,\lambda_{d_n}(A_n)\}\subseteq f(\Omega)$ for all $n$, and $f$ is a real function satisfying suitable conditions, a precise mathematical formulation and proof of the previous informal meaning was given by Bogoya, B\"ottcher, Grudsky, and Maximenko first in the Toeplitz case $A_n=T_n(f)$ \cite[Theorem~1.5]{maximum_norm} and then in the case of an arbitrary $A_n$ \cite[Theorem~1.3]{Bottcher-ded-Grudsky}. In a nutshell, they proved the uniform convergence to $0$ of the difference $\boldsymbol{\lambda}(A_n)-f(\xx_n)$, i.e.,
\begin{equation}\label{uci}
\lim_{n\to\infty}\|f(\xx_n)-\boldsymbol\lambda(A_n)\|_\infty=0,
\end{equation}
where $\|\cdot\|_\infty$ is the usual $\infty$-norm of vectors, $\boldsymbol\lambda(A_n)$ is the properly sorted vector of eigenvalues of $A_n$, $f(\xx_n)$ is the vector of samples $[f(x_{1,n}),\ldots,f(x_{d_n,n})]$, and $\xx_n=[x_{1,n},\ldots,x_{d_n,n}]$ is a uniform grid in $\Omega$.

In this paper, using the concept of monotone rearrangement (quantile function) and matrix analysis arguments from the theory of GLT sequences, we provide deeper insights into the notion of asymptotic spectral distribution by presenting precise mathematical formulations and proofs of the informal meaning behind \eqref{sd} that are more general than \cite[Theorem~1.3]{Bottcher-ded-Grudsky}. Our formulations are made in terms of uniform convergence to $0$ of differences of vectors, in complete analogy with \eqref{uci}.
\begin{itemize}[leftmargin=*,nolistsep]
	\item In our first main result (Theorem~\ref{th:main4-n}), we extend \cite[Theorem~1.3]{Bottcher-ded-Grudsky} by formulating 
	and proving the informal meaning behind \eqref{sd} in the case where $d\ge1$ and the domain $\Omega$ of the spectral symbol $f$ is a Peano--Jordan measurable set (i.e., a bounded set with $\mu_d(\partial\Omega)=0$).
	In Corollary~\ref{eccoGi-n}, we prove (a slightly more general version of) \cite[Theorem~1.3]{Bottcher-ded-Grudsky} as a corollary of Theorem~\ref{th:main4-n}. 
	\item In our second main result (Theorem~\ref{thm:mmd}), we prove for $d=1$ that Theorem~\ref{th:main4-n} can be strengthened in the case where the spectrum of $A_n$ is contained in the image of $f$ for every $n$ and $f$ satisfies some mild assumptions. More precisely, in this case we prove that the eigenvalues of $A_n$ are exact samples of $f$ over an asymptotically uniform grid (see Section~\ref{main} for the corresponding definition).
	\item In our last main result (Theorem~\ref{eccoGi'}), we extend \cite[Theorem~1.3]{Bottcher-ded-Grudsky} by formulating  
	and proving the informal meaning behind \eqref{sd} in the case where $k\ge1$, i.e., the spectral symbol $f$ is a matrix-valued function.
\end{itemize}
The results herein, including the main Theorems~\ref{th:main4-n}--\ref{eccoGi'}, are actually formulated in terms of asymptotic distributions of sequences of finite multisets of numbers (Definition~\ref{adn}) and not in terms of asymptotic spectral distributions (Definition~\ref{dd}). This is done to allow for a better comparison with the previous literature and especially with \cite{Bottcher-ded-Grudsky}. Reinterpreting the results in terms of asymptotic spectral distributions is a straightforward rephrasing exercise that is left to the reader.

The paper is organized as follows. 
In Section~\ref{main}, we formulate the main results. In Section~\ref{proofs}, we prove the main results. In Section~\ref{numexp}, we illustrate the main results through numerical experiments.  
In Section~\ref{conc}, we draw conclusions and we also highlight the relation existing between the asymptotic  
distribution and the vague convergence of probability measures (a relation that allows for a reinterpretation of the main results of this paper in a probabilistic perspective).

\section{Main results}\label{main}

\subsection{Notation and terminology}
Throughout this paper, the cardinality, the interior, the closure, and the characteristic (indicator) function of a set $E$ are denoted by $|E|$, $\accentset{\circ}E$, $\overline E$, and $\chi_E$, respectively.
We use ``increasing'' 
as a synonym of ``non-decreasing''. 
We use ``strictly increasing''  
whenever we want to specify that the increase 
is strict. Similarly, we use ``decreasing'' as a synonym of ``non-increasing'' and ``strictly decreasing'' whenever we want to specify that the decrease is strict. 
The word ``monotone'' means either ``increasing'' or ``decreasing'', while ``strictly monotone'' means either ``strictly increasing'' or ``strictly decreasing''.
If $z\in\mathbb C$ and $\epsilon>0$, we denote by $D(z,\epsilon)$ the open disk with center $z$ and radius $\epsilon$, i.e., $D(z,\epsilon)=\{w\in\mathbb C:|w-z|<\epsilon\}$.
If $S\subseteq\mathbb C$ and $\epsilon>0$, we denote by $S_\epsilon=\bigcup_{z\in S}D(z,\epsilon)$ the $\epsilon$-expansion of $S$, i.e., the set of points whose distance from $S$ is smaller than $\epsilon$.
We use a notation borrowed from probability theory to indicate sets. For example, if $f,g:\Omega\subseteq\mathbb R^d\to\mathbb R$, then $\{f\le1\}=\{\xx\in\Omega:f(\xx)\le1\}$, $\mu_d\{f>0,\:g<0\}$ is the measure of the set $\{\xx\in\Omega:f(\xx)>0,\:g(\xx)<0\}$, etc.

\medskip

\noindent\textbf{Multi-index notation.}
A multi-index $\ii$ of size $d$, also called a $d$-index, is a vector in $\mathbb Z^d$. 
$\mathbf0$ and $\bu$ are the vectors of all zeros and all ones, respectively (their size will be clear from the context). For any vector $\nn\in\mathbb R^d$, we set $N(\nn)=\prod_{j=1}^dn_j$ and we write $\nn\to\infty$ to indicate that $\min(\nn)\to\infty$. 
If $\hh,\kk\in\mathbb R^d$, an inequality such as $\hh\le\kk$ means that $h_j\le k_j$ for all $j=1,\ldots,d$.
If $\hh,\kk$ are $d$-indices such that $\hh\le\kk$, the $d$-index range $\{\hh,\ldots,\kk\}$ is the set $\{\ii\in\mathbb Z^d:\hh\le\ii\le\kk\}$. We assume for this set the standard lexicographic ordering:
\[ \Bigl[\ \ldots\ \bigl[\ [\ (i_1,\ldots,i_d)\ ]_{i_d=h_d,\ldots,k_d}\ \bigr]_{i_{d-1}=h_{d-1},\ldots,k_{d-1}}\ \ldots\ \Bigr]_{i_1=h_1,\ldots,k_1}. \]
For instance, in the case $d=2$ the ordering is
\begin{align*}
&(h_1,h_2),\,(h_1,h_2+1),\,\ldots,\,(h_1,k_2),\,(h_1+1,h_2),\,(h_1+1,h_2+1),\,\ldots,\,(h_1+1,k_2),\\
&\ldots\,\ldots\,\ldots,\,(k_1,h_2),\,(k_1,h_2+1),\,\ldots,\,(k_1,k_2).
\end{align*}
When a $d$-index $\ii$ varies in a finite set $\mathcal I\subset\mathbb Z^d$ (this is simply written as $\ii\in\mathcal I$), it is always understood that $\ii$ follows the lexicographic ordering.
For instance, if we write $\xx=[x_\ii]_{\ii\in\II}$, then $\xx$ is a vector of size $|\II|$ whose components are indexed by the $d$-index $\ii$ varying in $\II$ according to the lexicographic ordering. Similarly, if we write $X=[x_{\ii\jj}]_{\ii,\jj\in\II}$, then $X$ is a square matrix of size $|\II|$ whose components are indexed by a pair of $d$-indices $\ii,\jj$, both varying in $\II$ according to the lexicographic ordering.
When $\mathcal I$ is a $d$-index range $\{\hh,\ldots,\kk\}$, the notation $\ii\in\mathcal I$ is often replaced by $\ii=\hh,\ldots,\kk$. 
Operations involving $d$-indices (or general vectors with $d$ components) that have no meaning in the vector space $\mathbb R^d$ must always be interpreted in the componentwise sense. For instance, $\boldsymbol j\hh=(j_1h_1,\ldots,j_dh_d)$, $\ii/\boldsymbol j=(i_1/j_1,\ldots,i_d/j_d)$, etc.
If $\aa,\bb\in\mathbb R^d$ and $\aa\le\bb$, we denote by $[\aa,\bb]$ the closed $d$-dimensional rectangle given by $[a_1,b_1]\times\cdots\times[a_d,b_d]$.

\medskip

\noindent\textbf{Essential range.}
Given a measurable function $f:\Omega\subseteq\mathbb R^d\to\mathbb C$, the essential range of $f$ is denoted by $\mathcal{ER}(f)$. We recall that $\mathcal{ER}(f)$ is defined as 
\[ \mathcal{ER}(f)=\{z\in\mathbb C:\hspace{0.5pt}\mu_d\{f\in D(z,\epsilon)\}>0\hspace{0.5pt}\mbox{ for all }\hspace{0.5pt}\epsilon>0\}. \]
It is clear that $\mathcal{ER}(f)\subseteq\overline{f(\Omega)}$.
Moreover, $\mathcal{ER}(f)$ is closed and $f\in\mathcal{ER}(f)$ a.e.; see, e.g., \cite[Lemma~2.1]{GLTbookI}.
If $f$ is real then $\mathcal{ER}(f)$ is a subset of $\mathbb R$. In this case, we define the essential infimum (resp., supremum) of $f$ on $\Omega$ as the infimum (resp., supremum) of $\mathcal{ER}(f)$:
\[ \mathop{\rm ess\,inf}_\Omega f=\inf\mathcal{ER}(f),\qquad\mathop{\rm ess\,sup}_\Omega f=\sup\mathcal{ER}(f). \]

\medskip

\noindent\textbf{Asymptotically uniform grids.}
Let $[\aa,\bb]$ be a closed $d$-dimensional rectangle and let $\{\nn=\nn(n)\}_n$ be a sequence of $d$-indices in $\mathbb N^d$ such that $\nn\to\infty$ as $n\to\infty$.
For every $n$, let $\mathcal G_\nn^{(n)}=\{\xx_{\ii,\nn}^{(n)}\}_{\ii=\bu,\ldots,\nn}$ be a sequence of $N(\nn)$ points in $\mathbb R^d$. We say that $\mathcal G_\nn^{(n)}$ is an asymptotically uniform (a.u.)\ grid in $[\aa,\bb]$ if
\[ \lim_{n\to\infty}m(\mathcal G_\nn^{(n)})=0, \]
where
\[ m(\mathcal G_\nn^{(n)})=\max_{\ii=\bu,\ldots,\nn}\left\|\xx_{\ii,\nn}^{(n)}-\Bigl(\aa+\ii\,\frac{\bb-\aa}{\nn}\Bigr)\right\|_\infty \]
is referred to as the distance of $\mathcal G_\nn^{(n)}$ from the uniform grid $\{\aa+\ii(\bb-\aa)/\nn\}_{\ii=\bu,\ldots,\nn}$.  
Note that $\mathcal G_\nn^{(n)}$ needs not to be contained in $[\aa,\bb]$ in order to be an a.u.\ grid in $[\aa,\bb]$. Note also that the notation $\nn$ is used instead of $\nn(n)$ for simplicity, but it is understood that $\nn=\nn(n)$ depends on $n$.

\medskip

\noindent\textbf{Regular sets.}
We say that $\Omega\subset\mathbb R^d$ is a regular set if it is bounded and $\mu_d(\partial\Omega)=0$.
Note that the condition ``$\mu_d(\partial\Omega)=0$'' is equivalent to ``$\chi_\Omega$ is continuous a.e.\ on $\mathbb R^d$''. Any regular set $\Omega\subset\mathbb R^d$ is measurable and we have $\mu_d(\Omega)=\mu_d(\accentset{\circ}\Omega)=\mu_d(\overline\Omega)<\infty$. 
In Riemann integration theory, a regular set is simply a Peano--Jordan measurable set.

\subsection{Statements of the main results}
Theorem~\ref{th:main4-n} is our first main result.
It is a generalization to the multidimensional case of a previous result due to Bogoya, B\"ottcher, Grudsky, and Maximenko \cite[Theorem~1.3]{Bottcher-ded-Grudsky}.

\begin{theorem}\label{th:main4-n}
Let $f:\Omega\subset\mathbb R^d\to\mathbb R$ be bounded and continuous a.e.\ on the regular set $\Omega$ with $\mu_d(\Omega)>0$ and $\mathcal{ER}(f)=[\inf_\Omega f,\sup_\Omega f]$. 
Take any $d$-dimensional rectangle $[\aa,\bb]$ containing $\Omega$ and any a.u.\ grid $\mathcal G_\nn^{(n)}=\{\xx_{\ii,\nn}^{(n)}\}_{\ii=\bu,\ldots,\nn}$ in $[\aa,\bb]$, where $\nn=\nn(n)\in\mathbb N^d$ and $\nn\to\infty$ as $n\to\infty$.
For every $n$, define $\II_\nn^{(n)}(\Omega)=\{\ii\in\{\bu,\ldots,\nn\}:\xx_{\ii,\nn}^{(n)}\in\Omega\}$ and consider the multiset of samples $\{f(\xx_{\ii,\nn}^{(n)}):\ii\in\II_\nn^{(n)}(\Omega)\}$ and a multiset of $|\II_\nn^{(n)}(\Omega)|$  
real numbers $\{\lambda_{\ii,\nn}^{(n)}:\ii\in\II_\nn^{(n)}(\Omega)\}$ with the following properties:
\begin{itemize}[nolistsep,leftmargin=*]
	\item $\{\{\lambda_{\ii,\nn}^{(n)}:\ii\in\II_\nn^{(n)}(\Omega)\}\}_n\sim f$; 
	\item $\{\lambda_{\ii,\nn}^{(n)}:\ii\in\II_\nn^{(n)}(\Omega)\}\subseteq[\inf_\Omega f-\epsilon_n,\sup_\Omega f+\epsilon_n]$ for every $n$ and for some $\epsilon_n\to0$ as $n\to\infty$.
\end{itemize}
Then, if  
$\sigma_n$ and $\tau_n$ are two permutations of $\II_\nn^{(n)}(\Omega)$ such that the vectors $[f(\xx_{\sigma_n(\ii),\nn}^{(n)})]_{\ii\in\II_\nn^{(n)}(\Omega)}$ and $[\lambda_{\tau_n(\ii),\nn}^{(n)}]_{\ii\in\II_\nn^{(n)}(\Omega)}$ are sorted in increasing order, we have
\[ \max_{\ii\in\II_\nn^{(n)}(\Omega)}|f(\xx_{\sigma_n(\ii),\nn}^{(n)})-\lambda_{\tau_n(\ii),\nn}^{(n)}|\to0\ \,\mbox{as}\,\ n\to\infty. \]
In particular,
\[ \min_{\tau}\max_{\ii\in\II_\nn^{(n)}(\Omega)}|f(\xx_{\ii,\nn}^{(n)})-\lambda_{\tau(\ii),\nn}^{(n)}|\to0\ \,\mbox{as}\ \,n\to\infty, \]
where the minimum is taken over all permutations $\tau$ of $\II_\nn^{(n)}(\Omega)$.
\end{theorem}

\begin{remark}\label{ER(f)=[inf,sup]}
Since $\mathcal{ER}(f)$ is closed, 
the hypothesis $\mathcal{ER}(f)=[\inf_\Omega f,\sup_\Omega f]$ in Theorem~\ref{th:main4-n} is equivalent to asking that $\mathcal{ER}(f)$ is connected with $\inf_\Omega f=\mathop{\rm ess\,inf}_\Omega f$ and $\sup_\Omega f=\mathop{\rm ess\,sup}_\Omega f$. Note that the condition $\inf_\Omega f=\mathop{\rm ess\,inf}_\Omega f$ is equivalent to $\inf_\Omega f\ge\mathop{\rm ess\,inf}_\Omega f$ (the other inequality $\inf_\Omega f\le\mathop{\rm ess\,inf}_\Omega f$ is always satisfied because $\mathcal{ER}(f)\subseteq\overline{f(\Omega)}$). Similarly, the condition $\sup_\Omega f=\mathop{\rm ess\,sup}_\Omega f$ is equivalent to $\sup_\Omega f\le\mathop{\rm ess\,sup}_\Omega f$. 
Note also that the hypothesis $\mathcal{ER}(f)=[\inf_\Omega f,\sup_\Omega f]$ implies 
$\overline{f(\Omega)}=\mathcal{ER}(f)$.
\end{remark}

As a consequence of Theorem~\ref{th:main4-n}, in Corollary~\ref{eccoGi-n} we prove (a slightly more general version of) \cite[Theorem~1.3]{Bottcher-ded-Grudsky}.

\begin{corollary}\label{eccoGi-n}
Let $f:[a,b]\to\mathbb R$ be bounded and continuous a.e.\ with $\mathcal{ER}(f)=[\inf_{[a,b]}f,\sup_{[a,b]}f]$. 
Let $\{\Lambda_n=\{\lambda_{1,n},\ldots,\lambda_{d_n,n}\}\}_n$ be a sequence of finite multisets of  
real numbers such that $d_n\to\infty$ as $n\to\infty$. Assume the following. 
\begin{itemize}[nolistsep,leftmargin=*]
	\item $\{\Lambda_n\}_n\sim f$. 
	\item $\Lambda_n\subseteq[\inf_{[a,b]}f-\epsilon_n,\sup_{[a,b]}f+\epsilon_n]$ for every $n$ and for some $\epsilon_n\to0$ as $n\to\infty$.
\end{itemize}
Then, for every a.u.\ grid $\{x_{i,n}\}_{i=1,\ldots,d_n}$ in $[a,b]$ with $\{x_{i,n}\}_{i=1,\ldots,d_n}\subset[a,b]$, if $\sigma_n$ and $\tau_n$ are two permutations of $\{1,\ldots,d_n\}$ such that the vectors $[f(x_{\sigma_n(1),n}),\ldots,f(x_{\sigma_n(d_n),n})]$ and $[\lambda_{\tau_n(1),n},\ldots,\lambda_{\tau_n(d_n),n}]$ are sorted in increasing order, we have
\[ \max_{i=1,\ldots,d_n}|f(x_{\sigma_n(i),n})-\lambda_{\tau_n(i),n}|\to0\ \,\mbox{as}\ \,n\to\infty. \]
In particular,
\[ \min_\tau\max_{i=1,\ldots,d_n}|f(x_{i,n})-\lambda_{\tau(i),n}|\to0\ \,\mbox{as}\ \,n\to\infty, \]
where the minimum is taken over all permutations $\tau$ of $\{1,\ldots,d_n\}$.
\end{corollary}
\begin{proof}
Take an a.u.\ grid $\{x_{i,n}\}_{i=1,\ldots,d_n}$ in $[a,b]$ and two permutations $\sigma_n$ and $\tau_n$ as specified in the statement. 
Since $\{x_{i,n}\}_{i=1,\ldots,d_n}\subset[a,b]$ by assumption, for every $n$ we have $I_n([a,b])=\{i\in\{1,\ldots,d_n\}:x_{i,n}\in[a,b]\}=\{1,\ldots,d_n\}$.
To conclude that $\max_{i=1,\ldots,d_n}|f(x_{\sigma_n(i),n})-\lambda_{\tau_n(i),n}|\to0$ as $n\to\infty$, apply Theorem~\ref{th:main4-n} to the function $f$ with the a.u.\ grid $\{x_{i,n}\}_{i=1,\ldots,d_n}$ 
and the multiset of real numbers $\{\lambda_{i,n}\}_{i=1,\ldots,d_n}=\Lambda_n$.
\end{proof}

\begin{remark}\label{controesempio}
The hypothesis ``$\mathcal{ER}(f)=[\inf_{[a,b]}f,\sup_{[a,b]}f]$'' in Corollary~\ref{eccoGi-n} is replaced in \cite[Theorem~1.3]{Bottcher-ded-Grudsky} by the weaker version ``$\mathcal{ER}(f)$ is connected''. However, the latter is not enough to get the thesis, 
as shown by the following counterexample. Let $f(x)=\chi_{\{1\}}(x):[0,1]\to\mathbb R$ be the characteristic function of $\{1\}$. Note that $f$ is bounded and continuous a.e.\ with $\mathcal{ER}(f)=\{0\}$ connected. Take $d_n=n$ and $\lambda_{1,n}=\ldots=\lambda_{n,n}=0$ for all $n$. 
All the hypotheses of Corollary~\ref{eccoGi-n} are satisfied except for the assumption ``$\mathcal{ER}(f)=[\inf_{[0,1]}f,\sup_{[0,1]}f]$'', and the thesis does not hold. Indeed, if we take the a.u.\ grid in $[0,1]$ given by $\{x_{i,n}=i/n\}_{i=1,\ldots,n}$, then the samples $f(x_{1,n}),\ldots,f(x_{n,n})$ are sorted in increasing order just as the numbers $\lambda_{1,n},\ldots,\lambda_{n,n}$, and we have $\max_{i=1,\ldots,n}|f(x_{i,n})-\lambda_{i,n}|=1\not\to0$ as $n\to\infty$.
This same counterexample shows that the hypothesis ``$\mathcal{ER}(f)$ is connected'' in \cite[Theorem~1.3]{Bottcher-ded-Grudsky} must be replaced by the stronger version ``$\mathcal{ER}(f)=[\inf_{[a,b]}f,\sup_{[a,b]}f]$'' as in Corollary~\ref{eccoGi-n}, otherwise the result does not hold.\,\footnote{\,With the notations of \cite[Theorem~1.3]{Bottcher-ded-Grudsky}, the symbols $f,\mathcal{ER}(f),d_n,\lambda_{i,n},x_{i,n}$ used in the above counterexample should be changed to $X,\mathcal{R}(X),d(n),\alpha_i^{(n)},\xi_i^{(n)}$, respectively.}
\end{remark}

Theorem~\ref{thm:mmd} is our second main result. It shows that, if in Corollary~\ref{eccoGi-n} we assume that $\Lambda_n\subseteq f([a,b])$ and $f$ has a finite number of local maximum, local minimum, and discontinuity points, then the values in $\Lambda_n$, up to a suitable permutation, are exact samples of $f$ on an a.u.\ grid in $[a,b]$. 
It is important to point out that by ``local maximum/minimum point'' we here mean ``weak local maximum/minimum point'' according to the following definition.

\begin{definition}[\textbf{local extremum points}]\label{wlep}
Given a function $f:[a,b]\to\mathbb R$ and a point $x_0\in[a,b]$, we say that $x_0$ is a local maximum point (resp., local minimum point) for $f$ if $f(x_0)\ge f(x)$ (resp., $f(x_0)\le f(x)$) for all $x$ belonging to a neighborhood of $x_0$ in $[a,b]$.
\end{definition}

For example, if $f$ is constant on $[a,b]$, then all points of $[a,b]$ are both local maximum and local minimum points for $f$. 

\begin{theorem}\label{thm:mmd}
Let $f:[a,b]\to\mathbb R$ be bounded with a finite number of local maximum points, local minimum points, and discontinuity points, and with $\mathcal{ER}(f)=[\inf_{[a,b]}f,\sup_{[a,b]}f]$. Let $\{\Lambda_n=\{\lambda_{1,n},\ldots,\lambda_{d_n,n}\}\}_n$ be a sequence of finite multisets of real numbers such that $d_n\to\infty$ as $n\to\infty$. Assume the following. 
\begin{itemize}[nolistsep,leftmargin=*]
	\item $\{\Lambda_n\}_n\sim f$. 
	\item $\Lambda_n\subseteq f([a,b])$ for every $n$.
\end{itemize}
Then, 
there exist an a.u.\ grid $\{x_{i,n}\}_{i=1,\ldots,d_n}$ in $[a,b]$ with $\{x_{i,n}\}_{i=1,\ldots,d_n}\subset[a,b]$ and a permutation $\tau_n$ of $\{1,\ldots,d_n\}$ such that, for every $n$,
\[ \lambda_{\tau_n(i),n}=f(x_{i,n}),\qquad i=1,\ldots,d_n. \]
\end{theorem}

Theorem~\ref{eccoGi'} is our last main result. It is a generalization of Corollary~\ref{eccoGi-n} to the case where the scalar function $f$ is replaced by a matrix-valued function. 

\begin{theorem}\label{eccoGi'}
Let $f_1,\ldots,f_k:[a,b]\to\mathbb R$ be bounded and continuous a.e.\ with $\mathcal{ER}(f_1)=[\inf_{[a,b]}f_1,\sup_{[a,b]}f_1],\ldots,$\linebreak$\mathcal{ER}(f_k)=[\inf_{[a,b]}f_k,\sup_{[a,b]}f_k]$.
Let $\{\Lambda_n=\{\lambda_{1,n},\ldots,\lambda_{d_n,n}\}\}_n$ be a sequence of finite multisets of real numbers such that $d_n\to\infty$ as $n\to\infty$.
Assume the following.
\begin{itemize}[nolistsep,leftmargin=*]
	\item $\{\Lambda_n\}_n\sim f$ 
	with $f={\rm diag}(f_1,\ldots,f_k)$.
	\item For every $n$ there exists a partition $\{\tilde\Lambda_{n,1},\ldots,\tilde\Lambda_{n,k}\}$ of $\Lambda_n$ such that, for every $j=1,\ldots,k$, $|\tilde\Lambda_{n,j}|/d_n\to1/k$ as $n\to\infty$ and $\tilde\Lambda_{n,j}\subseteq[\inf_{[a,b]}f_j-\epsilon_n,\sup_{[a,b]}f_j+\epsilon_n]$ for some $\epsilon_n\to0$ as $n\to\infty$.
\end{itemize}
Then, for every $n$ there exists a partition $\{\Lambda_{n,1},\ldots,\Lambda_{n,k}\}$ of $\Lambda_n$ 
such that, for every $j=1,\ldots,k$, the following properties hold.
\begin{itemize}[nolistsep,leftmargin=*]
	\item $|\Lambda_{n,j}|=|\tilde\Lambda_{n,j}|$.
	\item $\Lambda_{n,j}\subseteq[\inf_{[a,b]}f_j-\delta_n,\sup_{[a,b]}f_j+\delta_n]$ for some $\delta_n\to0$ as $n\to\infty$.
	\item $\{\Lambda_{n,j}\}_n\sim f_j$.
	\item Let $\Lambda_{n,j}=\{\lambda_{1,n}^{(j)},\ldots,\lambda_{|\Lambda_{n,j}|,n}^{(j)}\}$. For every a.u.\ grid $\{x_{i,n}^{(j)}\}_{i=1,\ldots,|\Lambda_{n,j}|}$ in $[a,b]$ with $\{x_{i,n}^{(j)}\}_{i=1,\ldots,|\Lambda_{n,j}|}\subset[a,b]$, if $\sigma_{n,j}$ and $\tau_{n,j}$ are two permutations of $\{1,\ldots,|\Lambda_{n,j}|\}$ such that the vectors $[f_j(x_{\sigma_{n,j}(1),n}^{(j)}),\ldots,f_j(x_{\sigma_{n,j}(|\Lambda_{n,j}|),n}^{(j)})]$ and $[\lambda_{\tau_{n,j}(1),n}^{(j)},\ldots,\lambda_{\tau_{n,j}(|\Lambda_{n,j}|),n}^{(j)}]$ are sorted in increasing order, we have
	\[ \max_{i=1,\ldots,|\Lambda_{n,j}|}|f_j(x_{\sigma_{n,j}(i),n}^{(j)})-\lambda_{\tau_{n,j}(i),n}^{(j)}|\to0\ \,\mbox{as}\ \,n\to\infty. \]
	In particular,
	\[ \min_\tau\max_{i=1,\ldots,|\Lambda_{n,j}|}|f_j(x_{i,n}^{(j)})-\lambda_{\tau(i),n}^{(j)}|\to0\ \,\mbox{as}\ \,n\to\infty, \]
	where 
	the minimum is taken over all permutations $\tau$ of $\{1,\ldots,|\Lambda_{n,j}|\}$.
\end{itemize}
\end{theorem}

In order to prove Theorem~\ref{eccoGi'}, we shall need the following lemmas, which are reported here because they have a special interest in themselves and may be considered as further main results of this paper, although ``less important'' than Theorems~\ref{th:main4-n}--\ref{eccoGi'}.

\begin{lemma}\label{ss1}
Let $f_1,\ldots,f_k:[a,b]\to\mathbb R$ be measurable. Let $\{\Lambda_n=\{\lambda_{1,n},\ldots,\lambda_{d_n,n}\}\}_n$ be a sequence of finite multisets of real numbers such that $d_n\to\infty$ as $n\to\infty$. Assume the following. 
\begin{itemize}[nolistsep,leftmargin=*]
	\item $\{\Lambda_n\}_n\sim f$ with $f={\rm diag}(f_1,\ldots,f_k)$. 
	\item $\{L_{n,1}\}_n,\ldots,\{L_{n,k}\}_n$ are sequences of natural numbers such that $L_{n,1}+\ldots+L_{n,k}=d_n$ for every $n$ and $L_{n,j}/d_n\to1/k$ as $n\to\infty$ for every $j=1,\ldots,k$.
\end{itemize}
Then, for every $n$ there exists a partition $\{\Lambda_{n,1},\ldots,\Lambda_{n,k}\}$ of $\Lambda_n$ such that, for every $j=1,\ldots,k$, the following properties hold.
\begin{itemize}[nolistsep,leftmargin=*]
	\item $|\Lambda_{n,j}|=L_{n,j}$.
	\item $\{\Lambda_{n,j}\}_n\sim f_j$. 
\end{itemize}
\end{lemma}

\begin{lemma}\label{ss2.g}
Let $f_1,\ldots,f_k:[a,b]\to\mathbb R$ be measurable. Let $\{\Lambda_n=\{\lambda_{1,n},\ldots,\lambda_{d_n,n}\}\}_n$ be a sequence of finite multisets of real numbers such that $d_n\to\infty$ as $n\to\infty$. Assume the following. 
\begin{itemize}[nolistsep,leftmargin=*]
	\item $\{\Lambda_n\}_n\sim f$ with $f={\rm diag}(f_1,\ldots,f_k)$. 
	\item For every $n$ there exists a partition $\{\tilde\Lambda_{n,1},\ldots,\tilde\Lambda_{n,k}\}$ of $\Lambda_n$ such that, for every $j=1,\ldots,k$, $|\tilde\Lambda_{n,j}|/d_n\to1/k$ as $n\to\infty$ and $\tilde\Lambda_{n,j}\subseteq(\mathcal{ER}(f_j))_{\epsilon_n}$ for some $\epsilon_n\to0$ as $n\to\infty$.
\end{itemize}
Then, for every $n$ there exists a partition $\{\Lambda_{n,1},\ldots,\Lambda_{n,k}\}$ of $\Lambda_n$ 
such that, for every $j=1,\ldots,k$, the following properties hold.
\begin{itemize}[nolistsep,leftmargin=*]
	\item $|\Lambda_{n,j}|=|\tilde\Lambda_{n,j}|$.
	\item $\{\Lambda_{n,j}\}_n\sim f_j$. 
	\item $\Lambda_{n,j}\subseteq(\mathcal{ER}(f_j))_{\delta_n}$ for some $\delta_n\to0$ as $n\to\infty$.
\end{itemize}
\end{lemma}

\section{Proofs of the main results}\label{proofs}

\subsection{Monotone rearrangement}

In this section, we recall the notion of monotone rearrangement and we collect some related results that we shall need in the proof of Theorem~\ref{th:main4-n}.

\begin{definition}
Let $f:\Omega\subset\mathbb R^d\to\mathbb R$ be measurable on a set $\Omega$ with $0<\mu_d(\Omega)<\infty$. 
The monotone rearrangement of $f$ is the function denoted by $f^\dag$ and defined as follows:
\begin{equation}\label{crv}
f^\dag(y)=\inf\biggl\{u\in\mathbb R:\frac{\mu_d\{f\le u\}}{\mu_d(\Omega)}\ge y\biggr\},\qquad y\in(0,1).
\end{equation}
\end{definition}

Note that $f^\dag(y)$ is a well-defined real number for every $y\in(0,1)$, because
\[ \lim_{u\to+\infty}\mu_d\{f\le u\}=\mu_d(\Omega),\qquad\lim_{u\to-\infty}\mu_d\{f\le u\}=0. \]
In probability theory, where $f$ is interpreted as a random variable on $\Omega$ with probability distribution $m_f$ and distribution function $F_f$ given by
\begin{align*}
m_f(A)&=\mathbb P\{f\in A\}=\frac{\mu_d\{f\in A\}}{\mu_d(\Omega)},\qquad A\subseteq\mathbb R\mbox{ is a Borel set},\\
F_f(u)&=\mathbb P\{f\le u\}=\frac{\mu_d\{f\le u\}}{\mu_d(\Omega)},\qquad u\in\mathbb R,
\end{align*}
the monotone rearrangement $f^\dag$ in \eqref{crv} can be rewritten as
\[ f^\dag(y)=\inf\bigl\{u\in\mathbb R:F_f(u)\ge y\bigr\},\qquad y\in(0,1), \]
and is referred to as the quantile function of $f$ or the generalized inverse of $F_f$.
It is clear that $f^\dag$ is monotone increasing on $(0,1)$, which implies that the limits $\lim_{y\to0}f^\dag(y)$ and $\lim_{y\to1}f^\dag(y)$ always exist.
The next lemma gives the exact values of these limits and allows us to complete the definition of $f^\dag$ by continuous extension; see Definition~\ref{def_ext}.

\begin{lemma}\label{lemma_ext}
Let $f:\Omega\subset\mathbb R^d\to\mathbb R$ be measurable on a set $\Omega$ with $0<\mu_d(\Omega)<\infty$. Then,
\[ \lim_{y\to0}f^\dag(y)=\mathop{\rm ess\,inf}_\Omega f,\qquad \lim_{y\to1}f^\dag(y)=\mathop{\rm ess\,sup}_\Omega f. \]
\end{lemma}
\begin{proof}
We only prove the equality $\lim_{y\to0}f^\dag(y)=\mathop{\rm ess\,inf}_\Omega f$ as the proof of the other equality is analogous.

\medskip

\noindent
{\em Case 1:} $\mathop{\rm ess\,inf}_\Omega f=-\infty$. In this case, by definition of $\mathop{\rm ess\,inf}_\Omega f$, we have $\mu_d\{f\le u\}>0$ for all $u\in\mathbb R$, i.e., $F_f(u)>0$ for all $u\in\mathbb R$. Hence, for every $u_0\in\mathbb R$ there exists $y_0=F_f(u_0)/2\in(0,1)$ such that, for $0<y\le y_0$, 
\[ f^\dag(y)\le f^\dag(y_0)=\inf\{u\in\mathbb R:F_f(u)\ge y_0\}\le u_0. \] 
This means that $\lim_{y\to0}f^\dag(y)=-\infty$.

\medskip

\noindent
{\em Case 2:} $\mathop{\rm ess\,inf}_\Omega f=m\in\mathbb R$. In this case, since $f\in\mathcal{ER}(f)$ a.e., we have $\mu_d\{f\le u\}=0$ for all $u<m$, i.e., $F_f(u)=0$ for all $u<m$. Thus, for every $y\in(0,1)$,
\[ f^\dag(y)=\inf\{u\in\mathbb R:F_f(u)\ge y\}\ge m\quad\implies\quad\lim_{y\to0}f^\dag(y)\ge m. \]
To prove the other inequality, fix $\epsilon>0$. By definition of $m$, we have
\[ \mu_d\{f\le m+\epsilon\}>0\quad\implies\quad F_f(m+\epsilon)=\frac{\mu_d\{f\le m+\epsilon\}}{\mu_d(\Omega)}=\alpha_\epsilon\in(0,1]. \]
Hence, $\alpha_\epsilon/2\in(0,1)$ and
\[ f^\dag(\alpha_\epsilon/2)=\inf\{u\in\mathbb R:F_f(u)\ge\alpha_\epsilon/2\}\le m+\epsilon\quad\implies\quad\lim_{y\to0}f^\dag(y)\le m+\epsilon. \]
Since this is true for all $\epsilon>0$, we conclude that $\lim_{y\to0}f^\dag(y)\le m$.
\end{proof}

\begin{definition}[\textbf{monotone rearrangement}]
\label{def_ext}
Let $f:\Omega\subset\mathbb R^d\to\mathbb R$ be measurable on a set $\Omega$ with $0<\mu_d(\Omega)<\infty$.
The monotone rearrangement of $f$ is the function denoted by $f^\dag$ and defined as follows:
\begin{alignat*}{3}
f^\dag(y)&=\inf\biggl\{u\in\mathbb R:\frac{\mu_d\{f\le u\}}{\mu_d(\Omega)}\ge y\biggr\}, &&\qquad\mbox{if }\,y\in(0,1),\\ 
f^\dag(0)&=\mathop{\rm ess\,inf}_\Omega f, &&\qquad\mbox{if }\,\mathop{\rm ess\,inf}_\Omega f>-\infty,\\ 
f^\dag(1)&=\mathop{\rm ess\,sup}_\Omega f, &&\qquad\mbox{if }\,\mathop{\rm ess\,sup}_\Omega f<\infty. 
\end{alignat*}
The domain of $f^\dag$ is denoted by $\Omega^\dag$ and is always assumed to be the largest possible.
This means that $\Omega^\dag$ always includes $(0,1)$ and it also includes $0$ (resp., $1$) whenever $f^\dag(0)$ (resp., $f^\dag(1)$) is defined, i.e., whenever $\mathop{\rm ess\,inf}_\Omega f>-\infty$ (resp., $\mathop{\rm ess\,sup}_\Omega f<\infty$).
\end{definition}

We remark that, according to Definition~\ref{def_ext} and Lemma~\ref{lemma_ext}, $f^\dag$ is always continuous at $0$ and $1$ whenever it is defined there. In particular, the discontinuity points of $f^\dag$, if any, must lie in $(0,1)$. The next lemma collects some basic properties of monotone rearrangements.

\begin{lemma}\label{lemma_completo}
Let $f:\Omega\subset\mathbb R^d\to\mathbb R$ be measurable on a set $\Omega$ with $0<\mu_d(\Omega)<\infty$.
\begin{enumerate}[nolistsep,leftmargin=*]
	\item $f^\dag$ is monotone increasing and left-continuous on $\Omega^\dag$. 
	\item For every Borel set $A\subseteq\mathbb R$, we have
	\[ \mu_1\{f^\dag\in A\}=\dfrac{\mu_d\{f\in A\}}{\mu_d(\Omega)}. \] 
	\item For every continuous bounded function $F:\mathbb R\to\mathbb R$, we have
	\[ \frac1{\mu_d(\Omega)}\int_\Omega F(f(\xx)){\rm d}\xx=\int_0^1F(f^\dag(y)){\rm d}y. 
	\]
	\item $\mathcal{ER}(f^\dag)=\mathcal{ER}(f)$.
	\item For every $y\in\Omega^\dag$ we have $f^\dag(y)\in\mathcal{ER}(f^\dag)$.
	\item If $f^\dag$ is continuous on $(0,1)$ then $\mathcal{ER}(f^\dag)=f^\dag(\Omega^\dag)$.
	\item $f^\dag$ is continuous on $(0,1)$ if and only if $\mathcal{ER}(f)$ is connected. 
\end{enumerate}
\end{lemma}
\begin{proof}
1--2. See \cite[Chapter~3, Proposition~4 and Problem~3]{Fristedt}.

3. See \cite[Chapter~14, Proposition~7]{Fristedt}.

4. By definition of $\mathcal{ER}(f)$ and property~2,
\begin{align*}
\mathcal{ER}(f)&=\{x\in\mathbb R:\mu_d\{f\in D(x,\epsilon)\}>0\mbox{ for all }\epsilon>0\}\\
&=\{x\in\mathbb R:\mu_d\{f^\dag\in D(x,\epsilon)\}>0\mbox{ for all }\epsilon>0\}=\mathcal{ER}(f^\dag).
\end{align*}

5. Let $y\in\Omega^\dag$. 
$f^\dag$ is left-continuous at $y$ (by property~1), and it is continuous at $y$ if $y=0$ or $y=1$ (because $f^\dag$ is always continuous at $0$ and $1$ whenever it is defined there). Thus, we have $\mu_1\{f^\dag\in D(f^\dag(y),\epsilon)\}>0$ for every $\epsilon>0$. 
Hence, $y\in\mathcal{ER}(f^\dag)$.

6. Suppose that $f^\dag$ is continuous on $(0,1)$. This means that $f^\dag$ is continuous on $\Omega^\dag$, since $f^\dag$ is always continuous at $0$ and $1$ whenever it is defined there. By property~5, we infer that $f^\dag(\Omega^\dag)\subseteq\mathcal{ER}(f^\dag)$. On the other hand, $\mathcal{ER}(f^\dag)\subseteq\overline{f^\dag(\Omega^\dag)}=f^\dag(\Omega^\dag)$, where the latter equality follows from the fact that $f^\dag(\Omega^\dag)$ is closed by definition of $\Omega^\dag$ and the continuity and monotonicity of $f^\dag$. Thus, $\mathcal{ER}(f^\dag)=f^\dag(\Omega^\dag)$.

7. $(\implies)$ Suppose that $f^\dag$ is continuous on $(0,1)$, i.e., continuous on $\Omega^\dag$. By properties~4 and~6, we have
\[ \mathcal{ER}(f)=\mathcal{ER}(f^\dag)=f^\dag(\Omega^\dag). \]
In particular, $\mathcal{ER}(f)$ is connected because it is equal to the image of a connected set $\Omega^\dag$ through a continuous function $f^\dag$; see \cite[Theorem~4.22]{Rudinino}.

$(\impliedby)$ Suppose that $f^\dag$ is not continuous on $(0,1)$. Then, there exists for $f^\dag$ a discontinuity point $y_0\in(0,1)$, which is necessarily a jump because $f^\dag$ is monotone increasing. In particular, we can find a point $\alpha$ in the open jump $(f^\dag_-(y_0),f^\dag_+(y_0))$, where $f^\dag_-(y_0)$ and $f^\dag_+(y_0)$ are the left and right limits of $f^\dag$ in $y_0$ (note that $f^\dag_-(y_0)=f^\dag(y_0)$ by the left-continuity of $f^\dag$; see property~1). Obviously, $\alpha\not\in\mathcal{ER}(f^\dag)$. We show that $\alpha$ disconnects $\mathcal{ER}(f^\dag)$. Take two 
points $u,v\in(0,1)$  
such that $u<y_0<v$. 
By property~5, we have $f^\dag(u),f^\dag(v)\in\mathcal{ER}(f^\dag)$. Moreover, by monotonicity, $f^\dag(u)<\alpha<f^\dag(v)$. Hence, $\mathcal{ER}(f^\dag)=\mathcal{ER}(f)$ is not connected.
\end{proof}

The main results we need on monotone rearrangements are Theorems~\ref{th:mainEM} and~\ref{th:mainEM'}.
For the proof of Theorem~\ref{th:mainEM}, see \cite[Theorem~3.1]{EM}.
Theorem~\ref{th:mainEM'} is a generalization of \cite[Theorem~3.2]{EM} and is proved below.

\begin{theorem}\label{th:mainEM}
Let $f:\Omega\subset\mathbb R^d\to\mathbb R$ be continuous a.e.\ on the regular set $\Omega$ with $\mu_d(\Omega)>0$.
Take any $d$-dimensional rectangle $[\aa,\bb]$ containing $\Omega$ and any a.u.\ grid $\mathcal G_\nn^{(n)}=\{\xx_{\ii,\nn}^{(n)}\}_{\ii=\bu,\ldots,\nn}$ in $[\aa,\bb]$, where $\nn=\nn(n)\in\mathbb N^d$ and $\nn\to\infty$ as $n\to\infty$. 
For every $n$, consider the samples
\begin{equation*}
f(\xx_{\ii,\nn}^{(n)}),\qquad\ii\in\II_\nn^{(n)}(\Omega)=\{\ii\in\{\bu,\ldots,\nn\}:\xx_{\ii,\nn}^{(n)}\in\Omega\},
\end{equation*}
sort them in increasing order, and put them into a vector $[s_0,\ldots,s_{\omega(n)}]$, where $\omega(n)=|\II_\nn^{(n)}(\Omega)|-1$.
Let $f^\dag_n:[0,1]\to\mathbb R$ be the linear spline function that interpolates the samples $[s_0,\ldots,s_{\omega(n)}]$ over the equally spaced nodes $[0,\frac{1}{\omega(n)},\frac{2}{\omega(n)},\ldots,1]$ in $[0,1]$. Then,
\[ \lim_{n\to\infty} f^\dag_n(y)=f^\dag(y) \]
for every continuity point $y$ of $f^\dag$ in $(0,1)$. 
\end{theorem}

\begin{theorem}\label{th:mainEM'}
In Theorem~{\rm\ref{th:mainEM}}, suppose that $\mathcal{ER}(f)$ is connected. Then, the following properties hold.
\begin{enumerate}[nolistsep,leftmargin=*]
	\item $f_n^\dag\to f^\dag$ uniformly on every compact interval $[\alpha,\beta]\subset(0,1)$.
	\item If $f$ is bounded from below on $\Omega$ with $\inf_\Omega f=\mathop{\rm ess\,inf}_\Omega f$, then $f_n^\dag\to f^\dag$ uniformly on every compact interval $[0,\alpha]\subset[0,1)$.
	\item If $f$ is bounded from above on $\Omega$ with $\sup_\Omega f=\mathop{\rm ess\,sup}_\Omega f$, then $f_n^\dag\to f^\dag$ uniformly on every compact interval $[\alpha,1]\subset(0,1]$.
	\item If $f$ is bounded on $\Omega$ and $\mathcal{ER}(f)=[\inf_\Omega f,\sup_\Omega f]$, then $f_n^\dag\to f^\dag$ uniformly on $[0,1]$.
\end{enumerate}
\end{theorem}

Recall from Remark~\ref{ER(f)=[inf,sup]} that the assumption $\inf_\Omega f=\mathop{\rm ess\,inf}_\Omega f$ in the second statement is equivalent to $\inf_\Omega f\ge\mathop{\rm ess\,inf}_\Omega f$ and the assumption $\sup_\Omega f=\mathop{\rm ess\,sup}_\Omega f$ in the third statement is equivalent to $\sup_\Omega f\le\mathop{\rm ess\,sup}_\Omega f$. Moreover, under the hypothesis that $\mathcal{ER}(f)$ is connected and bounded,
the assumption $\mathcal{ER}(f)=[\inf_\Omega f,\sup_\Omega f]$ in the fourth statement is equivalent to asking that $\inf_\Omega f=\mathop{\rm ess\,inf}_\Omega f$ and $\sup_\Omega f=\mathop{\rm ess\,sup}_\Omega f$, which in turn is equivalent to asking that $f(\Omega)\subseteq\mathcal{ER}(f)$, i.e., $\overline{f(\Omega)}=\mathcal{ER}(f)$.
The proof of Theorem~\ref{th:mainEM'} relies on the following lemma, which is sometimes referred to as Dini's second theorem \cite[pp.~81 and~270, Problem~127]{Polya-Szego}.

\begin{lemma}\label{Dini}
If a sequence of monotone functions converges pointwise on a compact interval to a continuous function, then it converges uniformly.
\end{lemma}

\begin{proof}[Proof of Theorem~{\rm\ref{th:mainEM'}}]\hfill
\begin{enumerate}[nolistsep,leftmargin=*]
	\item $f^\dag$ is continuous on $(0,1)$ by Lemma~\ref{lemma_completo} and $f_n^\dag\to f^\dag$ everywhere in $(0,1)$ by Theorem~\ref{th:mainEM}. Since the functions $f_n^\dag$ are continuous and increasing on $(0,1)$, the thesis follows from Lemma~\ref{Dini}.
	\item Since $f$ is bounded from below on $\Omega$, we have $\mathop{\rm ess\,inf}_\Omega f>-\infty$ and $f^\dag(0)=\lim_{y\to0}f^\dag(y)=\mathop{\rm ess\,inf}_\Omega f$ is defined. 
	Moreover, the function $f^\dag$ is continuous on $[0,1)$ by Lemma~\ref{lemma_completo} and the definition $f^\dag(0)=\lim_{y\to0}f^\dag(y)$.
	Since $f_n^\dag(0)=s_0$ is the evaluation of $f$ at a point of $\Omega$ and $\inf_\Omega f=\mathop{\rm ess\,inf}_\Omega f$ by assumption, for every $n$ we have 
	\[ f_n^\dag(0)\ge{\rm inf}_\Omega f=\mathop{\rm ess\,inf}_\Omega f=f^\dag(0). \]
	Since $f^\dag_n\to f^\dag$ everywhere in $(0,1)$ by Theorem~\ref{th:mainEM} and the functions $f_n^\dag$, $f^\dag$ are continuous and increasing on $[0,1)$, for every $\epsilon>0$ we have
	\[ f^\dag(0)\le\liminf_{n\to\infty}f_n^\dag(0)\le\limsup_{n\to\infty}f_n^\dag(0)\le\limsup_{n\to\infty}f_n^\dag(\epsilon)=f^\dag(\epsilon), \]
	hence
	\[ f^\dag(0) = \lim_{n\to\infty} f_n^\dag(0). \]
	We have therefore proved that $f_n^\dag\to f^\dag$ everywhere in $[0,1)$, and the thesis now follows from Lemma~\ref{Dini}.
	\item It is proved in the same way as the second statement.
	\item It follows immediately from the second and third statements. \qedhere
\end{enumerate}
\end{proof}

\subsection{Proof of Theorem~\ref{th:main4-n}}

The last result we need to prove Theorem~\ref{th:main4-n} is the following technical lemma \cite[Lemma~3.3]{EM}.

\begin{lemma}\label{ultimo}
Let $\omega(n)$ be a sequence of positive integers such that $\omega(n)\to\infty$ and let $g_n:[0,1]\to\mathbb R$ be a sequence of increasing functions such that
\[ \lim_{n\to\infty}\frac1{\omega(n)}\sum_{\ell=0}^{\omega(n)}F\Bigl(g_n\Bigl(\frac\ell{\omega(n)}\Bigr)\Bigr)=\int_0^1F(g(y)){\rm d}y,\qquad\forall\,F\in C_c(\mathbb R), \]
where $g:(0,1)\to\mathbb R$ is increasing. Then, $g_n(y)\to g(y)$ for every continuity point $y$ of $g$ in $(0,1)$.
\end{lemma}

\begin{proof}[Proof of Theorem~{\rm\ref{th:main4-n}}]
We begin with the following observation: since $\mathcal{ER}(f)$ is connected and bounded by assumption, the domain of $f^\dag$ is $\Omega^\dag=[0,1]$ and $f^\dag$ is continuous on $[0,1]$; see Lemma~\ref{lemma_completo} (property~7) and Definition~\ref{def_ext}.

Sort the samples $\{f(\xx_{\ii,\nn}^{(n)}):\ii\in\II_\nn^{(n)}(\Omega)\}$ in increasing order through a permutation $\sigma_n$ as in the statement of the theorem, and put them into a vector $[s_0,\ldots,s_{\omega(n)}]$, where $\omega(n)=|\II_\nn^{(n)}(\Omega)|-1$. Note that 
$[s_0,\ldots,s_{\omega(n)}]=[f(\xx_{\sigma_n(\ii),\nn})]_{\ii\in\II_\nn^{(n)}(\Omega)}$. 
Let $f^\dag_n:[0,1]\to\mathbb R$ be the linear spline function that interpolates the samples $[s_0,\ldots,s_{\omega(n)}]$ over the equally spaced nodes $[0,\frac{1}{\omega(n)},\frac{2}{\omega(n)},\ldots,1]$ in $[0,1]$. By Theorem~\ref{th:mainEM'}, 
\begin{equation}\label{fdag_c1-n}
f_n^\dag\to f^\dag\ \mbox{uniformly on $[0,1]$ as}\ n\to\infty.
\end{equation}

Sort the real numbers $\{\lambda_{\ii,\nn}^{(n)}:\ii\in\II_\nn^{(n)}(\Omega)\}$ in increasing order through a permutation $\tau_n$ as in the statement of the theorem, and put them into a vector $[t_0,\ldots,t_{\omega(n)}]$. Note that 
$[t_0,\ldots,t_{\omega(n)}]=[\lambda_{\tau_n(\ii),\nn}^{(n)}]_{\ii\in\II_\nn^{(n)}(\Omega)}$.
Let $g_n:[0,1]\to\mathbb R$ be the linear spline function that interpolates the samples $[t_0,\ldots,t_{\omega(n)}]$ over the equally spaced nodes $[0,\frac{1}{\omega(n)},\frac{2}{\omega(n)},\ldots,1]$ in $[0,1]$. Since $\{\{\lambda_{\ii,\nn}^{(n)}:\ii\in\II_\nn^{(n)}(\Omega)\}\}_n\sim f^\dag$ by the assumption $\{\{\lambda_{\ii,\nn}^{(n)}:\ii\in\II_\nn^{(n)}(\Omega)\}\}_n\sim f$ and Lemma~\ref{lemma_completo} (property~3), the hypotheses of Lemma~\ref{ultimo} are satisfied with  
$g=f^\dag$. Since $f^\dag$ is continuous on $[0,1]$, we infer that $g_n(y)\to f^\dag(y)$ for every $y\in(0,1)$. Moreover, $g_n(y)\to f^\dag(y)$ also for $y=0,1$. Indeed, for every $n$ we have 
\[ g_n(0)\ge f^\dag(0)-\epsilon_n, \]
because on the one hand $f^\dag(0)=\mathop{\rm ess\,inf}_\Omega f=\inf_\Omega f$ by our assumption $\mathcal{ER}(f)=[\inf_\Omega f,\sup_\Omega f]$, and on the other hand $\{\lambda_{\ii,\nn}^{(n)}:\ii\in\II_\nn^{(n)}(\Omega)\}\subseteq[\inf_\Omega f-\epsilon_n,\sup_\Omega f+\epsilon_n]$ by hypothesis.
Thus, for every $\epsilon>0$ we have
\[ f^\dag(0)\le\liminf_{n\to\infty}g_n(0)\le\limsup_{n\to\infty}g_n(0)\le\limsup_{n\to\infty}g_n(\epsilon)=f^\dag(\epsilon), \]
and so, by the continuity of $f^\dag$ at $0$,
\[ \lim_{n\to\infty}g_n(0)=f^\dag(0). \]
The same argument shows that $g_n(1)\to f^\dag(1)$ as $n\to\infty$. We conclude that $g_n(y)\to f^\dag(y)$ for all $y\in[0,1]$ and so, by Lemma~\ref{Dini},
\begin{equation}\label{fdag_c2-n}
g_n\to f^\dag\ \mbox{uniformly on $[0,1]$ as}\ n\to\infty.
\end{equation}
By combining \eqref{fdag_c1-n} and \eqref{fdag_c2-n}, we obtain that $\|f^\dag_n-g_n\|_{\infty,[0,1]}\to0$ as $n\to\infty$. In particular,
\[ \max_{i=0,\ldots,\omega(n)}\left|f^\dag_n\biggl(\frac i{\omega(n)}\biggr)-g_n\biggl(\frac i{\omega(n)}\biggr)\right|=\max_{i=0,\ldots,\omega(n)}|s_i-t_i|=\max_{\ii\in\II_\nn^{(n)}(\Omega)}|f(\xx_{\sigma_n(\ii),\nn}^{(n)})-\lambda_{\tau_n(\ii),\nn}^{(n)}|\to0\ \,\mbox{as}\ \,n\to\infty, \]
which proves the theorem.
\end{proof}

\subsection{Properties of a.u.\ grids}

In this section, we collect the properties of a.u.\ grids that we need in the proof of Theorem~\ref{thm:mmd}.
We begin with the following general result on real vectors. If $A\in\mathbb C^{m\times m}$, we denote by $\|A\|=\max(\sigma_1(A),\ldots,\sigma_m(A))$ the spectral (or Euclidean) norm of $A$.

\begin{lemma}\label{gelemma}
Let $\xx,\yy\in\mathbb R^m$ and let $\sigma,\tau$ be permutations of $\{1,\ldots,m\}$ that sort the components of $\xx,\yy$ in increasing order, i.e., $x_{\sigma(1)}\le\ldots\le x_{\sigma(m)}$ and $y_{\tau(1)}\le\ldots\le y_{\tau(m)}$. Then,
\begin{equation}\label{geresult}
\max_{i=1,\ldots,m}|x_{\sigma(i)}-y_{\tau(i)}|\le\max_{i=1,\ldots,m}|x_i-y_i|.
\end{equation}
\end{lemma}
\begin{proof}
By Weyl's perturbation theorem \cite[Corollary~III.2.6]{Bhatia}, for every pair of $m\times m$ Hermitian matrices $A$ and $B$ we have
\[ \max_{i=1,\ldots,m}|\lambda_i(A)-\lambda_i(B)|\le\|A-B\|, \]
where the eigenvalues of $A$ and $B$ are arranged in increasing order: $\lambda_1(A)\le\ldots\le\lambda_m(A)$ and $\lambda_1(B)\le\ldots\le\lambda_m(B)$. If we apply this result to the real diagonal matrices $A={\rm diag}(x_1,\ldots,x_m)$ and $B={\rm diag}(y_1,\ldots,y_m)$, we obtain \eqref{geresult}.
\end{proof}

As a consequence of Lemma~\ref{gelemma}, the increasing rearrangement of an a.u.\ grid is still an a.u.\ grid.

\begin{lemma}[\textbf{increasing rearrangement of an a.u.\ grid is still an a.u.\ grid}]\label{order_au}
Let $\mathcal G_n=\{x_{i,n}\}_{i=1,\ldots,d_n}$ be an a.u.\ grid in $[a,b]$ with $d_n\to\infty$, and let $\tau_n$ be a permutation of $\{1,\ldots,d_n\}$ such that $x_{\tau_n(1),n}\le\ldots\le x_{\tau_n(d_n),n}$. Then $\mathcal G_n'=\{x_{\tau_n(i),n}\}_{i=1,\ldots,d_n}$ is an a.u.\ grid in $[a,b]$ with $m(\mathcal G'_n)\le m(\mathcal G_n)$ for all $n$.
\end{lemma}
\begin{proof}
We apply Lemma~\ref{gelemma} with $\xx=[x_{i,n}]_{i=1}^{d_n}$ and $\yy=[a+i(b-a)/d_n]_{i=1}^{d_n}$, and we obtain
\[ m(\mathcal G_n')=\max_{i=1,\ldots,d_n}\left|x_{\tau_n(i),n}-\biggl(a+i\,\frac{b-a}{d_n}\biggr)\right|\le\max_{i=1,\ldots,d_n}\left|x_{i,n}-\biggl(a+i\,\frac{b-a}{d_n}\biggr)\right|=m(\mathcal G_n). \tag*{\qedhere} \]
\end{proof}

In view of the next lemma, we point out that if $\mathcal G_n=\{x_{i,n}\}_{i=1,\ldots,d_n}$ is an a.u.\ grid in $[a,b]$ then $\mathcal G_n$ is a sequence of points for each fixed~$n$. Moreover, every sequence of points $\mathcal G_n=\{x_{i,n}\}_{i=1,\ldots,d_n}$ can also be interpreted as a multiset consisting of $d_n$ elements. In particular, if $\mathcal G_n=\{x_{i,n}\}_{i=1,\ldots,d_n}$ and $\mathcal G'_n=\{x'_{i,n}\}_{i=1,\ldots,d'_n}$ are two sequences of points (and hence two multisets), the intersection $\mathcal G_n\cap\mathcal G_n'$, the union $\mathcal G_n\cup\mathcal G_n'$,  
the difference $\mathcal G_n\setminus\mathcal G_n'$, 
the symmetric difference $\mathcal G_n\triangle\mathcal G_n'=(\mathcal G_n\setminus\mathcal G_n')\cup(\mathcal G_n'\setminus\mathcal G_n)$, etc., are well-defined multisets.
Throughout this paper, if $\{a_n\}_n$ and $\{b_n\}_n$ are any two sequences such that $a_n,b_n\ne0$ for all $n$, we write $a_n\sim b_n$ as $n\to\infty$ to indicate that $a_n/b_n\to1$ as $n\to\infty$.

\begin{lemma}[\textbf{multiset differing little from an a.u.\ grid is an a.u.\ grid}]\label{fonda.u.}
Let $\mathcal G_n=\{x_{i,n}\}_{i=1,\ldots,d_n}$ be an a.u.\ grid in $[a,b]$ with $d_n\to\infty$, and let $\mathcal G_n'=\{x_{i,n}'\}_{i=1,\ldots,d_n'}$ be a sequence of $d_n'$ real points such that $\mathcal G_n'\subset[a-\epsilon_n,b+\epsilon_n]$ for some $\epsilon_n\to0$ and $|\mathcal G_n\triangle\mathcal G_n'|=o(d_n)$ as $n\to\infty$. Then, $d_n'\sim d_n$ as $n\to\infty$ and $\mathcal G_n'$ is an a.u.\ grid in $[a,b]$ provided that we rearrange its points in increasing order.
\end{lemma}
\begin{proof}
Without loss of generality, we can assume that the points of $\mathcal G_n$ and $\mathcal G_n'$ are arranged in increasing order:
\[ x_{1,n}\le\ldots\le x_{d_n,n},\qquad x_{1,n}'\le\ldots\le x_{d_n',n}'. \]
Indeed, after rearranging the points of $\mathcal G_n$ and $\mathcal G_n'$ in increasing order (if necessary), the assumptions of the lemma are still satisfied, because $\mathcal G_n$ remains an a.u.\ grid in $[a,b]$ by Lemma~\ref{order_au} and $\mathcal G_n\triangle\mathcal G_n'$ does not change. 
To simplify the notation, we prove the lemma in the case $[a,b]=[0,1]$; up to obvious modifications, the proof for a general interval $[a,b]$ is the same as the proof for the interval $[0,1]$. 
Let $a_n=|\mathcal G_n\cap\mathcal G_n'|$ and $b_n=|\mathcal G_n\triangle\mathcal G_n'|=o(d_n)$. 
Let $I_n=\{i_1,\ldots,i_{a_n}\}$ and $J_n=\{j_1,\ldots,j_{a_n}\}$ be two sets of indices such that
\begin{alignat*}{3}
\{x_{i_1,n},\ldots,x_{i_{a_n},n}\}&=\mathcal G_n\cap\mathcal G_n', &\qquad1&\le i_1<\ldots<i_{a_n}\le d_n,\\
\{x'_{j_1,n},\ldots,x'_{j_{a_n},n}\}&=\mathcal G_n\cap\mathcal G_n', &\qquad1&\le j_1<\ldots<j_{a_n}\le d_n';
\end{alignat*}
see Figure~\ref{griglie_au}. We make the following observations.
\begin{itemize}[nolistsep,leftmargin=*]
	\item $x_{i_k,n}=x'_{j_k,n}$ for all $k=1,\ldots,a_n$. Indeed, $x_{i_1,n},\ldots,x_{i_{a_n},n}$ and $x'_{j_1,n},\ldots,x'_{j_{a_n},n}$ are the points of the same multiset $\mathcal G_n\cap\mathcal G_n'$ arranged in increasing order. 
	\item $|d_n-d_n'|\le b_n$. Indeed, if $c_n=|\mathcal G_n\setminus\mathcal G_n'|$ and $e_n=|\mathcal G_n'\setminus\mathcal G_n|$, then 
	\begin{align*}
	d_n-c_n&=|\mathcal G_n|-|\mathcal G_n\setminus\mathcal G_n'|=|\mathcal G_n\cap\mathcal G_n'|=a_n,\\
	d_n'-e_n&=|\mathcal G_n'|-|\mathcal G_n'\setminus\mathcal G_n|=|\mathcal G_n'\cap\mathcal G_n|=a_n,\\
	c_n+e_n&=|\mathcal G_n\setminus\mathcal G_n'|+|\mathcal G_n'\setminus\mathcal G_n|=|\mathcal G_n\triangle\mathcal G_n'|=b_n,
	\end{align*}
	and
	\[ |d_n-d_n'|=|d_n-c_n+c_n-e_n+e_n-d_n'|=|a_n+c_n-e_n-a_n|=|c_n-e_n|\le c_n+e_n=b_n. \]
	\item $|i_k-j_k|\le b_n$ for all $k=1,\ldots,a_n$. Indeed, if $c_{n,k}=|\mathcal G_{n,k}\setminus\mathcal G_n'|$ with $\mathcal G_{n,k}=\{x_{1,n},\ldots,x_{i_k,n}\}\subseteq\mathcal G_n$ and $e_{n,k}=|\mathcal G_{n,k}'\setminus\mathcal G_n|$ with $\mathcal G_{n,k}'=\{x_{1,n}',\ldots,x_{j_k,n}'\}\subseteq\mathcal G_n'$, then\,\footnote{\,To follow the reasoning, look at Figure~\ref{griglie_au} and assume, for example, that $k=3$. Then, $i_3=8$, $j_3=10$, $c_{n,3}=5$, $e_{n,3}=7$.}
	\begin{figure}
	\centering
	\includegraphics[width=\textwidth]{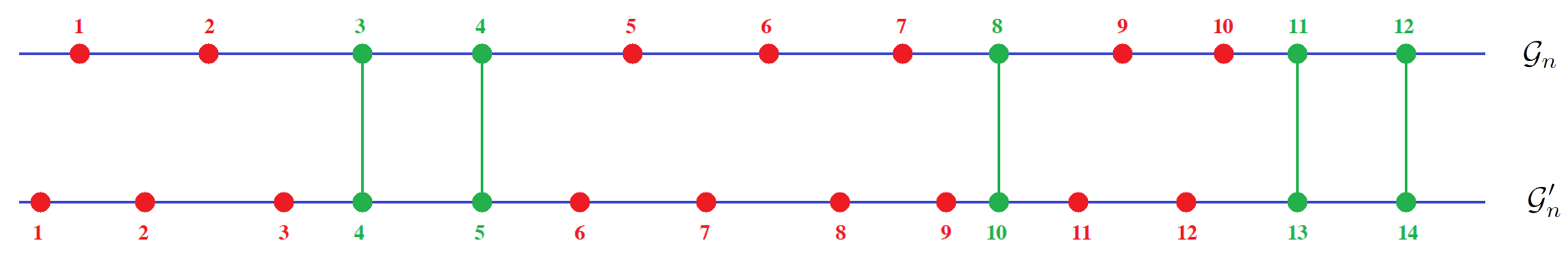}
	\caption{Illustration for the proof of Lemma~\ref{fonda.u.}. The dots on the first line are the points of $\mathcal G_n$ and the dots on the second line are the points of $\mathcal G_n'$. The green dots are the points of $\mathcal G_n\cap\mathcal G_n'$ while the red dots are the points of $\mathcal G_n\triangle\mathcal G_n'$. In this case, we have $d_n=12$, $d_n'=14$, $a_n=5$, $b_n=16$,  
$I_n=\{3,4,8,11,12\}$, $J_n=\{4,5,10,13,14\}$. 
	}
	\label{griglie_au}
	\end{figure}
	\begin{align*}
	i_k-c_{n,k}&=|\mathcal G_{n,k}|-|\mathcal G_{n,k}\setminus\mathcal G_n'|=|\mathcal G_{n,k}\cap\mathcal G_n'|=|\mathcal G_{n,k}\cap(\mathcal G_n\cap\mathcal G_n')|=|\{x_{i_1,n},\ldots,x_{i_k,n}\}|=k,\\
	j_k-e_{n,k}&=|\mathcal G_{n,k}'|-|\mathcal G_{n,k}'\setminus\mathcal G_n|=|\mathcal G_{n,k}'\cap\mathcal G_n|=|\mathcal G_{n,k}'\cap(\mathcal G_n'\cap\mathcal G_n)|=|\{x_{j_1,n}',\ldots,x_{j_k,n}'\}|=k,\\
	c_{n,k}+e_{n,k}&=|\mathcal G_{n,k}\setminus\mathcal G_n'|+|\mathcal G_{n,k}'\setminus\mathcal G_n|\le|\mathcal G_n\setminus\mathcal G_n'|+|\mathcal G_n'\setminus\mathcal G_n|=c_n+e_n=b_n,
	\end{align*}
	and
	\[ |i_k-j_k|=|i_k-c_{n,k}+c_{n,k}-e_{n,k}+e_{n,k}-j_k|=|k+c_{n,k}-e_{n,k}-k|=|c_{n,k}-e_{n,k}|\le c_{n,k}+e_{n,k}\le b_n. \]
\end{itemize}
From $|d_n-d_n'|\le b_n$ and the assumption $b_n=o(d_n)$, we immediately obtain that
\[
\left|\frac{d_n'}{d_n}-1\right|=\left|\frac{d_n'-d_n}{d_n}\right|\le\frac{b_n}{d_n}\to0\,\mbox{ as }\,n\to\infty\quad\implies\quad d_n'\sim d_n\,\mbox{ as }\,n\to\infty. 
\]
It remains to prove that $\mathcal G_n'$ is a.u.\ in $[0,1]$, i.e.,
\begin{equation}\label{max->0}
\lim_{n\to\infty}m(\mathcal G_n')=0,\qquad m(\mathcal G_n')=\max_{j=1,\ldots,d_n'}\left|x_{j,n}'-\frac j{d_n'}\right|.
\end{equation}
Recall that, by assumption, $\mathcal G_n$ is a.u.\ in $[0,1]$, i.e.,
\[ \lim_{n\to\infty}m(\mathcal G_n)=0,\qquad m(\mathcal G_n)=\max_{i=1,\ldots,d_n}\left|x_{i,n}-\frac i{d_n}\right|. \]
We consider two cases.
\begin{enumerate}[nolistsep,leftmargin=14pt]
	\item[$\blacktriangleright$] $j\in J_n$. In this case, $j=j_k$ for some $k=1,\ldots,a_n$ and $x_{j,n}'=x_{j_k,n}'=x_{i_k,n}$. Thus,
	\begin{align}\label{j=jk}
	\notag\left|x_{j,n}'-\frac j{d_n'}\right|&=\left|x_{i_k,n}-\frac{j_k}{d_n'}\right|\le\left|x_{i_k,n}-\frac{i_k}{d_n}\right|+\left|\frac{i_k}{d_n}-\frac{i_k}{d_n'}\right|+\left|\frac{i_k}{d_n'}-\frac{j_k}{d_n'}\right|=\left|x_{i_k,n}-\frac{i_k}{d_n}\right|+i_k\frac{|d_n'-d_n|}{d_nd_n'}+\frac{|i_k-j_k|}{d_n'}\\
	&\le m(\mathcal G_n)+2\frac{b_n}{d_n'}.
	\end{align}
	\item[$\blacktriangleright$] $j\not\in J_n$. In this case, let $j_k,j_{k+1}\in J_n$ be the two indices in $J_n$ surrounding $j$ as in Figure~\ref{griglie_au_bis}. Note that $j_k$ may not exist if $j$ is ``too close'' to the left boundary (as it would happen in Figure~\ref{griglie_au_bis} if $j$ were equal to $1$, $2$ or $3$); in this situation, we have $j_{k+1}=j_1$ and we set by convention $j_k=j_0=0$, $i_k=i_0=0$, $x_{0,n}'=x_{0,n}=0$. Similarly, if $j_{k+1}$ does not exist (this may happen if $j$ is ``too close'' to the right boundary), 
	then $j_k=j_{a_n}$ and we set by convention $j_{k+1}=j_{a_n+1}=d_n'+1$, $i_{k+1}=i_{a_n+1}=d_n+1$, $x_{d_n'+1,n}'=x_{d_n+1,n}=1$. With these definitions and conventions, 
	we have the following.
	\begin{itemize}[nolistsep,leftmargin=*]
		\item $j_k<j<j_{k+1}$.
		\item $x_{i_k,n}=x_{j_k,n}'\le x_{j_{k+1},n}'=x_{i_{k+1},n}$.
		\item $x_{j_k,n}'-\epsilon_n\le x_{j,n}'\le x_{j_{k+1},n}'+\epsilon_n$ (recall that $\mathcal G_n'\subset[-\epsilon_n,1+\epsilon_n]$ by assumption).
		\item $|j-j_k|\le b_n$. Indeed, looking at Figure~\ref{griglie_au_bis} and keeping in mind our conventions for the left and right boundaries, we have
		\begin{align*}
		|j-j_k|&\le j_{k+1}-j_k-1=\mbox{number of indices strictly between $j_k$ and $j_{k+1}$}\\
		&\le|\mathcal G_n'\setminus\mathcal G_n|\le|\mathcal G_n'\triangle\mathcal G_n|=b_n.
		\end{align*}
		\item $|i_{k+1}-i_k|\le b_n+1$. Indeed, looking at Figure~\ref{griglie_au_bis} and keeping in mind our conventions for the left and right boundaries, we have
		\begin{align*}
		|i_{k+1}-i_k|&=i_{k+1}-i_k=\mbox{number of indices between $i_k$ and $i_{k+1}$ (including $i_{k+1}$)}\\
		&\le|\mathcal G_n\setminus\mathcal G_n'|+1\le|\mathcal G_n\triangle\mathcal G_n'|+1=b_n+1.
		\end{align*}
		\item $\displaystyle\left|x_{j_k,n}'-\frac{j_k}{d_n'}\right|\le m(\mathcal G_n)+2\frac{b_n}{d_n'}$ by \eqref{j=jk} (if $j_k\in J_n$) or by direct verification (if $j_k=0$). \vspace{5pt}
		\item $\displaystyle\left|x_{i_k,n}-\frac{i_k}{d_n}\right|\le m(\mathcal G_n)$ by definition of $m(\mathcal G_n)$ (if $i_k\in I_n$) or by direct verification (if $i_k=0$). \vspace{5pt}
		\item $\displaystyle\left|x_{i_{k+1},n}-\frac{i_{k+1}}{d_n}\right|\le\max\left(m(\mathcal G_n),\left|1-\frac{d_n+1}{d_n}\right|\right)=\max\left(m(\mathcal G_n),\frac{1}{d_n}\right)$, where the quantity $1/d_n$ takes into account the boundary case $i_{k+1}=d_n+1$.
	\end{itemize}
	Thus,
	\begin{align*}
	\left|x_{j,n}'-\frac{j}{d_n'}\right|&\le|x_{j,n}'-x_{j_k,n}'|+\left|x_{j_k,n}'-\frac{j_k}{d_n'}\right|+\left|\frac{j_k}{d_n'}-\frac j{d_n'}\right|\\
	&\le x_{j_{k+1},n}'-x_{j_k,n}'+2\epsilon_n+m(\mathcal G_n)+2\frac{b_n}{d_n'}+\frac{b_n}{d_n'}\\
	&=|x_{i_{k+1},n}-x_{i_k,n}|+2\epsilon_n+m(\mathcal G_n)+3\frac{b_n}{d_n'}\\
	&\le\left|x_{i_{k+1},n}-\frac{i_{k+1}}{d_n}\right|+\left|\frac{i_{k+1}}{d_n}-\frac{i_k}{d_n}\right|+\left|\frac{i_k}{d_n}-x_{i_k,n}\right|+2\epsilon_n+m(\mathcal G_n)+3\frac{b_n}{d_n'}\\
	&\le\max\left(m(\mathcal G_n),\frac{1}{d_n}\right)+\frac{b_n+1}{d_n}+m(\mathcal G_n)+2\epsilon_n+m(\mathcal G_n)+3\frac{b_n}{d_n'}\\
	&\le3m(\mathcal G_n)+\frac{b_n+2}{d_n}+2\epsilon_n+3\frac{b_n}{d_n'}.
	\end{align*}
\end{enumerate}
\begin{figure}
\centering
\includegraphics[width=\textwidth]{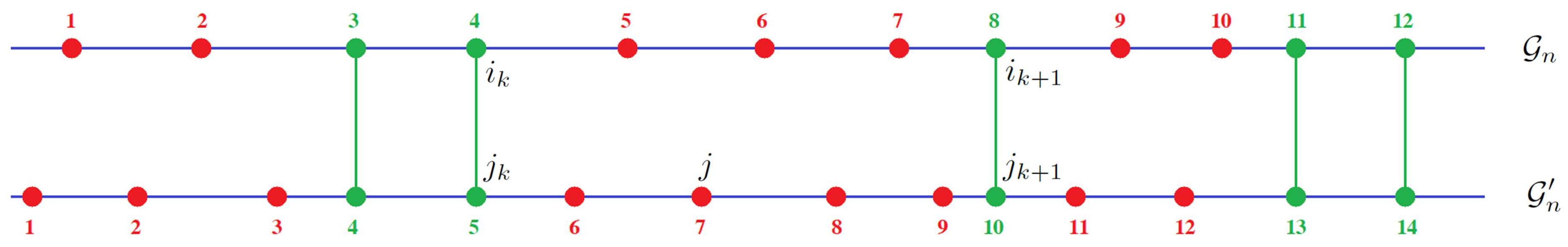}
\caption{Illustration for the proof of Lemma~\ref{fonda.u.}.}
\label{griglie_au_bis}
\end{figure}
In conclusion, by combining the two considered cases, for all $j=1,\ldots,d_n'$ we have
\[ \left|x_{j,n}'-\frac{j}{d_n'}\right|\le3m(\mathcal G_n)+\frac{b_n+2}{d_n}+2\epsilon_n+3\frac{b_n}{d_n'}, \]
which is a quantity independent of $j$ and tending to $0$ as $n\to\infty$ (recall that $b_n=o(d_n)$ and $d'_n\sim d_n$ as $n\to\infty$). Thus, the thesis \eqref{max->0} follows.
\end{proof}

\subsection{Properties of continuous functions satisfying particular conditions} 

Lemma~\ref{ucf-1} highlights a property of continuous monotone functions on a compact interval. This property can be proved on the basis of the following more general results:
\begin{itemize}[leftmargin=*,nolistsep]
	\item if $f:[a,b]\to\mathbb R$ is continuous and strictly monotone, 
	then its inverse $f^{-1}:f([a,b])\to[a,b]$ is continuous and strictly monotone; 
	\item if $f:[a,b]\to\mathbb R$ is continuous and monotone,  
	then $f$ is uniformly continuous. 
\end{itemize}
However, for the reader's convenience, we include a direct proof of Lemma~\ref{ucf-1}.

\begin{lemma}\label{ucf-1}
Let $f:[a,b]\to\mathbb R$ be continuous and strictly monotone on $[a,b]$. 
Then, for every $\delta>0$ there exists $\epsilon>0$ such that
\begin{equation*}
[f(x)-\epsilon,f(x)+\epsilon]\subseteq f([x-\delta,x+\delta]),\qquad\forall\,x\in[a+\delta,b-\delta].
\end{equation*}
\end{lemma}
\begin{proof}
We prove the lemma under the assumption that $f$ is strictly increasing on $[a,b]$; the proof in the case where $f$ is strictly decreasing on $[a,b]$ is completely analogous. Fix $\delta>0$. Since $f$ is continuous and strictly increasing on $[a,b]$, the function $f^+_\delta(x) = f(x+\delta)-f(x)$ is continuous and strictly positive on $[a,b-\delta]$. 
As a consequence, $f^+_\delta(x)$ has a strictly positive minimum $\epsilon^+>0$ on $[a,b-\delta]$: 
\begin{equation*}
f(x+\delta)-f(x)\ge\epsilon^+,\qquad\forall\,x\in[a,b-\delta].
\end{equation*}
Similarly, the function $f^-_\delta(x) = f(x) - f(x-\delta)$ is continuous and strictly positive on $[a+\delta,b]$ and so it has a strictly positive minimum $\epsilon^->0$ on $[a+\delta,b]$: 
\begin{equation*}
f(x)-f(x-\delta)\ge\epsilon^-,\qquad\forall\,x\in[a+\delta,b].
\end{equation*}
If we set $\epsilon = \min(\epsilon^+,\epsilon^-)>0$, then we see that
\[ \left\{\begin{aligned}
&f(x+\delta)-f(x)\ge\epsilon,\\
&f(x)-f(x-\delta)\ge\epsilon,
\end{aligned}\right.\qquad\forall\,x\in[a+\delta,b-\delta], \]
which is equivalent to
\[ \left\{\begin{aligned}
&f(x-\delta)\le f(x)-\epsilon,\\
&f(x)+\epsilon\le f(x+\delta),
\end{aligned}\right.\qquad\forall\,x\in[a+\delta,b-\delta]. \]
Thus, for every $x\in[a+\delta,b-\delta]$ we have $[f(x)-\epsilon,f(x)+\epsilon]\subseteq[f(x-\delta),f(x+\delta)]=f([x-\delta,x+\delta])$, where the latter equality follows from the fact that $f$ is continuous and monotone increasing.
\end{proof}

Lemma~\ref{f-1(l)finito} shows that the preimage of any point through a continuous function $f:[a,b]\to\mathbb R$ having a finite number of local maximum/minimum points is finite.

\begin{lemma}\label{f-1(l)finito}
Let $I$ be a bounded real interval and let $f:I\to\mathbb R$ be continuous on $I$ with a finite number of local maximum points and local minimum points. Then, for every $\lambda\in\mathbb R$, the set $f^{-1}(\lambda)=\{x\in I:f(x)=\lambda\}$ is finite.
\end{lemma}
\begin{proof}
For a given $\lambda\in\mathbb R$, suppose that $\xi,\eta$ are any two distinct points in $f^{-1}(\lambda)$ with $\xi<\eta$. The function $f:[\xi,\eta]\to\mathbb R$ is continuous by assumption, is not constant on $[\xi,\eta]$ by the assumption that it has a finite number of local maximum/minimum points, and satisfies $f(\xi)=f(\eta)=\lambda$. Thus, at least one local extremum point of $f$ (either the absolute maximum or the absolute minimum of $f$ on $[\xi,\eta]$) lies in the open interval $(\xi,\eta)$.

Now, suppose by contradiction that $f^{-1}(\lambda)$ is infinite for a certain $\lambda\in\mathbb R$. Since $f^{-1}(\lambda)$ is contained in the compact interval $\overline I$, it must have an accumulation point $\alpha\in\overline I$. Hence, we can find a strictly monotone sequence $\{\xi_i\}_i\subseteq f^{-1}(\lambda)$ such that $\xi_i\to\alpha$. 
In each interval between two consecutive points $\xi_i$ and $\xi_{i+1}$ we have at least one local extremum point of $f$ by the reasoning at the beginning of the proof. We conclude that $f$ has infinitely many local extremum points, which is a contradiction to the hypothesis.
\end{proof}

\begin{corollary}\label{f-1(l)finito-d}
Let $I$ be a bounded real interval and let $f:I\to\mathbb R$ be continuous on $I$ with a finite number of local maximum points, local minimum points, and discontinuity points. Then, for every $\lambda\in\mathbb R$, the set $f^{-1}(\lambda)=\{x\in I:f(x)=\lambda\}$ is finite.
\end{corollary}
\begin{proof}
Let $a<b$ be the endpoints of the interval $I$ and let $a=\eta_0<\eta_1<\ldots<\eta_\ell<\eta_{\ell+1}=b$ be the discontinuity points of $f$, to which we also add the boundary points $\eta_0=a$ and $\eta_{\ell+1}=b$. Note that the restriction $f:(\eta_j,\eta_{j+1})\to\mathbb R$ is continuous and satisfies the assumptions of Lemma~\ref{f-1(l)finito} for every $j=0,\ldots,\ell$. Hence, for every $\lambda\in\mathbb R$, the set
\[ f^{-1}(\lambda)=\{x\in I:f(x)=\lambda\}\subseteq\bigcup_{j=1}^\ell\{x\in(\eta_j,\eta_{j+1}):f(x)=\lambda\}\cup\{\eta_0,\eta_1,\ldots,\eta_{\ell+1}\} \]
is finite by Lemma~\ref{f-1(l)finito}.
\end{proof}

\subsection{Proof of Theorem~\ref{thm:mmd}}

The last results we need to prove Theorem~\ref{thm:mmd} are the following two technical lemmas.
Lemma~\ref{conteggio_griglia} provides a straightforward estimate of the largest number of points taken from a uniform grid that can lie in a fixed interval; the proof is left to the reader.
Lemma~\ref{Ex3.3} is a slight generalization of \cite[Exercise~3.3]{GLTbookI}.

\begin{lemma}\label{conteggio_griglia}
Let $h>0$ and let $\{\vartheta_{i,h}\}_{i\in\mathbb Z}$ be a uniform grid in $\mathbb R$ with stepsize $h$, say $\vartheta_{i,h}=x_0+ih$ with $x_0\in\mathbb R$ and $i\in\mathbb Z$. 
Then, for any interval $[\alpha,\beta]\subset\mathbb R$, we have
\begin{equation*}
|\{i\in\mathbb Z:\vartheta_{i,h}\in[\alpha,\beta]\}|\le\left\lfloor(\beta-\alpha)/h\right\rfloor+1.
\end{equation*}
\end{lemma}

\begin{lemma}\label{Ex3.3}
For every $\epsilon>0$, let $\{q_n(\epsilon)\}_n$ be a sequence of numbers such that $q_n(\epsilon)\to q(\epsilon)$ as $n\to\infty$ and $q(\epsilon)\to 0$ as $\epsilon\to0$. Then, there exists a sequence of positive numbers $\{\epsilon_n\}_n$ such that $\epsilon_n\to0$ and $q_n(\epsilon_n)\to 0$.
\end{lemma}
\begin{proof}
Since $q_n(\epsilon)\to q(\epsilon)$ for every $\epsilon>0$,
\begin{itemize}[leftmargin=*,nolistsep]
	\item for $\epsilon=1$ there exists $n_1$ such that $|q_n(1)-q(1)|\le1$ for $n\ge n_1$, \vspace{3pt}
	\item for $\epsilon=\frac12$ there exists $n_2>n_1$ such that $|q_n(\frac12)-q(\frac12)|\le\frac12$ for $n\ge n_2$, \vspace{3pt}
	\item for $\epsilon=\frac13$ there exists $n_3>n_2$ such that $|q_n(\frac13)-q(\frac13)|\le\frac13$ for $n\ge n_3$, \vspace{3pt}
	\item $\ldots$
\end{itemize}
Define
\begin{itemize}[leftmargin=*,nolistsep]
	\item $\epsilon_n=1$ for $n<n_2$, \vspace{3pt}
	\item $\epsilon_n=\frac12$ for $n_2\le n<n_3$, \vspace{3pt}
	\item $\epsilon_n=\frac13$ for $n_3\le n<n_4$, \vspace{3pt}
	\item $\ldots$
\end{itemize}
By construction, $\epsilon_n\to0$ and $|q_n(\epsilon_n)-q(\epsilon_n)|\le\epsilon_n$ for $n\ge n_2$, so $|q_n(\epsilon_n)|\le|q(\epsilon_n)|+\epsilon_n$ for $n\ge n_2$ and $q_n(\epsilon_n)\to0$.
\end{proof}

\begin{proof}[Proof of Theorem~{\rm\ref{thm:mmd}}]
Let $\theta_{i,n}=a+i(b-a)/d_n$, $i=1,\ldots,d_n$. It is clear that the grid $\{\theta_{i,n}\}_{i=1,\ldots,d_n}\subset[a,b]$ is a.u.\ in $[a,b]$.
Hence, by Corollary~\ref{eccoGi-n}, for every $n$ there exists a permutation $\tau_n$ of $\{1,\ldots,d_n\}$ such that
\begin{equation}\label{tal_eq}
\max_{i=1,\ldots,d_n}|f(\theta_{i,n})-\lambda_{\tau_n(i),n}|=\max_{i=1,\ldots,d_n}|f(\theta_{i,n})-\mu_{i,n}|=\epsilon_n\to0,
\end{equation}
where for simplicity we have set $\mu_{i,n}=\lambda_{\tau_n(i),n}$ for all $i=1,\ldots,d_n$.
Moreover, $\Lambda_n\subseteq f([a,b])$ by hypothesis. This implies that, for every $n$ and every $i=1,\ldots,d_n$, the set $f^{-1}(\mu_{i,n})$ is finite (by Corollary~\ref{f-1(l)finito-d}) and non-empty. Thus, we can define the grid $\mathcal G_n=\{x_{i,n}\}_{i=1,\ldots,d_n}\subset[a,b]$ such that, for every $n$ and every $i=1,\ldots,d_n$, the point $x_{i,n}$ is chosen as one of the closest point to $\theta_{i,n}$ in $f^{-1}(\mu_{i,n})$, i.e.,
\begin{equation}\label{xin-def}
f(x_{i,n})=\mu_{i,n},\qquad |x_{i,n}-\theta_{i,n}|=\min_{x\in f^{-1}(\mu_{i,n})}|x-\theta_{i,n}|=\min_{\substack{x\in[a,b]:\\f(x)=\mu_{i,n}}}|x-\theta_{i,n}|.
\end{equation}
We show that $\mathcal G_n$ is a.u.\ in $[a,b]$ (provided that we arrange its points in increasing order). Once this is done, the theorem is proved.

For every $\delta,\epsilon\ge0$ and every $n$, we define the ``bad'' sets
\begin{align*}
E_{\delta,\epsilon}&=\left\{x\in[a+\delta,b-\delta]:[f(x)-\epsilon,f(x)+\epsilon]\not\subseteq f([x-\delta,x+\delta])\right\}\cup[a,a+\delta)\cup(b-\delta,b],\\
\mathcal E_{\delta,n}&=\{i\in\{1,\ldots,d_n\}:\theta_{i,n}\in E_{\delta,\epsilon_n}\}.
\end{align*}
We call them ``bad'' sets, because if $i\not\in\mathcal E_{\delta,n}$, i.e., $\theta_{i,n}\not\in E_{\delta,\epsilon_n}$, then ``things are fine'' in the sense that $|x_{i,n}-\theta_{i,n}|\le\delta$. In formulas, for every $\delta>0$, every $n$ and every $i=1,\ldots,d_n$, we have
\begin{equation}\label{result}
i\in(\mathcal E_{\delta,n})^c\quad\iff\quad\theta_{i,n}\in(E_{\delta,\epsilon_n})^c\quad\implies\quad|x_{i,n}-\theta_{i,n}|\le\delta,
\end{equation}
where $(\mathcal E_{\delta,n})^c$ is the complement of $\mathcal E_{\delta,n}$ in $\{1,\ldots,d_n\}$ and $(E_{\delta,\epsilon_n})^c$ is the complement of $E_{\delta,\epsilon_n}$ in $[a,b]$. To prove \eqref{result}, suppose that $\theta_{i,n}\in(E_{\delta,\epsilon_n})^c$. Then, by definition of $E_{\delta,\epsilon_n}$, we have $\theta_{i,n}\in[a+\delta,b-\delta]$ and $[f(\theta_{i,n})-\epsilon_n,f(\theta_{i,n})+\epsilon_n]\subseteq f([\theta_{i,n}-\delta,\theta_{i,n}+\delta])$. Since $\mu_{i,n}\in[f(\theta_{i,n})-\epsilon_n,f(\theta_{i,n})+\epsilon_n]$ by \eqref{tal_eq}, we infer that $\mu_{i,n}\in f([\theta_{i,n}-\delta,\theta_{i,n}+\delta])$. Hence, there exists $y_{i,n}\in[\theta_{i,n}-\delta,\theta_{i,n}+\delta]$ such that $f(y_{i,n})=\mu_{i,n}$. But then we have $y_{i,n}\in f^{-1}(\mu_{i,n})$ and $|y_{i,n}-\theta_{i,n}|\le\delta$, which implies $|x_{i,n}-\theta_{i,n}|\le|y_{i,n}-\theta_{i,n}|\le\delta$ by our choice of $x_{i,n}$ as one of the closest point to $\theta_{i,n}$ in $f^{-1}(\mu_{i,n})$; see \eqref{xin-def}. This concludes the proof of \eqref{result}.

Now, let $a=\xi_0<\xi_1<\ldots<\xi_k<\xi_{k+1}=b$ be the local maximum points, local minimum points, and discontinuity points of $f$, to which we also add the boundary points $\xi_0=a$ and $\xi_{k+1}=b$. For every $j=0,\ldots,k$, the function $f$ is continuous on $(\xi_j,\xi_{j+1})$ and has no local maximum/minimum points on $(\xi_j,\xi_{j+1})$, so it is strictly monotone on $(\xi_j,\xi_{j+1})$. Thus, by Lemma~\ref{ucf-1} applied to $f:[\xi_j+\delta/2,\xi_{j+1}-\delta/2]\to\mathbb R$, for every $j=0,\ldots,k$ and every $\delta>0$ there exists $\epsilon^{(j,\delta)}>0$ such that
\[ [f(x)-\epsilon^{(j,\delta)},f(x)+\epsilon^{(j,\delta)}]\subseteq f([x-\delta/2,x+\delta/2])\subseteq f([x-\delta,x+\delta]),\qquad\forall\,x\in[\xi_j+\delta,\xi_{j+1}-\delta]. \]
Hence, for every $\delta>0$ there exists $\epsilon^{(\delta)}=\min_{j=0,\ldots,k}\epsilon^{(j,\delta)}>0$ such that
\begin{equation}\label{eps_delta-cup}
[f(x)-\epsilon^{(\delta)},f(x)+\epsilon^{(\delta)}]\subseteq f([x-\delta,x+\delta]),\qquad\forall\,x\in\bigcup_{j=0}^k[\xi_j+\delta,\xi_{j+1}-\delta].
\end{equation}

For every $\delta>0$, let $n_\delta$ be such that $\epsilon_n\le\epsilon^{(\delta)}$ for $n\ge n_\delta$. If $n\ge n_\delta$ and $i\in\{1,\ldots,d_n\}$ is an index such that $\theta_{i,n}\in\bigcup_{j=0}^k[\xi_j+\delta,\xi_{j+1}-\delta]$, then in particular $\theta_{i,n}\in[a+\delta,b-\delta]$ and, by \eqref{eps_delta-cup},
\[ [f(\theta_{i,n})-\epsilon_n,f(\theta_{i,n})+\epsilon_n]\subseteq[f(\theta_{i,n})-\epsilon^{(\delta)},f(\theta_{i,n})+\epsilon^{(\delta)}]\subseteq f([\theta_{i,n}-\delta,\theta_{i,n}+\delta]). \]
Hence, $\theta_{i,n}\not\in E_{\delta,\epsilon_n}$, i.e., $i\not\in\mathcal E_{\delta,n}$. In follows that, for every $\delta>0$ and every $n\ge n_\delta$,
\begin{align*}
|\mathcal E_{\delta,n}|&=|\{i\in\{1,\ldots,d_n\}:\theta_{i,n}\in E_{\delta,\epsilon_n}\}|\le\biggl|\biggl\{i\in\{1,\ldots,d_n\}:\theta_{i,n}\in\bigcup_{j=0}^{k+1}[\xi_j-\delta,\xi_j+\delta]\biggr\}\biggr|\le(k+2)\left(2\delta\,\frac{d_n}{b-a}+1\right),
\end{align*}
where the latter inequality is due to Lemma~\ref{conteggio_griglia}. We can therefore choose, by Lemma~\ref{Ex3.3}, a sequence of positive numbers $\{\delta_n\}_n$ such that $\delta_n\to0$ and $|\mathcal E_{\delta_n,n}|/d_n\to0$.

To conclude the proof, let $\mathcal G'_n=\{x'_{i,n}\}_{i=1,\ldots,d_n}$ be the sequence of $d_n$ points defined as follows:
\[ x'_{i,n}=\left\{\begin{aligned}&x_{i,n}, &&\mbox{if }\,\theta_{i,n}\in(E_{\delta_n,\epsilon_n})^c,\\ &\theta_{i,n}, &&\mbox{if }\,\theta_{i,n}\in E_{\delta_n,\epsilon_n}.\end{aligned}\right. \]
$\mathcal G'_n\subset[a,b]$ and $\mathcal G'_n$ is a.u.\ in $[a,b]$ because its distance from the a.u.\ grid $\{\theta_{i,n}\}_{i=1,\ldots,d_n}$ is uniformly bounded by $\delta_n\to0$. Indeed, by \eqref{result},
\[ \max_{i=1,\ldots,d_n}|x'_{i,n}-\theta_{i,n}|=\max_{\substack{i\in\{1,\ldots,d_n\}:\\\theta_{i,n}\in(E_{\delta_n,\epsilon_n})^c}}|x_{i,n}-\theta_{i,n}|\le\delta_n. \]
The grid $\mathcal G'_n$ differs from the original grid $\mathcal G_n$ by at most $2|\mathcal E_{\delta_n,\epsilon_n}|=o(d_n)$ elements, in the sense that
$|\mathcal G'_n\triangle\mathcal G_n|\le2|\mathcal E_{\delta_n,\epsilon_n}|$.
Thus, by Lemma~\ref{fonda.u.}, $\mathcal G_n$ is a.u.\ in $[a,b]$ (provided that we arrange its points in increasing order).
\end{proof}

\subsection{Concatenation lemma}
The following lemma is a plain consequence of Definition~\ref{dd} and has often been used in the literature, but lucid statement and proof have never been provided.
We therefore provide the details below.

\begin{lemma}[\textbf{concatenation lemma}]\label{lem:concat}
Let $\{A_n\}_n$ be a matrix-sequence, 
let $f:[a,b]\to\mathbb C^{r\times r}$ be measurable, and suppose that $\{A_n\}_n\sim_\lambda f$. 
Let $\lambda_1(f),\ldots,\lambda_r(f):[a,b]\to\mathbb C$ be $r$ measurable functions such that $\lambda_1(f(x)),\ldots,\lambda_r(f(x))$ are the eigenvalues of $f(x)$ for every $x\in[a,b]$. Then $\{A_n\}_n\sim_\lambda\tilde f$, where $\tilde f$ is the concatenation of resized versions of $\lambda_1(f),\ldots,\lambda_r(f)$ given by
\begin{equation*}
\tilde f:[0,1]\to\mathbb C,\qquad \tilde f(y)=\left\{\begin{aligned}
&\textstyle{\lambda_1(f(a+(b-a)ry)),} &&\textstyle{0\le y<\frac1r,}\\
&\textstyle{\lambda_2(f(a+(b-a)(ry-1))),} &&\textstyle{\frac1r\le y<\frac2r,}\\
&\textstyle{\lambda_3(f(a+(b-a)(ry-2))),} &&\textstyle{\frac2r\le y<\frac3r,}\\
&\qquad\vdots && \qquad\vdots\\
&\textstyle{\lambda_r(f(a+(b-a)(ry-r+1))),} &&\textstyle{\frac{r-1}r\le y\le1.}\\
\end{aligned}\right.
\end{equation*}
\end{lemma}
\begin{proof}
The result follows from Definition~\ref{dd}, after observing that, for every $F\in C_c(\mathbb C)$,
\begin{align*}
\int_0^1F(\tilde f(y)){\rm d}y&=\sum_{i=1}^r\int_{(i-1)/r}^{i/r}F(\lambda_i(f(a+(b-a)(ry-i+1)))){\rm d}y\\
&=\sum_{i=1}^r\frac{1}{(b-a)r}\int_a^b F(\lambda_i(f(x))){\rm d}x=\frac1{b-a}\int_a^b\frac{\sum_{i=1}^rF(\lambda_i(f(x)))}r{\rm d}x,
\end{align*}
where in the second equality we have used the change of variable formula for the Lebesgue integral. 
\end{proof}

\subsection{Restriction operator and asymptotic spectral distribution of restricted matrix-sequences}
For every $n\ge1$, let $\Xi_n$ be the uniform grid in $[0,1]$ given by
$\Xi_n=\{\frac i{n+1}:i=1,\ldots,n\}$.
If $E\subseteq\mathbb R$, we define $d_n^E$ as the number of points of $\Xi_n$ inside $E$, i.e., $d_n^E=|\Xi_n\cap E|$.
If $A$ is a square matrix of size~$n$ and $E\subseteq\mathbb R$, we define $R_E(A)$ as the principal submatrix of $A$ of size $d_n^E$ obtained from $A$ by selecting the rows and columns corresponding to indices $i\in\{1,\ldots,n\}$ such that $\frac i{n+1}\in E$. 
For the proof of the next lemma, see \cite[Lemma~4.9]{barbarinoREDUCED}.

\begin{lemma}\label{lem:different_grids}
Let $\Omega\subseteq[0,1]$ be a regular set with $\mu_1(\Omega)>0$, let $\{\Gamma_n\}_n$ be a sequence of measurable sets contained in $[0,1]$, and suppose that
$d_{d_n}^{\Omega\triangle\Gamma_n}\to\infty$ and $d_{d_n}^{\Omega\triangle\Gamma_n}=o(d_n)$.
Then, for every matrix-sequence $\{A_n\}_n$ formed by Hermitian matrices with $A_n$ of size $d_n$, we have the equivalence
\[ \{R_\Omega(A_n)\}_n\sim_\lambda f\quad\iff\quad\{R_{\Gamma_n}(A_n)\}_n\sim_\lambda f. \]
\end{lemma}

\subsection{GLT sequences}
Let $\{A_n\}_n$ be a matrix-sequence and let $\kappa:[0,1]\times[-\pi,\pi]\to\mathbb C$ be a measurable function. We say that $\{A_n\}_n$ is a GLT sequence with symbol $\kappa$, and we write $\{A_n\}_n\sim_{\rm GLT}\kappa$, if the pair $(\{A_n\}_n,\kappa)$ satisfies some special properties that are not reported here as they are difficult to formulate.
The interested reader is referred to \cite{SbMath} and \cite[Chapter~8]{GLTbookI} for details. Here, we only collect the properties of GLT sequences that we need in the proof of Theorem~\ref{eccoGi'}.
The first property is reported in the next lemma \cite[Lemma~5.1]{barbarinoREDUCED}.

\begin{lemma}\label{lem:reduced_hermitian_GLT}
Let $\{A_n\}_n$ be a matrix-sequence formed by Hermitian matrices, 
let $\kappa:[0,1]\times[-\pi,\pi]\to\mathbb C$ be measurable, and suppose that $\{A_n\}_n\sim_{\rm GLT}\kappa$. Then,
$\{R_\Omega(A_n)\}_n\sim_\lambda\kappa|_{\Omega\times[-\pi,\pi]}$
for every regular set $\Omega\subseteq[0,1]$.
\end{lemma}

The second property is reported in the next lemma \cite[Theorem~2]{sm}. 

\begin{lemma}\label{lem:main_spectral}
Let $g:[0,1]\to\mathbb C$ be measurable and let $\{D_n\}_n$ be a matrix-sequence formed by diagonal matrices such that $\{D_n\}_n\sim_\lambda g$. Then, there exists a matrix-sequence formed by permutation matrices $\{P_n\}_n$ such that $\{P_nD_nP_n^T\}_n\sim_{\rm GLT}\kappa(x,\theta)=g(x)$.
\end{lemma}

The third property is reported in the next lemma, which has never appeared in the literature.

\begin{lemma}[\textbf{splitting of GLT sequences formed by diagonal matrices}]\label{lem:reduced_GLT_split}
Let $g:[0,1]\to\mathbb C$ be measurable and let $\{\Lambda_n=\{\lambda_{1,n},\ldots,\lambda_{d_n,n}\}\}_n$ be a sequence of finite multisets of real numbers such that $d_n\to\infty$ as $n\to\infty$. Assume the following.
\begin{itemize}[nolistsep,leftmargin=*]
	\item $\{D_n\}_n\sim_{\rm GLT}\kappa(x,\theta)=g(x)$ with $D_n={\rm diag}(\lambda_{1,n},\ldots,\lambda_{d_n,n})$.
	\item $\{L_{n,1}\}_n,\ldots,\{L_{n,k}\}_n$ are sequences of natural numbers such that $L_{n,1}+\ldots+L_{n,k}=d_n$ for every $n$ and $L_{n,j}/d_n\to1/k$ as $n\to\infty$ for every $j=1,\ldots,k$.
\end{itemize}
Then
\[ \left\{\mathop{\rm diag}_{i=1,\ldots,L_{n,j}}(\lambda_{L_{n,1}+\ldots+L_{n,j-1}+i,n})\right\}_n\sim_\lambda g|_{[(j-1)/k,j/k]},\qquad j=1,\ldots,k. \]
\end{lemma}
\begin{proof}
Let $\Omega_j=[(j-1)/k,j/k]$ for $j=1,\ldots,k$, and fix $j\in\{1,\ldots,k\}$. Since $\{D_n\}_n\sim_{\rm GLT}\kappa(x,\theta)=g(x)$ and the matrices $D_n$ are Hermitian, by Lemma~\ref{lem:reduced_hermitian_GLT} we have $\{R_{\Omega_j}(D_n)\}_n \sim_\lambda \kappa|_{\Omega_j\times[-\pi,\pi]}$, which is equivalent to
\begin{equation}\label{iff_R}
\{R_{\Omega_j}(D_n)\}_n\sim_\lambda g|_{\Omega_j}.
\end{equation}
Let
\[ \Gamma_{n,j}= \left\{\frac{L_{n,1} +\ldots + L_{n,j-1} + i}{d_n+1}: i=1,\ldots,L_{n,j}\right\}. \]
For every $n$, we have
{\allowdisplaybreaks\begin{align*}
d_{d_n}^{\Omega_j\triangle\Gamma_{n,j}} &= \left|\left\{i\in\{1,\ldots,d_n\}:\frac{i}{d_n+1}\in\Omega_j\triangle\Gamma_{n,j}\right\}\right| \\
&=\left|\left\{i\in\{1,\ldots,d_n\}:\frac{i}{d_n+1}\in\Omega_j\setminus\Gamma_{n,j}\right\}\right|+\left|\left\{i\in\{1,\ldots,d_n\}: \frac{i}{d_n+1} \in \Gamma_{n,j}\setminus\Omega_j\right\}\right|\\
&=\left|\left\{i\in\{1,\ldots,d_n\}:\frac{i}{d_n+1}\in\Omega_j,\ \frac{i}{d_n+1}\le\frac{L_{n,1}+\ldots+L_{n,j-1}}{d_n+1}\right\}\right|\\*
&\qquad+\left|\left\{i\in\{1,\ldots,d_n\}:\frac{i}{d_n+1}\in\Omega_j,\ \frac{i}{d_n+1}>\frac{L_{n,1}+\ldots+L_{n,j}}{d_n+1}\right\}\right| \\*
&\qquad+\left|\left\{i\in\{1,\ldots,d_n\}:\frac{i}{d_n+1}\in\Gamma_{n,j},\ \frac{i}{d_n+1}<\frac{j-1}k\right\}\right|\\*
&\qquad+\left|\left\{i\in\{1,\ldots,d_n\}:\frac{i}{d_n+1}\in\Gamma_{n,j},\ \frac{i}{d_n+1}>\frac jk\right\}\right|\\
&\le\color{blue}\left|\left\{i\in\{1,\ldots,d_n\}:\frac{j-1}k\le\frac{i}{d_n+1}\le\frac{L_{n,1}+\ldots+L_{n,j-1}}{d_n+1}\right\}\right|\\*
&\qquad+\left|\left\{i\in\{1,\ldots,d_n\}:\frac{L_{n,1}+\ldots+L_{n,j}}{d_n+1}<\frac{i}{d_n+1}\le\frac jk\right\}\right|\\*
&\qquad\color{blue}+\left|\left\{i\in\{1,\ldots,d_n\}:\frac{L_{n,1}+\ldots+L_{n,j-1}}{d_n+1}<\frac{i}{d_n+1}<\frac{j-1}k\right\}\right|\\*
&\qquad+\left|\left\{i\in\{1,\ldots,d_n\}:\frac{j}k<\frac{i}{d_n+1}\le\frac{L_{n,1}+\ldots+L_{n,j}}{d_n+1}\right\}\right|\\
&\le\color{blue}\left|\left\{i\in\{1,\ldots,d_n\}:\min\left(\frac{j-1}k,\frac{L_{n,1}+\ldots+L_{n,j-1}}{d_n+1}\right)\le\frac{i}{d_n+1}\le\max\left(\frac{j-1}k,\frac{L_{n,1}+\ldots+L_{n,j-1}}{d_n+1}\right)\right\}\right|\\
&\qquad+\left|\left\{i\in\{1,\ldots,d_n\}:\min\left(\frac{j}k,\frac{L_{n,1}+\ldots+L_{n,j}}{d_n+1}\right)\le\frac{i}{d_n+1}\le\max\left(\frac{j}k,\frac{L_{n,1}+\ldots+L_{n,j}}{d_n+1}\right)\right\}\right|\\
&\le\left\lfloor\,\left|\frac{(d_n+1)(j-1)}{k}-L_{n,1}-\ldots-L_{n,j-1}\right|\,\right\rfloor+\left\lfloor\,\left|\frac{(d_n+1)j}{k}-L_{n,1}-\ldots-L_{n,j}\right|\,\right\rfloor+2,
\end{align*}}
where the last inequality is due to Lemma~\ref{conteggio_griglia}.
Since $L_{n,i}/d_n \to 1/k$ as $n\to\infty$ for every $i=1,\ldots,k$ by assumption, we have 
\[ \frac1{d_n}\left|\frac{(d_n+1)(j-1)}{k}-L_{n,1}-\ldots-L_{n,j-1}\right|\to0,\qquad\frac1{d_n}\left|\frac{(d_n+1)j}{k}-L_{n,1}-\ldots-L_{n,j}\right|\to0, \]
and so $d_{d_n}^{\Omega_j\triangle\Gamma_{n,j}} = o(d_n)$.
Therefore, by \eqref{iff_R} and Lemma~\ref{lem:different_grids}, 
\[ \{R_{\Gamma_{n,j}}(D_n)\}_n=\left\{\mathop{\rm diag}_{i=1,\ldots,L_{n,j}}(\lambda_{L_{n,1}+\ldots+L_{n,j-1}+i,n})\right\}_n\sim_\lambda g|_{\Omega_j}. \tag*{\qedhere} \]
\end{proof}

\subsection{Proof of Lemma~\ref{ss1}}

We have now collected all the ingredients to prove Lemma~\ref{ss1}.

\begin{proof}[Proof of Lemma~{\rm\ref{ss1}}]
By Lemma~\ref{lem:concat}, the hypothesis $\{D_n\}_n\sim_\lambda f$ is equivalent to $\{D_n\}_n\sim_\lambda\tilde f$, where $\tilde f:[0,1]\to\mathbb R$ is a concatenation of resized versions of the functions $f_1,\ldots,f_k$. More precisely,
\[ \tilde f(x)=f_j(a+(b-a)(kx-j+1)),\qquad x\in\left[\frac{j-1}{k},\frac jk\right),\qquad j=1,\ldots,k,\qquad\tilde f(1)=f_k(b). \]
By Lemma~\ref{lem:main_spectral}, there exists a sequence $\{\tau_n\}_n$ such that $\tau_n$ is a permutation of $\{1,\ldots,d_n\}$ and
\[ \{\tilde D_n={\rm diag}(\lambda_{\tau_n(1),n},\ldots,\lambda_{\tau_n(d_n),n})\}_n\sim_{\rm GLT}\tilde f(x). \]
By Lemma~\ref{lem:reduced_GLT_split}, we conclude that
\[ \left\{\mathop{\rm diag}_{i=1,\ldots,L_{n,j}}(\lambda_{\tau_n(L_{n,1}+\ldots+L_{n,j-1}+i),n})\right\}_n\sim_\lambda \tilde f|_{[(j-1)/k,j/k]},\qquad j=1,\ldots,k. \]
Since
\[ k\int_{(j-1)/k}^{j/k}\tilde f|_{[(j-1)/k,j/k]}(y){\rm d}y=k\int_{(j-1)/k}^{j/k}f_j(a+(b-a)(ky-j+1)){\rm d}y=\frac1{b-a}\int_a^bf_j(x){\rm d}x \]
(this is proved by direct computation using the change of variable formula for the Lebesgue integral as in the proof of Lemma~\ref{lem:concat}),
the thesis is proved with 
\begin{align*}
D_{n,j}&=\mathop{\rm diag}_{i=1,\ldots,L_{n,j}}(\lambda_{\tau_n(L_{n,1}+\ldots+L_{n,j-1}+i),n}),\\
\Lambda_{n,j}&=\{\lambda_{\tau_n(L_{n,1}+\ldots+L_{n,j-1}+i),n}:i=1,\ldots,L_{n,j}\}. \qedhere
\end{align*}
\end{proof}

\subsection{Proof of Lemma~\ref{ss2.g}}

In order to prove Lemma~\ref{ss2.g}, we need some auxiliary results. The first result is reported in the next lemma \cite[Theorem~3.1]{GLTbookI}.

\begin{lemma}\label{s-attr}
If $\{A_n\}_n\sim_\lambda f$, then
\[ \lim_{n\to\infty}\frac{|\{i\in\{1,\ldots,d_n\}:\lambda_i(A_n)\not\in(\mathcal{ER}(f))_\epsilon\}|}{d_n}=0,\qquad\forall\,\epsilon>0, \]
where $d_n$ is the size of $A_n$.
\end{lemma}

The second result is reported in the next lemma  \cite[Theorem~3.2]{GLTbookI}.
In what follows, the notation $\{Z_n\}_n\sim_\sigma0$ means that $\{Z_n\}_n$ is a matrix-sequence with an asymptotic singular value distribution described by the identically zero function defined on any subset $\Omega$ of some $\mathbb R^d$ with $0<\mu_d(\Omega)<\infty$. Hence, regardless of $\Omega$, the notation $\{Z_n\}_n\sim_\sigma0$ means that $\frac1{d_n}\sum_{i=1}^{d_n}F(\sigma_i(Z_n))\to F(0)$ for all $F\in C_c(\mathbb R)$, where $d_n$ is the size of $Z_n$.

\begin{lemma}\label{0s}
Let $\{Z_n\}_n$ be a matrix-sequence with $Z_n$ of size $d_n$. We have $\{Z_n\}_n\sim_\sigma0$ if and only if $Z_n=R_n+N_n$ for every $n$ with
$\lim_{n\to\infty}(d_n)^{-1}{\rm rank}(R_n)=\lim_{n\to\infty}\|N_n\|=0$.
\end{lemma}

The third result is reported in the next lemma \cite[Exercise~5.3]{GLTbookI}.

\begin{lemma}\label{Ex5.3}
Let $\{X_n\}_n$ and $\{Y_n\}_n$ be matrix-sequences formed by Hermitian matrices, with $X_n$ and $Y_n$ of the same size. If $\{X_n\}_n\sim_\lambda f$ and $\{Y_n\}_n\sim_\sigma0$ then $\{X_n+Y_n\}_n\sim_\lambda f$.
\end{lemma}

The last result is the following lemma of graph theory. In what follows, given a directed graph $\mathcal G=(V,E)$ and any two nodes $i,j\in V$, a directed path from $i$ to $j$ is any sequence of nodes $i_1i_2\ldots i_q$ such that $i_1=i$, $i_q=j$, and $(i_a,i_{a+1})\in E$ for all $a=1,\ldots,q-1$. Note that a directed path from a node $i$ to itself always exists (take the sequence $i$ consisting only of the node $i$).

\begin{lemma}\label{lem:graph}
Let $X$ be a finite set, and let $A_1,\dots,A_k$ and $B_1,\dots,B_k$ be two partitions of $X$ with $|A_i| = |B_i|$ for every $i=1,\ldots,k$.
Let $\mathcal G = (V,E)$ be a directed graph on $k$ nodes $V=\{1,\ldots,k\}$ such that a directed edge $(i,j)\in E$ exists if and only if $A_i\cap B_j$ is not empty. 
If $(i,j)\in E$ then there exists a directed path from $j$ to $i$.
\end{lemma}
\begin{proof}
Suppose by contradiction that $(i,j)\in E$ but there is no directed path from $j$ to $i$. 
Then, the sets of nodes
\[ N_i=\{\text{nodes with a directed path to }i\},\qquad N^j=\{\text{nodes with a directed path from }j\} \]
are disjoint. 
Moreover, there is no edge from $N^j$ to $(N^j)^c$, hence
\begin{align*}
\sum_{x\in N^j} |A_x| &=\sum_{x\in N^j}\sum_{y\in V}|A_x\cap B_y|=\sum_{x\in N^j} \sum_{y\in N^j}|A_x\cap B_y|, \\
\sum_{y\in N^j} |B_y| &=\sum_{y\in N^j}\sum_{x\in V}|A_x\cap B_y|\ge|A_i\cap B_j|+\sum_{x\in N^j}\sum_{y\in N^j}|A_x\cap B_y|,
\end{align*}
where the last inequality follows by letting $x$ vary in $N^j\cup\{i\}$ instead of $V$.
We have thus obtained a contradiction, because $\sum_{x\in N^j}|A_x|=\sum_{y\in N^j}|B_y|$ and $|A_i\cap B_j|>0$.
\end{proof}

\begin{proof}[Proof of Lemma~{\rm\ref{ss2.g}}]
The hypotheses of Lemma~\ref{ss1} are satisfied with $f_1,\ldots,f_k,\Lambda_n$ as in the statement of Lemma~\ref{ss2.g} and with $L_{n,j}=|\tilde\Lambda_{n,j}|$ for every $n$ and every $j=1,\ldots,k$. Thus, by Lemma~\ref{ss1}, for every $n$ there exists a partition $\{\hat\Lambda_{n,1},\ldots,\hat\Lambda_{n,k}\}$ of $\Lambda_n$ such that, for every $j=1,\ldots,k$, the following properties hold.
\begin{itemize}[nolistsep,leftmargin=*]
	\item $|\hat\Lambda_{n,j}|=L_{n,j}=|\tilde\Lambda_{n,j}|$.
	\item $\{\hat\Lambda_{n,j}\}_n\sim f_j$, i.e., $\{\hat D_{n,j}\}_n\sim_\lambda f_j$, where $\hat D_{n,j}={\rm diag}(\hat\lambda_{1,n}^{(j)},\ldots,\hat\lambda_{L_{n,j},n}^{(j)})$ and $\{\hat\lambda_{1,n}^{(j)},\ldots,\hat\lambda_{L_{n,j},n}^{(j)}\}=\hat\Lambda_{n,j}$.
\end{itemize}
The partition $\{\hat\Lambda_{n,1},\ldots,\hat\Lambda_{n,k}\}$ satisfies the first two properties required in the thesis of the lemma, but it may not satisfy the third property. Through ``successive displacements'', we want to change the partition $\{\hat\Lambda_{n,1},\ldots,\hat\Lambda_{n,k}\}$ into a new partition $\{\Lambda_{n,1},\ldots,\Lambda_{n,k}\}$ that satisfies also the third property.

For every $j=1,\ldots,k$, since $\{\hat D_{n,j}\}_n\sim_\lambda f_j$, by Lemmas~\ref{s-attr} and~\ref{Ex3.3} there exists some $\delta_{n,j}$ tending to $0$ as $n\to\infty$ such that
\[ \frac{|\hat\Lambda_{n,j}\cap(\mathcal{ER}(f_j))_{\delta_{n,j}}^c|}{L_{n,j}}\to0\ \,\mbox{as}\ \,n\to\infty. \]
Note that the previous limit relation continues to hold if we replace $L_{n,j}$ with $d_n$ (because $L_{n,j}/d_n=|\tilde\Lambda_{n,j}|/d_n\to1/k$ by hypothesis) and $\delta_{n,j}$ with $\delta_n=\max(\delta_{n,1},\ldots,\delta_{n,k},\epsilon_n)$, where $\epsilon_n$ is the same as in the assumptions of the lemma. 
Thus, if we define
\begin{equation}\label{Eset.g}
\hat E_{n,j}=\hat\Lambda_{n,j}\cap(\mathcal{ER}(f_j))_{\delta_n}^c,\qquad j=1,\ldots,k,
\end{equation}
then we have the following: for every $n$ there exists some $\delta_n$ tending to $0$ as $n\to\infty$ such that $\epsilon_n\le\delta_n$ and, for every $j=1,\ldots,k$, 
\begin{equation}\label{o(d_n).g}
\frac{|\hat E_{n,j}|}{d_n}\to0\ \,\mbox{as}\ \,n\to\infty.
\end{equation}
We remark that, since $\tilde\Lambda_{n,j}\subseteq(\mathcal{ER}(f_j))_{\epsilon_n}$ by assumption, we have
\begin{equation}\label{remark.g}
\hat E_{n,j}\subseteq(\mathcal{ER}(f_j))_{\delta_n}^c\subseteq(\mathcal{ER}(f_j))_{\epsilon_n}^c\subseteq(\tilde\Lambda_{n,j})^c.
\end{equation}

Now, fix $n$ and take an element $x\in\hat E_{n,1}\cup\cdots\cup\hat E_{n,k}$. To fix ideas, suppose that $x\in\hat E_{n,1}$. 
By definition of $\hat E_{n,1}$ we have $x\in\hat\Lambda_{n,1}$, and by \eqref{remark.g} we have $x\not\in\tilde\Lambda_{n,1}$. Since $\{\tilde\Lambda_{n,1},\ldots,\tilde\Lambda_{n,k}\}$ is a partition of $\Lambda_n$ just like $\{\hat\Lambda_{n,1},\ldots,\hat\Lambda_{n,k}\}$, there exists $p\in\{1,\ldots,k\}$ with $p\ne1$ such that $x\in\tilde\Lambda_{n,p}$. Note that all hypotheses of Lemma~\ref{lem:graph} are satisfied for $X=\Lambda_n$ and the partitions $\{A_1,\ldots,A_k\}=\{\hat\Lambda_{n,1},\ldots,\hat\Lambda_{n,k}\}$ and $\{B_1,\ldots,B_k\}=\{\tilde\Lambda_{n,1},\ldots,\tilde\Lambda_{n,k}\}$, and moreover $(1,p)$ is an edge of the graph $\mathcal G$ mentioned in Lemma~\ref{lem:graph} due to the element $x\in\hat\Lambda_{n,1}\cap\tilde\Lambda_{n,p}$. Hence, by Lemma~\ref{lem:graph}, there exists in $\mathcal G$ a directed path from $p$ to $1$. This means that there exist indices
\[ i_0=1,\quad i_1=p,\quad i_2,\quad i_3,\quad \ldots,\quad i_q,\quad i_{q+1}=1 \]
(with $i_0,\ldots,i_{q+1}\in\{1,\ldots,k\}$ and $i_0,\ldots,i_q$ distinct) and corresponding elements
\[ x_0=x,\quad x_1,\quad x_2,\quad x_3,\quad \ldots,\quad x_q \]
(with $x_0,\ldots,x_q\in\Lambda_n$) such that 
\[ x_s\in\hat\Lambda_{n,i_s}\cap\tilde\Lambda_{n,i_{s+1}},\qquad s=0,\ldots,q. \]
As a consequence, we can produce a new partition $\{\overline\Lambda_{n,1},\ldots,\overline\Lambda_{n,k}\}$ of $\Lambda_n$ with the same cardinalities
\[ |\overline\Lambda_{n,j}|=L_{n,j}=|\hat\Lambda_{n,j}|,\qquad j=1,\ldots,k, \]
by removing $x_s$ from $\hat\Lambda_{n,i_s}$ and adding it to $\hat\Lambda_{n,i_{s+1}}$ for $s=1,\ldots,q$. Note that, for every $s=0,\ldots,q$,
\[ x_s\in\tilde\Lambda_{n,i_{s+1}}\subseteq(\mathcal{ER}(f_{i_{s+1}}))_{\epsilon_n}\subseteq(\mathcal{ER}(f_{i_{s+1}}))_{\delta_n}, \]
hence 
\begin{equation}\label{xsNOT.g}
x_s\not\in\hat E_{n,i_{s+1}},\qquad s=0,\ldots,q. 
\end{equation}
Therefore, if in analogy with \eqref{Eset.g} we define
\begin{equation}\label{Eset'.g}
\overline E_{n,j}=\overline\Lambda_{n,j}\cap(\mathcal{ER}(f_j))_{\delta_n}^c,\qquad j=1,\ldots,k,
\end{equation}
then we have
\begin{align*}
\overline E_{n,j}&\subseteq\hat E_{n,j},\qquad j=1,\ldots,k,\\
|\overline E_{n,1}|&=|\hat E_{n,1}|-1,
\end{align*}
where the latter equation is due to the fact that $x=x_0$ has been removed from $\hat E_{n,1}$ and has been replaced with $x_q\not\in\hat E_{n,1}$; see \eqref{xsNOT.g}.
In conclusion, starting from the original partition
\[ \{\hat\Lambda_{n,1},\ldots,\hat\Lambda_{n,k}\},\qquad\{\hat E_{n,1},\ldots,\hat E_{n,k}\} \]
we have produced a new partition
\[ \{\overline\Lambda_{n,1},\ldots,\overline\Lambda_{n,k}\},\qquad\{\overline E_{n,1},\ldots,\overline E_{n,k}\} \]
with the same cardinalities
\[ |\overline\Lambda_{n,1}|=L_{n,1}=|\hat\Lambda_{n,1}|,\quad\ldots,\quad|\overline\Lambda_{n,k}|=L_{n,k}=|\hat\Lambda_{n,k}| \]
and with
\begin{align*}
\overline E_{n,1}\cup\cdots\cup\overline E_{n,k}\subsetneqq\hat E_{n,1}\cup\cdots\cup\hat E_{n,k}.
\end{align*}
We can now repeat the same procedure for another element $x\in\overline E_{n,1}\cup\cdots\cup\overline E_{n,k}$ until all the ``$E$-sets'' are empty. At the end of the whole construction, we obtain a final partition
\[ \{\Lambda_{n,1},\ldots,\Lambda_{n,k}\},\qquad\{E_{n,1},\ldots,E_{n,k}\} \]
with the same cardinalities
\begin{equation}\label{Lambda.c.g}
|\Lambda_{n,1}|=L_{n,1}=|\hat\Lambda_{n,1}|,\quad\ldots,\quad|\Lambda_{n,k}|=L_{n,k}=|\hat\Lambda_{n,k}| 
\end{equation}
and with corresponding ``$E$-sets''
\begin{equation}\label{emptyE.g}
E_{n,1}=\ldots=E_{n,k}=\emptyset,
\end{equation}
where $E_{n,j}$ is defined in analogy with \eqref{Eset.g} and \eqref{Eset'.g} as follows:
\begin{equation}\label{Eset.f.g}
E_{n,j}=\Lambda_{n,j}\cap(\mathcal{ER}(f_j))_{\delta_n}^c,\qquad j=1,\ldots,k.
\end{equation}
We prove that the partition $\{\Lambda_{n,1},\ldots,\Lambda_{n,k}\}$ satisfies the three properties required in the thesis of the lemma.

The partition $\{\Lambda_{n,1},\ldots,\Lambda_{n,k}\}$ satisfies the first property by \eqref{Lambda.c.g} and the third property by \eqref{emptyE.g}--\eqref{Eset.f.g}. It only remains to prove that $\{\Lambda_{n,1},\ldots,\Lambda_{n,k}\}$ satisfies the second property. To this end, we note that the above procedure must be repeated at most a number $N_n$ of times equal to
\begin{equation}\label{N_n.g}
N_n=|\hat E_{n,1}\cup\cdots\cup\hat E_{n,k}|\le|\hat E_{n,1}|+\ldots+|\hat E_{n,k}|=o(d_n),
\end{equation}
because each time we apply the procedure, the union of the ``$E$-sets'' loses an element (the final equality in \eqref{N_n.g} is due to \eqref{o(d_n).g}). Moreover, each time we apply the procedure, the new partition $\{\Lambda_{n,1}^{({\rm new})},\ldots,\Lambda_{n,k}^{({\rm new})}\}$ differs from the previous partition $\{\Lambda_{n,1}^{({\rm old})},\ldots,\Lambda_{n,k}^{({\rm old})}\}$ by at most $1$ element per set, in the sense that
\[ |\Lambda_{n,j}^{({\rm new})}\setminus\Lambda_{n,j}^{({\rm old})}|\le1,\qquad j=1,\ldots,k. \]
For example, the first time we apply the procedure, we obtain
\[ |\overline\Lambda_{n,j}\setminus\hat\Lambda_{n,j}|\le1,\qquad j=1,\ldots,k. \]
So, after $N_n$ applications of the procedure, we obtain
\[ |\Lambda_{n,j}\setminus\hat\Lambda_{n,j}|\le N_n,\qquad j=1,\ldots,k. \]
Thus, for every $j=1,\ldots,k$, if we define $D_{n,j}$ as in the second property of the thesis of the lemma, the previous inequality implies that, after a suitable permutation of its diagonal elements, $D_{n,j}$ becomes equal to $\hat D_{n,j}+\Delta_{n,j}$ with $\Delta_{n,j}$ a diagonal matrix with ${\rm rank}(\Delta_{n,j})=|\Lambda_{n,j}\setminus\hat\Lambda_{n,j}|\le N_n=o(d_n)$. This implies that $\{\Delta_{n,j}\}_n\sim_\sigma0$ by Lemma~\ref{0s} and $\{D_{n,j}\}_n\sim_\lambda f_j$ by Lemma~\ref{Ex5.3}.
\end{proof}

\subsection{Proof of Theorem~\ref{eccoGi'}}

We have now collected all the ingredients to prove Theorem~\ref{eccoGi'}.

\begin{proof}[Proof of Theorem~{\rm\ref{eccoGi'}}]
The theorem follows immediately from Lemma~\ref{ss2.g} and Corollary~\ref{eccoGi-n}.
Indeed, by Lemma~\ref{ss2.g}, for every $n$ there exists a partition $\{\Lambda_{n,1},\ldots,\Lambda_{n,k}\}$ of $\Lambda_n$ 
such that, for every $j=1,\ldots,k$, the following properties hold.
\begin{itemize}[nolistsep,leftmargin=*]
	\item $|\Lambda_{n,j}|=|\tilde\Lambda_{n,j}|$.
	\item $\{D_{n,j}\}_n\sim_\lambda f_j$, where $D_{n,j}={\rm diag}(\lambda_{1,n}^{(j)},\ldots,\lambda_{|\Lambda_{n,j}|,n}^{(j)})$ and $\{\lambda_{1,n}^{(j)},\ldots,\lambda_{|\Lambda_{n,j}|,n}^{(j)}\}=\Lambda_{n,j}$.
	\item $\Lambda_{n,j}\subseteq[\inf_{[a,b]}f_j-\delta_n,\sup_{[a,b]}f_j+\delta_n]$ for some $\delta_n\to0$ as $n\to\infty$.
\end{itemize}
By Corollary~\ref{eccoGi-n}, for every $j=1,\ldots,k$ and every a.u.\ grid $\{x_{i,n}^{(j)}\}_{i=1,\ldots,|\Lambda_{n,j}|}$ in $[a,b]$ with $\{x_{i,n}^{(j)}\}_{i=1,\ldots,|\Lambda_{n,j}|}\subset[a,b]$, if $\sigma_{n,j}$ and $\tau_{n,j}$ are two permutations of $\{1,\ldots,|\Lambda_{n,j}|\}$ such that the vectors $[f_j(x_{\sigma_{n,j}(1),n}^{(j)}),\ldots,f_j(x_{\sigma_{n,j}(|\Lambda_{n,j}|),n}^{(j)})]$ and $[\lambda_{\tau_{n,j}(1),n}^{(j)},\ldots,\lambda_{\tau_{n,j}(|\Lambda_{n,j}|),n}^{(j)}]$ are sorted in increasing order, we have
\[ \max_{i=1,\ldots,|\Lambda_{n,j}|}|f_j(x_{\sigma_{n,j}(i),n}^{(j)})-\lambda_{\tau_{n,j}(i),n}^{(j)}|\to0\ \,\mbox{as}\ \,n\to\infty. \tag*{\qedhere} \]
\end{proof}

\section{Numerical experiments}\label{numexp}

In this section, after recalling some properties of Toeplitz matrices, we illustrate our main results through numerical examples.

\subsection{Preliminaries on Toeplitz matrices}

It is not difficult to see that the conjugate transpose of $T_n(f)$ is given by
\begin{equation*}
T_n(f)^*=T_n(\overline f)
\end{equation*}
for every $f\in L^1([-\pi,\pi])$ and every $n$; see, e.g., \cite[Section~6.2]{GLTbookI}.
In particular, if $f$ is real a.e., then $\overline f=f$ a.e.\ and the matrices $T_n(f)$ are Hermitian.
The next theorem collects some properties of Toeplitz matrices generated by a real function. For the proof, see \cite[Theorems~6.1 and~6.5]{GLTbookI}.

\begin{theorem}\label{block-lemma}
Let $f\in L^1([-\pi,\pi])$ be real and let
$m_f=\mathop{\rm ess\,inf}_{[-\pi,\pi]}f$ and $M_f=\mathop{\rm ess\,sup}_{[-\pi,\pi]}f$.
Then, the following properties hold.
\begin{enumerate}[nolistsep,leftmargin=*]
	\item $T_n(f)$ is Hermitian and the eigenvalues of $T_n(f)$ lie in the interval $[m_f,M_f]$ for all $n$.
	\item If $f$ is not a.e.\ constant, then the eigenvalues of $T_n(f)$ lie in $(m_f,M_f)$ for all $n$.
	\item $\{T_n(f)\}_n\sim_\lambda f$.
\end{enumerate}
\end{theorem}

\subsection{Numerical examples} 

\begin{example}\label{e1}
Let $f(\theta)=a+b\cos\theta:[-\pi,\pi]\to\mathbb R$, with $a,b\in\mathbb R$ and $b\ne0$, and let $\Lambda_n=\{\lambda_{1,n},\ldots,\lambda_{n,n}\}$ be the multiset consisting of the eigenvalues of the Hermitian Toeplitz matrix $T_n(f)$. By Theorem~\ref{block-lemma} and the fact that $f$ is an even function, we have $\{\Lambda_n\}_n\sim f|_{[0,\pi]}$ and $\Lambda_n\subseteq(\min_{[0,\pi]}f,\max_{[0,\pi]}f)$. Since $f$ is continuous, $f|_{[0,\pi]}$ and $\Lambda_n$ satisfy all the assumptions of Corollary~\ref{eccoGi-n} and Theorem~\ref{thm:mmd}, and we therefore conclude the following.
\begin{itemize}[nolistsep,leftmargin=*]
	\item For every a.u.\ grid $\{\theta_{i,n}\}_{i=1,\ldots,n}$ in $[0,\pi]$ with $\{\theta_{i,n}\}_{i=1,\ldots,n}\subset[0,\pi]$, we have
	\[ \max_{i=1,\ldots,n}|f(\theta_{i,n})-\lambda_{\tau_n(i),n}|\to0\ \,\mbox{as}\, \ n\to\infty, \]
	where $\tau_n$ is a suitable permutation of $\{1,\ldots,n\}$.
	\item There exists an a.u.\ grid $\{\theta_{i,n}\}_{i=1,\ldots,n}$ in $[0,\pi]$ with $\{\theta_{i,n}\}_{i=1,\ldots,n}\subset[0,\pi]$ such that, for every $n$,
	\[ \lambda_{\tau_n(i),n}=f(\theta_{i,n}),\qquad i=1,\ldots,n, \]
	where $\tau_n$ is a suitable permutation of $\{1,\ldots,n\}$.
\end{itemize}
The two previous assertions are actually well known in this case, because $\Lambda_n=\{f(\frac{i\pi}{n+1}):i=1,\ldots,n\}$; see \cite[Theorem~2.4]{BG}.
\end{example}

\begin{example}\label{e2}

\begin{figure}
\centering
\begin{minipage}{0.485\textwidth}
\centering
\includegraphics[width=\textwidth]{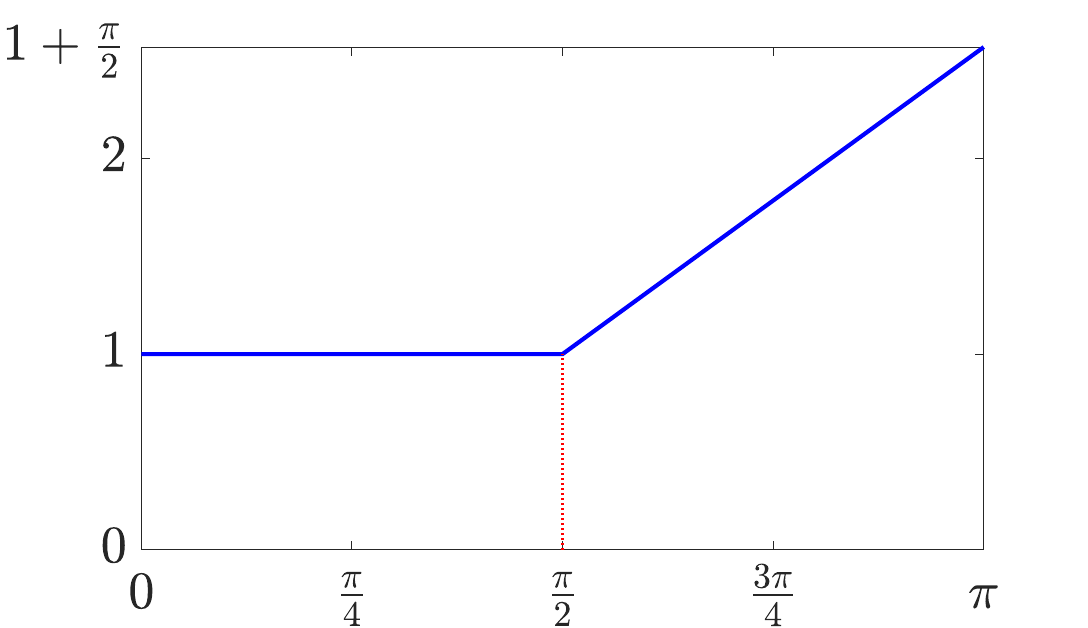}
\caption{Example~\ref{e2}: Graph on the interval $[0,\pi]$ of the function $f(\theta)$ defined in \eqref{fc}.}
\label{fc_graph}
\end{minipage}
\hspace{0.01\textwidth}
\begin{minipage}{0.485\textwidth}
\centering
\captionof{table}{Example~\ref{e2}: Computation of $M_n$ for increasing values of $n$.}
\begin{tabular}{rc}
\toprule
$n$ & $M_n$\\
\midrule
8   & 0.0851\\
16  & 0.0632\\
32  & 0.0454\\
64  & 0.0312\\
128 & 0.0206\\
256 & 0.0132\\
512 & 0.0082\\
1024 & 0.0050\\
\bottomrule
\end{tabular}
\label{e2t}
\end{minipage}
\end{figure}

Let $f:[-\pi,\pi]\to\mathbb R$,
\begin{equation}\label{fc}
f(\theta)=\left\{\begin{aligned}
&1, &\quad0&\le\theta<\pi/2,\\
&\theta+1-\pi/2, &\quad\pi/2&\le\theta\le\pi,\\
&f(-\theta), &\quad-\pi&\le\theta<0,
\end{aligned}\right.
\end{equation}
and let $\Lambda_n=\{\lambda_{1,n},\ldots,\lambda_{n,n}\}$ be the multiset consisting of the eigenvalues of the Hermitian Toeplitz matrix $T_n(f)$. Figure~\ref{fc_graph} shows the graph of $f$ over the interval $[0,\pi]$. By Theorem~\ref{block-lemma} and the fact that $f$ is an even function, we have $\{\Lambda_n\}_n\sim f|_{[0,\pi]}$ and $\Lambda_n\subseteq(\min_{[0,\pi]}f,\max_{[0,\pi]}f)=(1,1+\pi/2)$. Since $f$ is continuous, $f|_{[0,\pi]}$ and $\Lambda_n$ satisfy all the assumptions of Corollary~\ref{eccoGi-n}, and we therefore conclude that, for every a.u.\ grid $\{\theta_{i,n}\}_{i=1,\ldots,n}$ in $[0,\pi]$ with $\{\theta_{i,n}\}_{i=1,\ldots,n}\subset[0,\pi]$, we have
\begin{equation}\label{max-->0}
M_n=\max_{i=1,\ldots,n}|f(\theta_{i,n})-\lambda_{\tau_n(i),n}|\to0\ \,\mbox{as}\, \ n\to\infty,
\end{equation}
where $\tau_n$ is the permutation of $\{1,\ldots,n\}$ that sorts $\lambda_{1,n},\ldots,\lambda_{n,n}$ in increasing order (note that $f|_{[0,\pi]}$ is increasing).
To provide numerical evidence of \eqref{max-->0}, in Table~\ref{e2t} we compute $M_n$ for increasing values of $n$ in the case of the a.u.\ grid $\theta_{i,n}=\frac{i\pi}{n+1}$, $i=1,\ldots,n$. We see from the table that $M_n\to0$ as $n\to\infty$, though the convergence is slow.

Now we observe that $f|_{[0,\pi]}$ and $\Lambda_n$ do not satisfy the assumptions of Theorem~\ref{thm:mmd}. Actually, they satisfy all the assumptions of Theorem~\ref{thm:mmd} except the hypothesis that $f$ has a finite number of local maximum/minimum points. Indeed, $f$ is constant on $[0,\pi/2]$ and so all points in $[0,\pi/2)$ are both local maximum and local minimum points for $f$ according to our Definition~\ref{wlep}. We observe that, in fact, the thesis of Theorem~\ref{thm:mmd} does not hold in this case, i.e., there is no a.u.\ grid $\{\theta_{i,n}\}_{i=1,\ldots,n}$ in $[0,\pi]$ with $\{\theta_{i,n}\}_{i=1,\ldots,n}\subset[0,\pi]$ such that, for every $n$,
\[ \lambda_{\tau_n(i),n}=f(\theta_{i,n}),\qquad i=1,\ldots,n, \]
for a suitable permutation $\tau_n$ of $\{1,\ldots,n\}$. This is clear, because $\Lambda_n\subset(1,1+\pi/2)$ and so any grid $\{\theta_{i,n}\}_{i=1,\ldots,n}\subset[0,\pi]$ satisfying the previous condition must be contained in $(\pi/2,\pi)$, which implies that it cannot be a.u.\ in $[0,\pi]$.
\end{example}

\begin{example}\label{e3}

\begin{figure}
\centering
\begin{minipage}{0.485\textwidth}
\centering
\includegraphics[width=\textwidth]{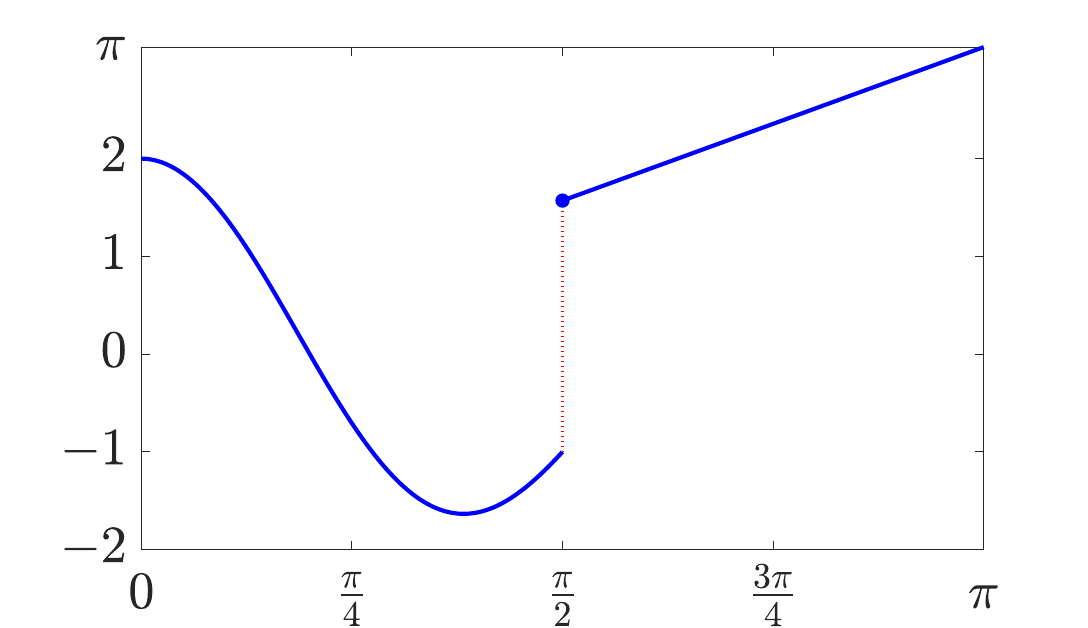}
\caption{Example~\ref{e3}: Graph on the interval $[0,\pi]$ of the function $f(\theta)$ defined in \eqref{fd}.}
\label{fd_graph}
\end{minipage}
\hspace{0.01\textwidth}
\begin{minipage}{0.485\textwidth}
\centering
\captionof{table}{Example~\ref{e3}: Computation of $M_n$ for increasing values of $n$.}
\begin{tabular}{rc}
\toprule
$n$ & $M_n$\\
\midrule
8   & 0.7220\\
16  & 0.5625\\
32  & 0.4471\\
64  & 0.2956\\
128 & 0.1783\\
256 & 0.1096\\
512 & 0.0605\\
1024 & 0.0373\\
\bottomrule
\end{tabular}
\label{e3t}
\end{minipage}
\end{figure}

Let $f:[-\pi,\pi]\to\mathbb R$,
\begin{equation}\label{fd}
f(\theta)=\left\{\begin{aligned}
&\cos(2\theta)+\cos(3\theta), &\quad0&\le\theta<\pi/2,\\
&\theta, &\quad\pi/2&\le\theta\le\pi,\\
&f(-\theta), &\quad-\pi&\le\theta<0,
\end{aligned}\right.
\end{equation}
and let $\Lambda_n=\{\lambda_{1,n},\ldots,\lambda_{n,n}\}$ be the multiset consisting of the eigenvalues of the Hermitian Toeplitz matrix $T_n(f)$. Figure~\ref{fd_graph} shows the graph of $f$ over the interval $[0,\pi]$. By Theorem~\ref{block-lemma} and the fact that $f$ is an even function, we have $\{\Lambda_n\}_n\sim f|_{[0,\pi]}$ and $\Lambda_n\subseteq(\min_{[0,\pi]}f,\max_{[0,\pi]}f)=(-\frac{25}{54}-\frac{10\sqrt{10}}{27},\pi)$. Note that the function $f|_{[0,\pi]}$ is not continuous, but it satisfies anyway all the assumptions of Corollary~\ref{eccoGi-n} and Theorem~\ref{thm:mmd}, and we therefore conclude the following.
\begin{itemize}[nolistsep,leftmargin=*]
	\item For every a.u.\ grid $\{\theta_{i,n}\}_{i=1,\ldots,n}$ in $[0,\pi]$ with $\{\theta_{i,n}\}_{i=1,\ldots,n}\subset[0,\pi]$, we have
	\begin{equation}\label{max-->0'}
	M_n=\max_{i=1,\ldots,n}|f(\theta_{\sigma_n(i),n})-\lambda_{\tau_n(i),n}|\to0\ \,\mbox{as}\, \ n\to\infty,
	\end{equation}
	where $\sigma_n$ and $\tau_n$ are two permutations of $\{1,\ldots,n\}$ such that the vectors $[f(\theta_{\sigma_n(1),n}),\ldots,f(\theta_{\sigma_n(n),n})]$ and $[\lambda_{\tau_n(1),n},\ldots,\lambda_{\tau_n(n),n}]$ are sorted in increasing order.
	\item There exists an a.u.\ grid $\{\theta_{i,n}\}_{i=1,\ldots,n}$ in $[0,\pi]$ with $\{\theta_{i,n}\}_{i=1,\ldots,n}\subset[0,\pi]$ such that, for every $n$,
	\[ \lambda_{\tau_n(i),n}=f(\theta_{i,n}),\qquad i=1,\ldots,n, \]
	where $\tau_n$ is a suitable permutation of $\{1,\ldots,n\}$.
\end{itemize}
To provide numerical evidence of \eqref{max-->0'}, in Table~\ref{e3t} we compute $M_n$ for increasing values of $n$ in the case of the a.u.\ grid $\theta_{i,n}=\frac{i\pi}{n+1}$, $i=1,\ldots,n$.
We see from the table that $M_n\to0$ as $n\to\infty$, though the convergence is slow.
\end{example}

\begin{example}\label{e4}
Consider the following second-order differential problem:
\begin{equation*}
\left\{\begin{aligned}
&-(a(x)u'(x))'=g(x), &&x\in(0,1),\\[3pt]
&u(0)=\alpha,\quad u(1)=\beta,
\end{aligned}\right.
\end{equation*}
where $a:[0,1]\to\mathbb R$ is assumed to be continuous and non-negative on $[0,1]$.
In the classical finite difference method based on second-order central finite differences over the uniform grid $x_i=\frac{i}{n+1}$, $i=0,\ldots,n+1$, the computation of the numerical solution reduces to solving a linear system whose coefficient matrix is the symmetric $n\times n$ tridiagonal matrix given by
\[ A_n = \left[\begin{array}{ccccc}
a_{\frac1{2^{\vphantom{\mbox{\tiny 1}}}}}+a_{\frac3{2^{\vphantom{\mbox{\tiny 1}}}}} & \ \ -a_{\frac3{2^{\vphantom{\mbox{\tiny 1}}}}} & \ \ \ \ & \ \ \ \ & \ \ \ \ \\[5pt]
-a_{\frac3{2^{\vphantom{\mbox{\tiny 1}}}}} & \ \ a_{\frac3{2^{\vphantom{\mbox{\tiny 1}}}}}+a_{\frac5{2^{\vphantom{\mbox{\tiny 1}}}}} & \ \ \ \ -a_{\frac5{2^{\vphantom{\mbox{\tiny 1}}}}} & \ \ \ \ & \ \ \ \ \\[5pt]
& \ \ -a_{\frac5{2^{\vphantom{\mbox{\tiny 1}}}}} & \ \ \ \ \ddots & \ \ \ \ \ddots & \ \ \ \ \\[5pt]
& \ \ & \ \ \ \ \ddots & \ \ \ \ \ddots & \ \ \ \ -a_{n-\frac1{2^{\vphantom{\mbox{\tiny 1}}}}} \\[5pt]
& \ \ & \ \ \ \ & \ \ \ \ -a_{n-\frac1{2^{\vphantom{\mbox{\tiny 1}}}}} & \ \ \ \ a_{n-\frac1{2^{\vphantom{\mbox{\tiny 1}}}}}+a_{n+\frac1{2^{\vphantom{\mbox{\tiny 1}}}}}
\end{array}\right],
\]
where $a_i=a(x_i)$ for all $i$ in the real interval $[0,n+1]$; see \cite[Section~10.5.1]{GLTbookI} for more details. Let $f(x,\theta)=a(x)(2-2\cos\theta):[0,1]\times[0,\pi]\to\mathbb R$, and let $\Lambda_n=\{\lambda_{1,n},\ldots,\lambda_{n,n}\}$ be the multiset consisting of the eigenvalues of $A_n$. 
We know from \cite[Theorem~10.5]{GLTbookI} that $\{A_n\}_n\sim_\lambda f$, i.e., $\{\Lambda_n\}_n\sim f$. Moreover, in view of the dyadic decomposition of $A_n$ in \cite[Section~2]{NSP-LAA}, we have
\begin{align*} \Lambda_n&\subseteq\left[\lambda_{\min}(T_n(2-2\cos\theta))\cdot\min_{[0,1]}a,\ \lambda_{\max}(T_n(2-2\cos\theta))\cdot\max_{[0,1]}a\right]\subseteq\left[0,4\max_{[0,1]}a\right]\\
&=\left[\min_{[0,1]\times[0,\pi]}f,\max_{[0,1]\times[0,\pi]}f\right]=f([0,1]\times[0,\pi])=\mathcal{ER}(f),
\end{align*}
where the latter equality follows from the continuity of $f$ and the fact that the domain $[0,1]\times[0,\pi]$ is not ``too wild'' (in particular, it is contained in the closure of its interior); see \cite[Exercise~2.1]{GLTbookI}. 

Now, following the notations of Theorem~\ref{th:main4-n}, let $\aa=(0,0)$ and $\bb=(1,\pi)$, so that $[\aa,\bb]=[0,1]\times[0,\pi]$.
Assume that $n$ is a perfect square, let $\nn=\nn(n)=(\sqrt n,\sqrt n)$, let $[\lambda_{\bu,\nn}^{(n)},\ldots,\lambda_{\nn,\nn}^{(n)}]=[\lambda_{\ii,\nn}^{(n)}]_{\ii=\bu,\ldots,\nn}$ be the same as the vector $[\lambda_{1,n},\ldots,\lambda_{n,n}]=[\lambda_{i,n}]_{i=1,\ldots,n}$ but indexed with the 2-index $\ii=\bu,\ldots,\nn$ instead of the 1-index $i=1,\ldots,n$, and consider the a.u.\ grid in $[\aa,\bb]$ given by
\[ \mathcal G_\nn^{(n)}=\{\xx_{\ii,\nn}^{(n)}\}_{\ii=\bu,\ldots,\nn},\qquad \xx_{\ii,\nn}^{(n)}=\aa+\frac{\ii(\bb-\aa)}{\nn}=\biggl(\frac{i_1}{\sqrt n},\frac{i_2\pi}{\sqrt n}\biggr),\qquad \ii=\bu,\ldots,\nn. \]
Then, by Theorem~\ref{th:main4-n}, we conclude that
\begin{equation}\label{max-->0''}
M_n=\max_{\ii=\bu,\ldots,\nn}|f(\xx_{\sigma_n(\ii),\nn}^{(n)})-\lambda_{\tau_n(\ii),\nn}^{(n)}|\to0\ \,\mbox{as}\ \,n\to\infty,
\end{equation}
where $\sigma_n$ and $\tau_n$ are two permutations of the 2-index range $\{\bu,\ldots,\nn\}$ such that the vectors 
\begin{align*}
[f(\xx_{\sigma_n(\bu),\nn}^{(n)}),\ldots,f(\xx_{\sigma_n(\nn),\nn}^{(n)})]&=[f(\xx_{\sigma_n(\ii),\nn}^{(n)})]_{\ii=\bu,\ldots,\nn},\\ [\lambda_{\tau_n(\bu),\nn}^{(n)},\ldots,\lambda_{\tau_n(\nn),\nn}^{(n)}]&=[\lambda_{\tau_n(\ii),\nn}^{(n)}]_{\ii=\bu,\ldots,\nn}
\end{align*}
are sorted in increasing order.
To provide numerical evidence of \eqref{max-->0''}, in Table~\ref{e4t} we compute $M_n$ for increasing values of $n$ and different choices of $a(x)$.
In all cases, we see from the table that $M_n\to0$ as $n\to\infty$, though the convergence is slow.

\begin{table}
\centering
\caption{Example~\ref{e4}: Computation of $M_n$ for increasing values of $n$ and different choices of $a(x)$.}
\begin{subtable}{0.32\textwidth}
\centering
\caption{$a(x)=\e^{-x}$}
\begin{tabular}{rc}
\toprule
$n$ & $M_n$\\
\midrule
900  & 0.0684\\
1600 & 0.0559\\
2500 & 0.0473\\
3600 & 0.0411\\
4900 & 0.0364\\
6400 & 0.0326\\
8100 & 0.0296\\
10000 & 0.0271\\
\bottomrule
\end{tabular}
\end{subtable}
\begin{subtable}{0.32\textwidth}
\centering
\caption{$a(x)=2+\cos(3x)$}
\begin{tabular}{rc}
\toprule
$n$ & $M_n$\\
\midrule
900  & 0.1471\\
1600 & 0.1132\\
2500 & 0.0890\\
3600 & 0.0738\\
4900 & 0.0634\\
6400 & 0.0558\\
8100 & 0.0484\\
10000 & 0.0436\\
\bottomrule
\end{tabular}
\end{subtable}
\begin{subtable}{0.32\textwidth}
\centering
\caption{$a(x)=x\log(1+x)$}
\begin{tabular}{rc}
\toprule
$n$ & $M_n$\\
\midrule
900  & 0.1240\\
1600 & 0.0915\\
2500 & 0.0717\\
3600 & 0.0583\\
4900 & 0.0497\\
6400 & 0.0435\\
8100 & 0.0383\\
10000 & 0.0344\\
\bottomrule
\end{tabular}
\end{subtable}
\label{e4t}
\end{table}
\end{example}

\begin{example}\label{e4p}
Consider the two-dimensional Poisson problem
\begin{equation*}
\left\{\begin{aligned}
&-\Delta u(\xx)=g(\xx), &&\xx\in(0,1)^2,\\[3pt]
&u(\xx)=h(\xx), &&\xx\in\partial((0,1)^2).
\end{aligned}\right.
\end{equation*}
In the isogeometric Galerkin discretization based on tensor-product biquadratic B-splines defined over the uniform grid $\ii/\nn$ for $\ii=\mathbf0,\ldots,\nn$ and $\nn=\nn(n)=(n,n)$, the computation of the numerical solution reduces to solving a linear system whose coefficient matrix is the symmetric $n^2\times n^2$ matrix given by 
\[ A_n=K_n\otimes M_n+M_n\otimes K_n, \]
where $\otimes$ is the Kronecker tensor product and $K_n$, $M_n$ are the symmetric $n\times n$ matrices given by
\begin{align*}
K_n=\frac16\begin{bmatrix}
8 & -1 & -1 & & & & \\
-1 & 6 & -2 & -1 & & & \\
-1 & -2 & 6 & -2 & -1 & & \\
& \ddots & \ddots & \ddots & \ddots & \ddots & \\
& & -1 & -2 & 6 & -2 & -1\\
& & & -1 & -2 & 6 & -1\\
& & & & -1 & -1 & 8
\end{bmatrix},\qquad M_n=\frac1{120}\begin{bmatrix}
40 & 25 & 1 & & & & \\
25 & 66 & 26 & 1 & & & \\
1 & 26 & 66 & 26 & 1 & & \\
& \ddots & \ddots & \ddots & \ddots & \ddots & \\
& & 1 & 26 & 66 & 26 & 1\\
& & & 1 & 26 & 66 & 25\\
& & & & 1 & 25 & 40
\end{bmatrix};
\end{align*}
see \cite[Section~7.6]{GLTbookII} for more details. Let
\[ f:[0,\pi]^2\to\mathbb R,\qquad f(\theta_1,\theta_2)=\kappa(\theta_1)\mu(\theta_2)+\mu(\theta_1)\kappa(\theta_2), \]
where
\[ \kappa(\theta)=1-\frac23\cos\theta-\frac13\cos(2\theta),\qquad \mu(\theta)=\frac{11}{20}+\frac{13}{30}\cos\theta+\frac1{60}\cos(2\theta), \]
and let $\Lambda_n=\{\lambda_{1,n},\ldots,\lambda_{n^2,n}\}$ be the multiset consisting of the eigenvalues of $A_n$. We know from \cite[Theorem~7.7]{GLTbookII} that $\{A_n\}_n\sim_\lambda f$, i.e., $\{\Lambda_n\}_n\sim f$. Moreover, numerical experiments reveal that there are no outliers in the biquadratic B-spline case, i.e.,
\[ \Lambda_n\subseteq\left[0,\frac32\right]=\left[\min_{[0,\pi]^2}f,\max_{[0,\pi]^2}f\right]=f([0,\pi]^2)=\mathcal{ER}(f) \]
for all $n$. Thus, Theorem~\ref{th:main4-n} applies in this case. In fact, 
in view of the spectral decompositions obtained in \cite[Section~3.3]{nlaa2018}, the eigenvalues of $A_n$ are given by
\[ f(\xx_{\ii,\nn}^{(n)}),\qquad\ii=\bu,\ldots,\nn, \]
where $\nn=\nn(n)=(n,n)$ as before and $\mathcal G_\nn^{(n)}=\{\xx_{\ii,\nn}^{(n)}\}_{\ii=\bu,\ldots,\nn}$ is the a.u.\ grid in $[0,\pi]^2$ given by
\[ \xx_{\ii,\nn}^{(n)}=\biggl(\frac{i_1\pi}{n},\frac{i_2\pi}{n}\biggr),\qquad \ii=\bu,\ldots,\nn. \]
\end{example}

\begin{example}\label{e5}
Consider the following second-order differential problem:
\begin{equation*}
\left\{\begin{aligned}
&-u''(x)=g(x), &&x\in(0,1),\\[3pt]
&u(0)=\alpha,\quad u(1)=\beta.
\end{aligned}\right.
\end{equation*}
In the classical Gelerkin method with basis functions 
given by the $2n-1$ B-splines of degree $2$ and smoothness $C^0([0,1])$ defined over the uniform knot sequence $\{0,0,0,\frac1n,\frac1n,\frac2n,\frac2n,\ldots,\frac{n-1}n,\frac{n-1}n,1,1,1\}$ and vanishing at the boundary points $x=0$ and $x=1$, the computation of the numerical solution reduces to solving a linear system whose coefficient matrix is the symmetric $(2n-1)\times(2n-1)$ matrix given by
\begin{equation*}
A_n = \frac13
\left[\,\begin{array}{|rr|rr|rr|rr|rr|r}
\cline{1-4}
4 & -2 & & & \multicolumn{7}{r}{} \\
-2 & 8 & -2 & -2 & \multicolumn{7}{r}{} \\
\cline{1-6}
& -2 & 4 & -2 & & & \multicolumn{5}{r}{} \\
& -2 & -2 & 8 & -2 & -2 & \multicolumn{5}{r}{} \\
\cline{1-6}
\multicolumn{2}{r}{} & \ddots & \multicolumn{1}{r}{} & \ddots & \multicolumn{1}{r}{} & \ddots & \multicolumn{4}{r}{} \\
\multicolumn{3}{r}{} & \multicolumn{1}{r}{\ddots} & & \multicolumn{1}{r}{\ddots} & & \multicolumn{1}{r}{\ddots} & \multicolumn{3}{r}{} \\
\cline{5-10}
\multicolumn{2}{r}{} & & & & -2 & 4 & -2 & & & \\
\multicolumn{2}{r}{} & & & & -2 & -2 & 8 & -2 & -2 & \\
\cline{5-10}
\multicolumn{4}{r}{} & & & & -2 & 4 & -2 & \\
\multicolumn{4}{r}{} & & & & -2 & -2 & 8 & -2\\
\cline{7-10}
\multicolumn{8}{r}{} & & \multicolumn{1}{r}{-2} & 4
\end{array}\,\right];
\end{equation*}
see \cite[Section~2.3.2]{Tom-paper} for more details. Let
\[ f(\theta)=\frac13\left[\begin{array}{cc}
4 & -2-2\e^{\i\theta}\\
-2-2\e^{-\i\theta} & 8-4\cos\theta
\end{array}\right]:[0,\pi]\to\mathbb C^{2\times 2}, \]
and let $\Lambda_n=\{\lambda_{1,n},\ldots,\lambda_{2n-1,n}\}$ be the multiset consisting of the eigenvalues of $A_n$ (sorted in increasing order for later convenience). We know from \cite[Section~2.3.2]{Tom-paper} and \cite[Theorem~6.5]{GLTbookIII} that $\{A_n\}_n\sim_\lambda f$, i.e., $\{\Lambda_n\}_n\sim f$. By Definition~\ref{dd}, the latter is equivalent to $\{\Lambda_n\}_n\sim\mathop{\rm diag}(f_1,f_2)$, where $f_1,f_2:[0,\pi]\to\mathbb R$ are given by
\begin{align}
f_1(\theta)&=\lambda_1(f(\theta))=2-\frac23\cos\theta-\frac23\sqrt{3+\cos^2\theta},\label{f1theta}\\[4pt]
f_2(\theta)&=\lambda_2(f(\theta))=2-\frac23\cos\theta+\frac23\sqrt{3+\cos^2\theta}.\label{f2theta}
\end{align}
Figure~\ref{Bspline_p2k0} shows the graphs of the functions $f_{1,2}(\theta)$ and the set of eigenvalues $\Lambda_n$ for $n=20$. The eigenvalues $\lambda_{i,n}$, $i=1,\ldots,2n-1$, 
are positioned at $\frac{i\pi}{n}$ for $i=1,\ldots,n$ and $\frac{(i-n)\pi}n$ for $i=n+1,\ldots,2n-1$. Note that 
\begin{align*}
\mathcal{ER}(f_1)&=f_1([0,\pi])=[f_1(0),f_1(\pi)]=\bigl[0,\textstyle{\frac43}\bigr],\\
\mathcal{ER}(f_2)&=f_2([0,\pi])=[f_2(0),f_2(\pi)]=\bigl[\textstyle{\frac83},4\bigr].
\end{align*}
From the figure, we may assume that the hypotheses of Theorem~\ref{eccoGi'} are satisfied with the partition $\{\tilde\Lambda_{n,1},\tilde\Lambda_{n,2}\}$ of $\Lambda_n$ given by $\tilde\Lambda_{n,1}=\{\lambda_{1,n},\ldots,\lambda_{n,n}\}$ and $\tilde\Lambda_{n,2}=\{\lambda_{n+1,n},\ldots,\lambda_{2n-1,n}\}$. Thus, by Theorem~\ref{eccoGi'}, for every $n$ there exists a partition $\{\Lambda_{n,1},\Lambda_{n,2}\}$ of $\Lambda_n$---which must necessarily coincide with the original partition $\{\tilde\Lambda_{n,1},\tilde\Lambda_{n,2}\}$---with the following properties.
\begin{itemize}[nolistsep,leftmargin=*]
	\item $|\Lambda_{n,1}|=|\tilde\Lambda_{n,1}|=n$ and $|\Lambda_{n,2}|=|\tilde\Lambda_{n,2}|=n-1$.
	\item $\Lambda_{n,1}\subseteq\bigl[-\delta_n,\frac43+\delta_n\bigr]$ and $\Lambda_{n,2}\subseteq\bigl[\frac83-\delta_n,4+\delta_n\bigr]$ for some $\delta_n\to0$ as $n\to\infty$.
	\item $\{\Lambda_{n,1}\}_n\sim f_1$ and $\{\Lambda_{n,2}\}_n\sim f_2$.
	\item If $\{\theta_{i,n}\}_{i=1,\ldots,n}$ and $\{\vartheta_{i,n}\}_{i=1,\ldots,n-1}$ are any two a.u.\ grids in $[0,\pi]$ contained in $[0,\pi]$, then 
	\begin{align}
	M_{n,1}&=\max_{i=1,\ldots,n}|f_1(\theta_{i,n})-\lambda_{i,n}|\to0\ \,\mbox{as}\, \ n\to\infty,\label{max1-->0}\\
	M_{n,2}&=\max_{i=1,\ldots,n-1}|f_2(\vartheta_{i,n})-\lambda_{i+n,n}|\to0\ \,\mbox{as}\, \ n\to\infty.\label{max2-->0}
	\end{align}
\end{itemize}
Actually, we can say more than \eqref{max1-->0}--\eqref{max2-->0}. Indeed, numerical experiments reveal that $M_{n,1}=M_{n,2}=0$ for all $n$ if we choose the a.u.\ grids suggested by Figure~\ref{Bspline_p2k0}, i.e., $\theta_{i,n}=\frac{i\pi}n$, $i=1,\ldots,n$, and $\vartheta_{i,n}=\frac{i\pi}n$, $i=1,\ldots,n-1$. 
In other words, the eigenvalues of $A_n$ are explicitly given by
\[ \{f_1(\theta_{i,n}):i=1,\ldots,n\}\cup\{f_2(\vartheta_{i,n}):i=1,\ldots,n-1\}. \]
We refer the reader to Appendix~\ref{a} for further explicit formulas for the eigenvalues of B-spline Galerkin discretization matrices. These formulas have been obtained through numerical experiments and provide further confirmations of Theorem~\ref{eccoGi'}.

\begin{figure}
\centering
\centering
\includegraphics[width=0.5\textwidth]{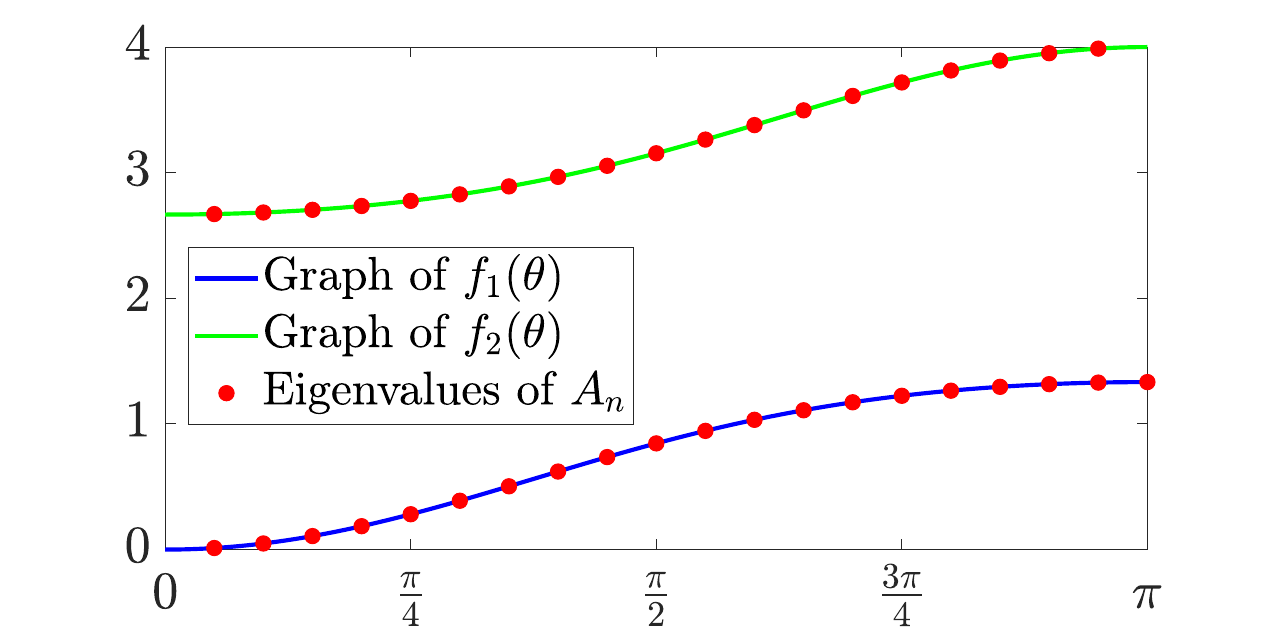}
\caption{Example~\ref{e5}: Graphs of the functions $f_{1,2}(\theta)$ in \eqref{f1theta}--\eqref{f2theta} and set of eigenvalues $\Lambda_n$ for $n=20$.}
\label{Bspline_p2k0}
\end{figure}

\end{example}

\section{Conclusions}\label{conc}

We have provided new insights into the notion of asymptotic (spectral) distribution by extending previous results due to Bogoya, B\"ottcher, Grudsky, and Maximenko \cite{maximum_norm,Bottcher-ded-Grudsky}. In particular, using the concept of monotone rearrangement (quantile function) and matrix analysis arguments from the theory of GLT sequences, we have shown that, under suitable assumptions, if the asymptotic distribution of a sequence of multisets $\Lambda_n=\{\lambda_{1,n},\ldots,\lambda_{d_n,n}\}$ is described by a function $f$ in the sense of Definition~\ref{adn}, then we observe the uniform convergence to $0$ of the difference between a proper permutation of the vector $[\lambda_{1,n},\ldots,\lambda_{d_n,n}]$ and the vector of samples of $f$ over an a.u.\ grid in the domain of $f$.
We have also illustrated through numerical experiments the main results of the paper.

We conclude this paper with a remark. The notion of asymptotic distribution given in Definition~\ref{adn} is deeply connected with the notion of vague convergence of probability measures, which is also referred to as convergence in distribution in \cite{Bottcher-ded-Grudsky}. More precisely, as shown in \cite{sm}, 
\begin{itemize}[nolistsep,leftmargin=*]
	\item if $\{\Lambda_n\}_n$ is as in Definition~\ref{adn}, then we can associate with each $\Lambda_n$ the atomic probability measure on $\mathbb C$ defined as
	\[ \mu_{\Lambda_n}=\frac1{d_n}\sum_{i=1}^{d_n}\delta_{\lambda_{i,n}}, \]
	where $\delta_z$ is the Dirac probability measure such that $\delta_z(E)=1$ if $z\in E$ and $\delta_z(E)=0$ otherwise;
	\item if $f$ is as in Definition~\ref{adn} with $k=1$, then we can associate with $f$ a uniquely determined probability measure $\mu_f$ on $\mathbb C$ such that
	\[ \frac1{\mu_k(D)}\int_DF(f(\xx)){\rm d}\xx=\int_\mathbb CF(z){\rm d}\mu_f(z),\qquad\forall\,F\in C_c(\mathbb C). \]
\end{itemize}
The asymptotic distribution relation 
\[ \lim_{n\to\infty}\frac1{d_n}\sum_{i=1}^{d_n}F(\lambda_{i,n})=\frac1{\mu_d(\Omega)}\int_\Omega F(f(\xx)){\rm d}\xx,\qquad\forall\,F\in C_c(\mathbb C), \]
can therefore be rewritten as
\begin{equation*}
\lim_{n\to\infty}\int_\mathbb C F(z){\rm d}\mu_{\Lambda_n}(z)=\int_\mathbb CF(z){\rm d}\mu_f(z),\qquad\forall\,F\in C_c(\mathbb C),
\end{equation*}
which is equivalent to saying that $\mu_{\Lambda_n}$ converges vaguely to $\mu_f$ \cite[Definition~13.12]{proba_comprehensive}. 
This equivalence allows for a reinterpretation of the main results of this paper in a probabilistic perspective. In this regard, it is worth pointing out that any multiset of complex numbers $\Lambda_n=\{\lambda_{1,n},\ldots,\lambda_{d_n,n}\}$ coincides with the spectrum of a matrix $A_n$ (take $A_n={\rm diag}(\lambda_{1,n},\ldots,\lambda_{d_n,n})$) and any probability measure $\mu$ on $\mathbb C$ coincides with $\mu_f$ for some $f$ \cite[Corollary~1]{sm}.

\appendix
\section{Formulas for the eigenvalues of B-spline Galerkin discretization matrices}\label{a}
Consider the following second-order differential eigenvalue problem:
\begin{equation*}
\left\{\begin{aligned}
&-u_j''(x)=\lambda_ju_j(x), &&\quad x\in(0,1),\\[3pt]
&u_j(0)=0,\quad u_j(1)=0.
\end{aligned}\right.
\end{equation*}
Let $p\ge1$ and $0\le k\le p-1$. In the classical Gelerkin method with basis functions given by the $n(p-k)+k-1$ B-splines $B_{2,p,k},\ldots,B_{n(p-k)+k,p,k}$ of degree $p$ and smoothness $C^k([0,1])$ defined over the uniform knot sequence
\[ \biggl\{\,\underbrace{0,\ldots,0}_{p+1},\underbrace{\frac1n,\ldots,\frac1n}_{p-k},\underbrace{\frac2n,\ldots,\frac2n}_{p-k},\ldots,\underbrace{\frac{n-1}n,\ldots,\frac{n-1}n}_{p-k},\underbrace{1,\ldots,1}_{p+1}\,\biggr\} \]
and vanishing at the boundary points $x=0$ and $x=1$, the computation of the numerical solution reduces to solving a linear system whose coefficient matrix is the $(n(p-k)+k-1)\times(n(p-k)+k-1)$ matrix given by
\[ L_{n,p,k} = M_{n,p,k}^{-1}K_{n,p,k},  \]
where $K_{n,p,k}$, $M_{n,p,k}$ are the symmetric positive definite matrices given by
\begin{align*}
K_{n,p,k}&=\left[\int_0^1B_{j+1,p,k}'(x)B_{i+1,p,k}'(x){\rm d}x\right]_{i,j=1}^{n(p-k)+k-1},\\
M_{n,p,k}&=\left[\int_0^1B_{j+1,p,k}(x)B_{i+1,p,k}(x){\rm d}x\right]_{i,j=1}^{n(p-k)+k-1};
\end{align*}
see \cite[Section~2.5]{Tom-paper} for more details. We remark that, for $p=2$ and $k=1$, the matrix $n^{-1}K_{n,2,1}$ coincides with the matrix $A_n$ of Example~\ref{e5}. 
As proved in \cite[Theorem~6.17]{GLTbookIII}, we have
\begin{align*}
\{n^{-1}K_{n,p,k}\}_n&\sim_\lambda f_{p,k},\\ 
\{nM_{n,p,k}\}_n&\sim_\lambda h_{p,k},\\ 
\{n^{-2}L_{n,p,k}\}_n&\sim_\lambda e_{p,k}=(h_{p,k})^{-1}f_{p,k}, 
\end{align*}
where:
\begin{itemize}[nolistsep,leftmargin=*]
	\item the functions $f_{p,k},h_{p,k}:[0,\pi]\to\mathbb C^{(p-k)\times(p-k)}$ are given by
	\begin{align*}
	f_{p,k}(\theta)=\sum_{\ell\in\mathbb Z}K_{p,k}^{[\ell]}\e^{\i\ell\theta}=K_{p,k}^{[0]}+\sum_{\ell>0}\Bigl(K_{p,k}^{[\ell]}\e^{\i\ell\theta}+(K_{p,k}^{[\ell]})^T\e^{-\i\ell\theta}\Bigr),\\
	h_{p,k}(\theta)=\sum_{\ell\in\mathbb Z}M_{p,k}^{[\ell]}\e^{\i\ell\theta}=M_{p,k}^{[0]}+\sum_{\ell>0}\Bigl(M_{p,k}^{[\ell]}\e^{\i\ell\theta}+(M_{p,k}^{[\ell]})^T\e^{-\i\ell\theta}\Bigr);
	\end{align*}
	\item the blocks $K_{p,k}^{[\ell]}$, $M_{p,k}^{[\ell]}$ are given by
	\begin{alignat*}{3}
	K_{p,k}^{[\ell]}&=\left[\int_{\mathbb R}\beta_{j,p,k}'(t)\beta_{i,p,k}'(t-\ell){\rm d}t\right]_{i,j=1}^{p-k}, &\qquad \ell&\in\mathbb Z,\\
	M_{p,k}^{[\ell]}&=\left[\int_{\mathbb R}\beta_{j,p,k}(t)\beta_{i,p,k}(t-\ell){\rm d}t\right]_{i,j=1}^{p-k}, &\qquad \ell&\in\mathbb Z;
	\end{alignat*}
	\item the functions $\beta_{1,p,k},\ldots,\beta_{p-k,p,k}:\mathbb R\to\mathbb R$ are the first $p-k$ B-splines defined on the knot sequence
	\[ \underbrace{0,\ldots,0}_{p-k},\underbrace{1,\ldots,1}_{p-k},\ldots,\underbrace{\eta,\ldots,\eta}_{p-k},\qquad\eta=\left\lceil\frac{p+1}{p-k}\right\rceil. \]
\end{itemize}
We remark that, for every $\theta\in[0,\pi]$, the matrix $f_{p,k}(\theta)$ is Hermitian positive semidefinite and the matrix $h_{p,k}(\theta)$ is Hermitian positive definite; see \cite[Remark~2.1]{Tom-paper}. 
The following {\sc Maple} worksheet computes $f_{p,k}(\theta)$ and $h_{p,k}(\theta)$ for the input pair $p,k$ defined at the beginning. Here, we have chosen $p=2$ and $k=0$ for comparison with 
Example~\ref{e5}.\\
\rule{\textwidth}{0.5pt}
$\rd>\bk p:=2:\ k:=0:$\\
$\rd>\bk\displaystyle\eta:=\mbox{\rm ceil}\biggl(\frac{p+1}{p-k}\biggr):$\\
$\rd>\bk\mbox{\emph{with}}(\mbox{\emph{CurveFitting}}):\ \mbox{\emph{with}}(\mbox{\emph{LinearAlgebra}}):$\\
$\rd>\bk\mbox{\emph{ReferenceKnotSequence}}:=\left[\mbox{\emph{seq}}(\mbox{\emph{seq}}(j,i=1..p-k),j=0..\eta)\right]:$\\
$\rd>\bk\#\ \mbox{\emph{Construction of the reference B-splines $\beta_{1,p,k},\ldots,\beta_{p-k,p,k}$}}$\\
$\rd>\bk\beta:=[\,\,]:$\\
$\hphantom{\rd>\bk{}}\textbf{for}\ i\ \textbf{from}\ 1\ \textbf{to}\ p-k\ \textbf{do}$\\
$\hphantom{\rd>\bk{}}\beta:=[\mbox{\emph{op}}(\beta),\mbox{\emph{BSpline}}(p+1,t,\mbox{\emph{knots}}=\mbox{\emph{ReferenceKnotSequence}}[i..i+p+1])]:$\\
$\hphantom{\rd>\bk{}}\textbf{od}:$\\
$\rd>\bk\#\ \mbox{\emph{Derivatives of the reference B-splines $\beta_{1,p,k},\ldots,\beta_{p-k,p,k}$}}$\\
$\rd>\bk D\_\beta:=\mbox{\emph{simplify}}(\mbox{\emph{diff}}(\beta,t)):$\\
$\rd>\bk\#\ \mbox{\emph{Construction of the nonzero K\_blocks $K_{p,k}^{[\ell]}$ and the nonzero M\_blocks $M_{p,k}^{[\ell]}$}}$\\
$\rd>\bk \mbox{\emph{Kblocks}}:=[\,\,]:\ \mbox{\emph{Mblocks}}:=[\,\,]:$\\
$\hphantom{\rd>\bk{}}\textbf{for}\ l\ \textbf{from}\ 0\ \textbf{to}\ \eta-1\ \textbf{do}$\\
$\hphantom{\rd>\bk{}}\mbox{\emph{K}}:=\mbox{\emph{Matrix}}(p-k):\ \mbox{\emph{M}}:=\mbox{\emph{Matrix}}(p-k):$\\
$\hphantom{\rd>\bk{}}\textbf{for}\ r\ \textbf{from}\ 1\ \textbf{to}\ p-k\ \textbf{do}:\ \textbf{for}\ s\ \textbf{from}\ 1\ \textbf{to}\ p-k\ \textbf{do}$\\
$\hphantom{\rd>\bk{}}\mbox{\emph{K}}(r,s):=\displaystyle\int_0^\eta D\_\beta[s]\cdot\mbox{\emph{eval}}(D\_\beta[r],t=t-l){\rm d}t:\ \mbox{\emph{M}}(r,s):=\int_0^\eta\beta[s]\cdot\mbox{\emph{eval}}(\beta[r],t=t-l){\rm d}t:$\\
$\hphantom{\rd>\bk{}}\textbf{od}:\ \textbf{od}:$\\
$\hphantom{\rd>\bk{}}\mbox{\emph{Kblocks}}:=[\mbox{\emph{op}}(\mbox{\emph{Kblocks}}),\mbox{\emph{K}}]:\ \mbox{\emph{Mblocks}}:=[\mbox{\emph{op}}(\mbox{\emph{Mblocks}}),\mbox{\emph{M}}]:$\\
$\hphantom{\rd>\bk{}}\textbf{od}:$\\
$\rd>\bk\#\ \mbox{\emph{Construction of the functions $f=f_{p,k}$ and $h=h_{p,k}$}}$\\
$\rd>\bk f(\theta):=\mbox{\emph{Kblocks}}[1]:\ h(\theta):=\mbox{\emph{Mblocks}}[1]:$\\
$\hphantom{\rd>\bk{}}\textbf{for}\ j\ \textbf{from}\ 2\ \textbf{to}\ \eta\ \textbf{do}$\\
$\hphantom{\rd>\bk{}}f(\theta):=\mbox{\emph{simplify}}(f(\theta)+\mbox{\emph{Kblocks}}[j]\cdot{\rm exp}({\rm I}\cdot(j-1)\cdot\theta)+\mbox{\emph{Transpose}}(\mbox{\emph{Kblocks}}[j])\cdot{\rm exp}(-I\cdot(j-1)\cdot\theta)):$\\
$\hphantom{\rd>\bk{}}h(\theta):=\mbox{\emph{simplify}}(h(\theta)+\mbox{\emph{Mblocks}}[j]\cdot{\rm exp}({\rm I}\cdot(j-1)\cdot\theta)+\mbox{\emph{Transpose}}(\mbox{\emph{Mblocks}}[j])\cdot{\rm exp}(-I\cdot(j-1)\cdot\theta)):$\\
$\hphantom{\rd>\bk{}}\textbf{od}:\ f(\theta),\ h(\theta)$
\[ \bl \left[\begin{matrix}\dfrac43 & -\dfrac23-\dfrac{2\,{\rm e}^{{\rm I}\theta}}3\\[10pt] -\dfrac23-\dfrac{2\,{\rm e}^{-{\rm I}\theta}}3 & \dfrac83-\dfrac{4\cos(\theta)}3\end{matrix}\right],\ \left[\begin{matrix}\dfrac2{15} & \dfrac1{10}+\dfrac{{\rm e}^{{\rm I}\theta}}{10}\\[10pt] \dfrac1{10}+\dfrac{{\rm e}^{-{\rm I}\theta}}{10} & \dfrac25+\dfrac{\cos(\theta)}{15}\end{matrix}\right] \bk \]
\rule{\textwidth}{0.5pt}
In Theorems~\ref{t_appK}--\ref{t_appL}, we report the formulas for the eigenvalues of the matrices $n^{-1}K_{n,p,k}$, $nM_{n,p,k}$, $n^{-2}L_{n,p,k}$ that we obtained through high-precision numerical computations performed 
in \textsc{Julia}, using an accuracy of at least $100$ decimal digits.
In what follows, for every $\theta\in[0,\pi]$ and every $j=1,\ldots,p-k$, we define $\lambda_j(f_{p,k}(\theta))$ (resp., $\lambda_j(h_{p,k}(\theta))$, $\lambda_j(e_{p,k}(\theta))$) as the $j$th eigenvalue of $f_{p,k}(\theta)$ (resp., $h_{p,k}(\theta)$, $e_{p,k}(\theta)$) according to the increasing ordering:
\begin{alignat*}{3}
&\lambda_1(f_{p,k}(\theta))\le\ldots\le\lambda_{p-k}(f_{p,k}(\theta)), &&\qquad\theta\in[0,\pi],\\
&\lambda_1(h_{p,k}(\theta))\le\ldots\le\lambda_{p-k}(h_{p,k}(\theta)), &&\qquad\theta\in[0,\pi],\\
&\lambda_1(e_{p,k}(\theta))\le\ldots\le\lambda_{p-k}(e_{p,k}(\theta)), &&\qquad\theta\in[0,\pi].
\end{alignat*}
Moreover, for every $n\ge1$, we define the uniform grids 
\begin{equation*}
\Theta_n=\biggl\{\frac{i\pi}n:i=0,\ldots,n\biggr\},\qquad\Theta_n^0=\Theta_n\setminus\{0\},\qquad\Theta_n^\pi=\Theta_n\setminus\{\pi\},\qquad\Theta_n^{0,\pi}=\Theta_n\setminus\{0,\pi\}.
\end{equation*}

\begin{theorem}\label{t_appK}
Let $1\le p\le100$ and $0\le k\le\min(1,p-1)$. Then, for every $n=1,\ldots,100$, the eigenvalues of $n^{-1}K_{n,p,k}$ are given by
\begin{equation*}\label{lambdas_K}
\lambda_j(f_{p,k}(\theta)),\qquad\theta\in\Theta_{n,p,k}^{(j)},\qquad j=1,\ldots,p-k,
\end{equation*}
where the grid $\Theta_{n,p,k}^{(j)}$ belongs to $\{\Theta_n,\Theta_n^0,\Theta_n^\pi,\Theta_n^{0,\pi}\}$  
and is given in Figure~{\rm\ref{K0K1_triangles}}. 
\end{theorem}

\begin{figure}
\centering
\includegraphics[width=0.435\textheight]{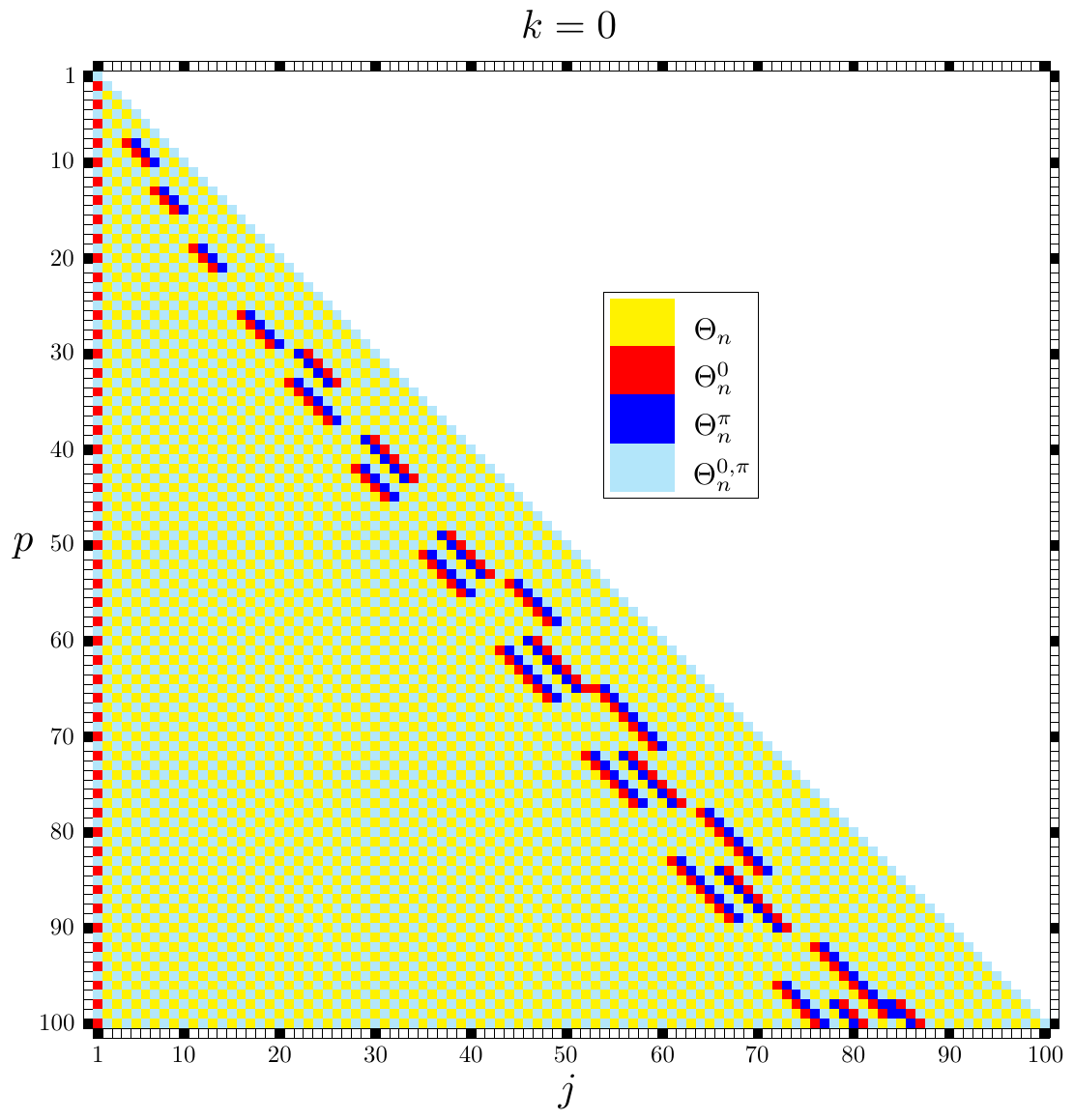}\\[10pt]
\includegraphics[width=0.435\textheight]{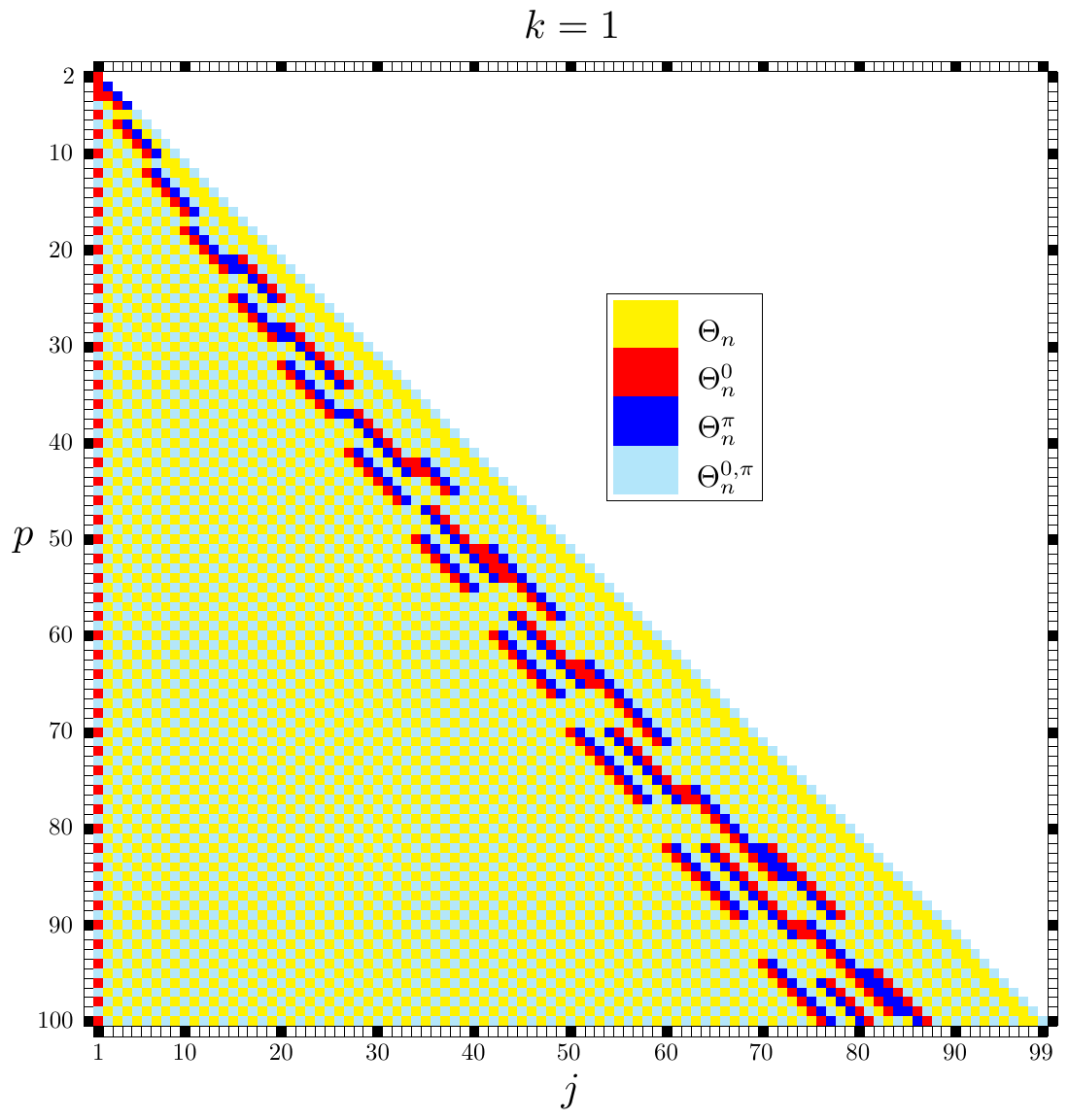}
\caption{Grid $\Theta_{n,p,k}^{(j)}$ for the values of $n,p,k$ considered in Theorem~\ref{t_appK} and for $j=1,\ldots,p-k$. For example, for $p=2$, $k=0$, and every $n=1,\ldots,100$, we have $\Theta_{n,2,0}^{(1)}=\Theta_n^0$ and $\Theta_{n,2,0}^{(2)}=\Theta_n^{0,\pi}$, in accordance with Example~\ref{e5}.}
\label{K0K1_triangles}
\end{figure}

\begin{conjecture}\label{n>=1}
Theorem~\ref{t_appK} continues to hold if we replace ``for every $n=1,\ldots,100$'' with ``for every $n\ge1$''.
\end{conjecture}

To simplify the statement of Theorem~\ref{t_appM}, we define the integer sequence
\[ a(m)=m+\bigl\lfloor\,\sqrt{8m}\,\bigr\rfloor,\qquad m\ge1. \]
The sequence $\{a_m\}_{m=1,2,\ldots}$ is referred to as the A186348 sequence in the on-line encyclopedia of integer sequences (OEIS); see \url{https://oeis.org/A186348}.
For every $p\ge3$, we define
\[ \alpha_p=\mbox{minimum positive integer such that }p\in\left[\,\sum_{m=1}^{\alpha_p}a(m),\sum_{m=1}^{\alpha_p+1}a(m)\right]. \]
It is not difficult to check that $\{p-\alpha_p\}_{p=3,4,\ldots}$ is an increasing sequence such that $p-\alpha_p\ge2$ for all $p\ge3$.

\begin{theorem}\label{t_appM}
Let $1\le p\le100$ and $0\le k\le\min(1,p-1)$. Then, for every $n=1,\ldots,100$, the eigenvalues of $nM_{n,p,k}$ are given by
\begin{equation*}\label{lambdas_M}
\lambda_j(h_{p,k}(\theta)),\qquad\theta\in\Theta_{n,p,k}^{[j]},\qquad j=1,\ldots,p-k,
\end{equation*}
where the grid $\Theta_{n,p,k}^{[j]}$ belongs to $\{\Theta_n,\Theta_n^0,\Theta_n^\pi,\Theta_n^{0,\pi}\}$ and is defined as follows:
\begin{align*}
\Theta_{n,p,0}^{[j]}&=\left\{\begin{aligned}
&\Theta_n^0, &&\mbox{if $p+j$ is odd,}\\
&\Theta_n^\pi, &&\mbox{if $p+j$ is even and $j\ne p$,}\\
&\Theta_n^{0,\pi}, &&\mbox{if $j=p$,}\\
\end{aligned}\right.\\
\Theta_{n,p,1}^{[j]}&=\left\{\begin{aligned}
&\Theta_n, &&\mbox{if $p>2$ and $j=p-\alpha_p-1$,}\\
&\Theta_n^0, &&\mbox{if $p=2$; or}\\
&&&\mbox{if $p>2$ and either $j<p-\alpha_p-1$ and $p+j$ is odd or $p-\alpha_p-1<j<p-1$ and $p+j$ is even,}\\
&\Theta_n^\pi, &&\mbox{if $p>2$ and either $j<p-\alpha_p-1$ and $p+j$ is even or $p-\alpha_p-1<j<p-1$ and $p+j$ is odd,}\\
&\Theta_n^{0,\pi}, &&\mbox{if $p>2$ and $j=p-1$.} 
\end{aligned}\right.
\end{align*}
\end{theorem}

\begin{conjecture}\label{p>100M}
Theorem~\ref{t_appM} continues to hold if we replace ``$1\le p\le100$'' with ``$p\ge1$'' and ``for every $n=1,\ldots,100$'' with ``for every $n\ge1$''.
\end{conjecture}

\begin{theorem}\label{t_appL}
Let $1\le p\le100$ and $0\le k\le\min(1,p-1)$. Then, for every $n=1,\ldots,20$, the eigenvalues of $n^{-2}L_{n,p,k}$ are given by
\begin{equation*}\label{lambdas_L}
\lambda_j(e_{p,k}(\theta)),\qquad\theta\in\Theta_{n,p,k}^{\langle j\rangle},\qquad j=1,\ldots,p-k,
\end{equation*}
where the grid $\Theta_{n,p,k}^{\langle j\rangle}$ belongs to $\{\Theta_n,\Theta_n^0,\Theta_n^\pi,\Theta_n^{0,\pi}\}$ and is defined as follows: 
\[ \Theta_{n,p,k}^{\langle j\rangle}=\left\{\begin{aligned}
&\Theta_n, &&\mbox{if $p+j$ is odd and $j>1$,}\\
&\Theta_n^0, &&\mbox{if $p+j$ is odd and $j=1$,}\\
&\Theta_n^{0,\pi}, &&\mbox{if $p+j$ is even.}
\end{aligned}\right. \]
\end{theorem}

\begin{conjecture}\label{p>100L}
Theorem~\ref{t_appL} continues to hold if we replace ``$1\le p\le100$'' with ``$p\ge1$'' and ``for every $n=1,\ldots,20$'' with ``for every $n\ge1$''.
\end{conjecture}

\section*{Acknowledgements}
Giovanni Barbarino, Carlo Garoni and Stefano Serra-Capizzano are members of the Research Group GNCS (Gruppo Nazionale per il Calcolo Scientifico) of INdAM (Istituto Nazionale di Alta Matematica).
Giovanni Barbarino was supported by the Alfred Kordelinin S\"a\"ati\"o Grant No~210122.
Carlo Garoni was supported by the Department of Mathematics of the University of Rome Tor Vergata through the MUR Excellence Department Project MatMod@TOV (CUP E83C23000330006) and the Project RICH\hspace{1pt}\rule{5pt}{0.4pt}GLT (CUP E83C22001650005).
David Meadon was funded by the Centre for Interdisciplinary Mathematics (CIM) at Uppsala University.
Stefano Serra-Capizzano was funded by the European High-Performance Computing Joint Undertaking (JU) under grant agreement No~955701. The JU receives support from the European Union's Horizon 2020 research and innovation programme and Belgium, France, Germany, Switzerland.
Stefano Serra-Capizzano is also grateful to the Theory, Economics and Systems Laboratory (TESLAB) of the Department of Computer Science at the Athens University of Economics and Business for providing financial support.

\end{document}